\documentclass{aart}

\usepackage{titlesec}
\titleformat{\subsection}[runin]{\normalfont\bfseries}{\thesubsection.}{.5em}{}[.]\titlespacing{\subsection}{0pt}{2ex plus .1ex minus .2ex}{.8em}
\titleformat{\subsubsection}[runin]{\normalfont\itshape}{\thesubsubsection.}{.3em}{}[.]\titlespacing{\subsubsection}{0pt}{1ex plus .1ex minus .2ex}{.5em}

\usepackage[labelfont=sc,font=small,labelsep=period]{caption}
\usepackage[letterpaper, hmargin=1in, top=1.5in, bottom=1.9in, footskip=0.7in]{geometry}

\usepackage{amsmath} %[intlimits]
\usepackage{amssymb}
\usepackage{amsfonts}
\usepackage{latexsym}
\usepackage{amsthm}
\usepackage{amsxtra}
\usepackage{amscd}
\usepackage{bbm}
\usepackage{mathrsfs}
\usepackage{bm}

\usepackage{graphicx, color}
\usepackage[nottoc,notlof,notlot]{tocbibind} 
\usepackage{cite} % Enabling `[1-4]` - style citing
\usepackage{xspace} %command to insert a space after a latex command

\numberwithin{equation}{section}
\numberwithin{figure}{section}
\theoremstyle{plain} %plain, definition, remark
\newtheorem{theorem}{Theorem}[section]
\newtheorem*{theorem*}{Theorem}
\newtheorem{lemma}[theorem]{Lemma}
\newtheorem*{lemma*}{Lemma}
\newtheorem{corollary}[theorem]{Corollary}
\newtheorem*{corollary*}{Corollary}
\newtheorem{proposition}[theorem]{Proposition}
\newtheorem*{proposition*}{Proposition}
\newtheorem{definition}[theorem]{Definition}
\newtheorem*{definition*}{Definition}

\theoremstyle{definition} %plain, definition, remark

\newtheorem*{example*}{Example}
\newtheorem{remark}[theorem]{Remark}

\newtheorem*{remark*}{Remark}
\newtheorem*{remarks*}{Remarks}

\newcommand{\f}[1]{\boldsymbol{\mathrm{#1}}} %bold
\newcommand{\bb}{\mathbb} %blackboard bold
\renewcommand{\cal}{\mathcal} 
 
\newcommand{\fra}{\mathfrak} 
\newcommand{\ul}[1]{\underline{#1} \!\,} %underline
\newcommand{\ol}[1]{\overline{#1} \!\,} %overline
\newcommand{\wh}{\widehat}
\newcommand{\wt}{\widetilde}

\definecolor{green}{rgb}{0.3,0.7,0.3}

\renewcommand{\P}{\mathbb{P}}
\newcommand{\E}{\mathbb{E}}
\newcommand{\R}{\mathbb{R}}
\newcommand{\C}{\mathbb{C}}
\newcommand{\N}{\mathbb{N}}
\newcommand{\Z}{\mathbb{Z}}

\newcommand{\ee}{\mathrm{e}}
\newcommand{\ii}{\mathrm{i}}
\newcommand{\dd}{\mathrm{d}}
\newcommand{\col}{\mathrel{\vcenter{\baselineskip0.75ex \lineskiplimit0pt \hbox{.}\hbox{.}}}}
\newcommand*{\deq}{\mathrel{\vcenter{\baselineskip0.65ex \lineskiplimit0pt \hbox{.}\hbox{.}}}=}

\renewcommand{\leq}{\leqslant}
\renewcommand{\geq}{\geqslant}
\renewcommand{\epsilon}{\varepsilon}

\newcommand{\ind}[1]{\f 1 (#1)}
\newcommand{\indb}[1]{\f 1 \pb{#1}}
\newcommand{\indB}[1]{\f 1 \pB{#1}}

\newcommand{\indBB}[1]{\f 1 \pBB{#1}}

\newcommand{\p}[1]{({#1})}
\newcommand{\pb}[1]{\bigl({#1}\bigr)}
\newcommand{\pB}[1]{\Bigl({#1}\Bigr)}
\newcommand{\pbb}[1]{\biggl({#1}\biggr)}
\newcommand{\pBB}[1]{\Biggl({#1}\Biggr)}

\newcommand{\qb}[1]{\bigl[{#1}\bigr]}

\newcommand{\qbb}[1]{\biggl[{#1}\biggr]}
\newcommand{\qBB}[1]{\Biggl[{#1}\Biggr]}

\newcommand{\h}[1]{\{{#1}\}}
\newcommand{\hb}[1]{\bigl\{{#1}\bigr\}}
\newcommand{\hB}[1]{\Bigl\{{#1}\Bigr\}}
\newcommand{\hbb}[1]{\biggl\{{#1}\biggr\}}

\newcommand{\abs}[1]{\lvert #1 \rvert}
\newcommand{\absb}[1]{\bigl\lvert #1 \bigr\rvert}
\newcommand{\absB}[1]{\Bigl\lvert #1 \Bigr\rvert}
\newcommand{\absbb}[1]{\biggl\lvert #1 \biggr\rvert}
\newcommand{\absBB}[1]{\Biggl\lvert #1 \Biggr\rvert}

\newcommand{\norm}[1]{\lVert #1 \rVert}

\newcommand{\normbb}[1]{\biggl\lVert #1 \biggr\rVert}

\newcommand{\avg}[1]{\langle #1 \rangle}
\newcommand{\avgb}[1]{\bigl\langle #1 \bigr\rangle}
\newcommand{\avgB}[1]{\Bigl\langle #1 \Bigr\rangle}
\newcommand{\avgbb}[1]{\biggl\langle #1 \biggr\rangle}

\DeclareMathOperator{\tr}{Tr}
\DeclareMathOperator{\var}{Var}

\DeclareMathOperator{\supp}{supp}

\DeclareMathOperator{\re}{Re}
\DeclareMathOperator{\im}{Im}

\DeclareMathOperator{\dist}{dist}

\begin{document}
\author{L\'aszl\'o Erd\H{o}s\footnote{IST Austria, Am Campus 1, Klosterneuburg A-3400, lerdos@ist.ac.at. On leave from Institute of Mathematics,
University of Munich. Partially supported by SFB-TR 12 Grant of the German Research Council  and by ERC 
Advanced Grant, RANMAT 338804.}
\and Antti Knowles\footnote{ETH Z\"urich, knowles@math.ethz.ch. Partially supported by Swiss National Science Foundation grant 144662.}}
\title{The Altshuler-Shklovskii Formulas for Random Band Matrices II: the General Case}
\maketitle

\begin{abstract}
The Altshuler-Shklovskii formulas \cite{AS} predict, for any disordered quantum system in the 
 diffusive regime, a universal power law behaviour for
the correlation functions of the mesoscopic eigenvalue density. 
 In this paper and its companion \cite{EK3}, 
 we prove these formulas for random band matrices.
In \cite{EK3} we introduced a diagrammatic approach and presented robust  estimates on general diagrams under certain
simplifying assumptions. In this paper we remove these assumptions
 by giving a general estimate of the subleading diagrams. 
We also give a precise analysis  of the leading diagrams which give rise to the Altschuler-Shklovskii power laws.
Moreover, we introduce a family of general random band matrices which
interpolates  between real symmetric $(\beta=1)$ and complex Hermitian
$(\beta=2)$ models, and track the transition for the mesoscopic density-density
correlation.
Finally, we address the higher-order correlation functions by proving that they behave asymptotically according to a Gaussian process whose covariance is given by the Altshuler-Shklovskii formulas. 
\end{abstract}

\section{Introduction} \label{sec:intro}

A fundamental  observation  from physics is 
 that the spectral statistics of disordered quantum systems exhibit universal patterns.
In the delocalized regime and on the microscopic  energy 
 scale of individual eigenvalues, the correlation functions
exhibit the celebrated Wigner-Dyson-Mehta  statistics, which depend  only  on the basic symmetry
class of the model. Above a certain critical energy scale $\eta_c$, called the 
\emph{Thouless energy}, the correlations of the spectral density exhibit a different type of statistics, 
 first predicted by Altshuler and Shklovskii in \cite{AS}. 
This behaviour is only present in systems possessing a nontrivial
 spatial structure that gives rise to  quantum diffusion.
The first Altshuler-Shklovskii formula states that the variance of
the number of eigenvalues ${\cal N}_\eta (E)$ in a spectral  window  of size $\eta\gg \eta_c$
about an energy $E$ behaves according to
\begin{equation}\label{ASvar2}
\mbox{Var} \; {\cal N}_\eta (E) \;\sim\; (\eta/\eta_c)^{d/2} \qquad  (d=1,2,3)\,.
\end{equation}
The second  Altshuler-Shklovskii formula states that the correlation function in the regime $E_2-E_1 \gg \eta\gg \eta_c$ behaves according to
\begin{equation}\label{AScorr2}
\avgb{{\cal N}_\eta (E_1)\,; {\cal N}_\eta (E_2)}  \;\sim\; (E_2-E_1)^{-2+d/2} \qquad  (d=1,2,3)\,.
\end{equation}
In this work and in its companion paper \cite{EK3}, we prove the Altshuler-Shklovskii formulas for  a specific type of  disordered quantum systems: random band matrices with independent entries. Random band matrices interpolate between the mean-field  Wigner matrices and random Schr\"odinger operators \cite{Spe, FM}. They have a sufficiently rich spatial structure to be diffusive (for a proof see \cite{EK1, EK2}),
and are more amenable to rigorous analysis than random Schr\"odinger operators. The detailed physical background of the problem and related mathematical works are presented in Section 1 of the companion paper \cite{EK3}, and will not be repeated here. Here we only explain how this paper is related to its companion \cite{EK3}.

The main tool in both papers is a diagrammatic expansion technique. The
correlation functions are expressed as a sum of many terms, which can be
conveniently represented using graphs.  The resulting expansion is highly oscillatory: 
the sum of the absolute values diverges rapidly, although the sum itself remains bounded. 
To handle  the oscillations, we apply two different resummation procedures before the 
resummed diagrams can be estimated individually in absolute value. 
The first resummation is performed using an expansion in Chebyshev polynomials, and is motivated by 
the work \cite{FS}. In the jargon of diagrammatic perturbation theory, this resummation 
step corresponds to the self-energy (or tadpole) renormalization. A similar resummation was used in 
\cite{EK1,EK2} to analyse the quantum diffusion of the unitary propagator. 
The quantity studied in the current paper -- the
local density-density correlation -- is considerably more difficult to analyse
because it arises from higher-order terms than the quantum diffusion. Hence, not only does the
leading term have to be analysed more precisely, but the error estimates also require a more careful analysis. Most importantly, even for the error terms
we need a second resummation, which bundles specific families of diagrams (so-called ladder graphs) that are strongly oscillatory. Hence, apart from a few basic algebraic tools on nonbacktracking powers, the argument of the current paper is entirely different from the one in \cite{EK1, EK2}.

In \cite{EK3} we introduced the necessary diagrammatic representation; 
for the convenience of the reader, we give a short summary of it in Section~\ref{sec_41}.  
With the diagrams at hand, the proof can be divided into two parts:
(i) estimating all subleading diagrams, and (ii) analysing the asymptotics of the leading diagrams.

Most diagrams give rise to error terms, in the sense
that they do not contribute in leading order. Estimating them, using the two
 resummation procedures sketched above,
constitutes part (i) -- the first, and most challenging,
 part of the proof. 
This estimate is presented in \cite{EK3} under some simplifying assumptions in order
 to highlight the main ideas. The argument of \cite{EK3} also singles out
 eight distinguished diagrams
(after all of the resummations have been performed) that contribute in leading order.
 In Section \ref{sec: part 2}, we give the complete estimates in part (i) by removing the simplifying assumptions made in \cite{EK3}. In addition, the asymptotic analysis of the leading diagrams (called part (ii) above) takes up most of Section \ref{sec:3} of this paper. To ease readability, we strive to keep this paper self-contained; in particular, when needed we review the setup
and key notations introduced in \cite{EK3}. We refer back to \cite{EK3} only
for a few explicit results, whose content we explain here.

Our main theorems on the Altshuler-Shklovskii formulas are stated in Section \ref{sec:unimodular_results}. In Sections \ref{sec:clt_result} and \ref{sec:gen1} we give two generalizations of our results, which are proved in Section \ref{sec:generalizations}.

In our first generalization, given in Section \ref{sec:clt_result}, we prove that the finite-dimensional marginals of mesoscopic eigenvalue densities agree asymptotically with those of a Gaussian process whose covariance is given by the Altshuler-Shklovskii formulas. Thus, high-order correlation functions factorize into two-point correlation functions. This may be interpreted as a central limit theorem for the mesoscopic eigenvalue densities.

In our second generalization, given in Section \ref{sec:gen1}, we consider a family of general random band matrices whose entries $H_{xy}$ have arbitrary translation-invariant variances, i.e.\ $\E \abs{H_{xy}}^2 = W^{-d} f((x-y)/W)$ for some profile function $f$. We find that density-density correlation depends on the matrix entries only through $f$.
Moreover, for $d = 1,2$ the leading and subleading terms of the density-density correlation are universal, for $d = 3,4$ only the leading terms are universal, and for $d \geq 5$ the density-density correlation is not universal. The leading terms depend only on the second moments of $f$, while the subleading terms depend in addition on the fourth moments of $f$. In addition, the general model from Section \ref{sec:gen1} interpolates between real symmetric $(\beta=1)$ and complex Hermitian $(\beta=2)$ models, and hence allows us to track the transition for the mesoscopic density-density correlation.

The basic algebraic identity used in the first resummation (see \eqref{H^n and U_n} below)
is much simpler if the matrix entries are constant in absolute
value (unimodular case); the proof in \cite{EK3} 
and in Section~\ref{sec: part 2} are presented in this case. 
The general band matrix model from Section \ref{sec:gen1} generates additional terms in the basic identity (compare \eqref{H^n and U_n} with \eqref{H^n and U_n general}); they have an ultimately negligible contribution, but nevertheless give rise to structurally new diagrams. 
Their treatment substantially complicates the analysis.
In the proof of the quantum diffusion, these complications were carefully treated in \cite{EK2}, where we extended the analysis of \cite{EK1}
from the unimodular case to the general case. In Section~\ref{sec:general_proof} we sketch how to extend the current analysis of the correlation functions from unimodular case to the general one. 

Moreover, in Section~\ref{sec: heavy tail} we explain how our analysis can also be applied to a special one-dimensional  band matrix model which exhibits
 \emph{critical behaviour},  in the sense that is supposed to lie at the metal-insulator transition point. 
We compute the so-called  \emph{compressibility} of the 
mesoscopic eigenvalue statistics  for this critical model, 
 and find that it  coincides with the prediction from the physics literature.

Finally, in Section \ref{sec:mean-field} we explain how to extend our results to include the \emph{mean-field regime} $\eta \ll \eta_c$ in addition to the \emph{diffusive regime} $\eta \gg \eta_c$.

\subsubsection*{Conventions}
We use $C$ to denote a generic large positive constant, which may depend on some fixed parameters and whose value may change from one expression to the next. Similarly, we use $c$ to denote a generic small positive constant. We use $a \asymp b$ to mean $c a \leq b \leq C a$. 
Also, for any finite set $A$ we use $\abs{A}$ to denote the cardinality of $A$.

\subsubsection*{Acknowledgements} We are very grateful to Alexander Altland and Yan Fyodorov  for
detailed discussions on the physics of the problem and for providing references.

\section{Setup and Results} \label{sec: setup}

\subsection{Definitions and assumptions} \label{sec:setup1}
Fix $d \in \N$, the physical dimension of
the configuration space. For $L \in \N$ we define the discrete torus of size $L$
\begin{equation*}
\bb T \;\equiv\; \bb T_L^d \;\deq\; \pb{[-L/2, L/2) \cap \Z}^d\,,
\end{equation*}
and abbreviate
\begin{equation}\label{Ndef}
N \;\deq\; \abs{\bb T_L} \;=\; L^d.
\end{equation}
Let $1 \ll W \leq  L$ denote the band width, and define the deterministic matrix $S = (S_{xy})$ through
\begin{equation} \label{step S}
S_{xy} \;\deq\; \frac{\ind{1 \leq \abs{x - y} \leq W}}{M - 1}\,, \qquad M \;\deq\; \sum_{x \in \bb T} \ind{1 \leq \abs{x} \leq W}\,,
\end{equation}
where $\abs{\cdot}$ denotes the periodic Euclidean norm on $\bb T$, i.e.\ $\abs{x} \deq \min_{\nu \in \Z^d} \abs{x + L \nu}_{\Z^d}$. Note that
\begin{equation} \label{link between M and W}
M \;\asymp\; W^d\,.
\end{equation}
The fundamental parameters of our model are the linear dimension of the torus, $L$, and the band width, $W$. The quantities $N$ and $M$ are introduced for notational convenience, since most of our estimates depend naturally on $N$ and $M$ rather than $L$ and $W$. We regard $L$ as the independent parameter, and $W \equiv W_L$ as a function of $L$.

Next, let $A = A^* = (A_{xy})$ be a Hermitian random matrix whose upper-triangular\footnote{We introduce an arbitrary and immaterial total ordering $\leq$ on the torus $\bb T$.} entries $(A_{xy} \col x \leq y)$ are independent random variables with zero expectation. We consider two cases.
\begin{itemize}
\item 
\emph{The real symmetric  case ($\beta = 1$)}, where $A_{xy}$ satisfies $\P(A_{xy} = 1) = \P(A_{xy} = -1) = 1/2$.
\item
\emph{The complex Hermitian case ($\beta = 2$)}, where $A_{xy}$ is uniformly distributed on the unit circle $\bb S^1 \subset \C$.
\end{itemize}
Here the index $\beta = 1,2$ is the customary symmetry index of random matrix theory.

We define the \emph{random band matrix} $H = (H_{xy})$ through
\begin{equation}\label{HA}
H_{xy} \;\deq\; \sqrt{S_{xy}} \, A_{xy}\,.
\end{equation}
Note that $H$ is Hermitian and $\abs{H_{xy}}^2 = S_{xy}$, i.e.\ $\abs{H_{xy}}$ is deterministic. Moreover, we have for all $x$
\begin{equation} \label{normalization of S}
\sum_{y} S_{xy} \;=\; \frac{M}{M - 1}\,.
\end{equation}
With this normalization, as $N, W\to \infty$ the bulk of the spectrum of  $H/2$  lies in  $[-1,1]$ 
and the eigenvalue density is given by the Wigner semicircle law with density
\begin{equation}
\nu(E) \;\deq\; \frac{2}{\pi}\sqrt{1-E^2} \qquad \text{for} \quad E \in [-1,1]\,.
\end{equation}

Let $\phi$ be a smooth, integrable, real-valued function on $\R$ satisfying $\int \phi(E) \, \dd E \neq 0$. 
We call such functions $\phi$ \emph{test functions}. We also require that our test functions $\phi$ satisfy one of the two following conditions.
\begin{itemize}
\item[\textbf{(C1)}]
$\phi$ is the Cauchy kernel
\begin{equation} \label{Cauchy}
\phi(E) \;=\; \im \frac{2}{E - \ii} \;=\; \frac{2}{E^2 + 1}\,.
\end{equation}
\item[\textbf{(C2)}]
For every $q > 0$ there exists a constant $C_q$ such that
\begin{equation} \label{non_Cauchy}
\abs{\phi(E)} \;\leq\; \frac{C_q}{1 + \abs{E}^q}\,.
\end{equation}
\end{itemize}

A typical example of a test function $\phi$ satisfying \textbf{(C2)} is the Gaussian $\phi(E) = \sqrt{2 \pi} \, \ee^{-E^2 / 2}$. We introduce the rescaled test function $\phi^\eta(E) \deq \eta^{-1} \phi(\eta^{-1} E)$. We shall be interested in correlations of observables depending on $E \in (-1,1)$ of the form
\begin{equation} \label{def_Y}
Y^\eta_\phi(E) \;\deq\; \frac{1}{N} \sum_i \phi^\eta(\lambda_i - E) \;=\; \frac{1}{N} \tr \phi^\eta (H/2 - E)\,,
\end{equation}
where $\lambda_1, \dots, \lambda_N$ denote the eigenvalues of $H/2$. 
(The factor $1/2$ is a mere convenience, chosen because as noted above  the asymptotic spectrum of $H/2$ is the interval $[-1,1]$.)
The quantity $Y^\eta_\phi(E)$ 
is the smoothed local density of states around the energy $E$ on the scale $\eta$.
We always choose
\begin{equation*}
\eta \;=\; M^{-\rho}
\end{equation*}
for some fixed $\rho \in (0,1/3)$, and we frequently drop the index $\eta$ from our notation. The strongest
results are for large $\rho$, so that one should think of $\rho \approx 1/3$. 
The restriction $\rho < 1/3$ is technical; see Remark \ref{rem:rho_assump} below for more details.

We are interested
in the correlation function of the local densities of states, $Y^\eta_{\phi_1}(E_1)$ and  $Y^\eta_{\phi_2}(E_2)$, around two energies
$E_1 \leq E_2$. We shall investigate two regimes: $\eta \ll E_2 - E_1$ and $E_1 = E_2$. In the former regime, 
we prove that the correlation decay in the energy difference $E_2-E_1$ is universal 
(in particular, independent of $\eta$, $\phi_1$, and $\phi_2$), and we compute the correlation function explicitly. In the latter regime, we prove that the variance has a universal dependence on $\eta$, and depends on $\phi_1$ and $\phi_2$ via their inner product in a homogeneous Sobolev space.

The case \textbf{(C2)} for the test functions  is the more interesting one,
since it corresponds to local densities on a definite scale. The heavy tail of
the  Cauchy kernel \textbf{(C1)} introduces unwanted correlations from the overlap of the test functions.
Nevertheless, we give our results for the specific case \textbf{(C1)} as well since 
it corresponds to the imaginary part of the resolvent, a quantity often considered
in the physics literature. Moreover, the case \textbf{(C1)} is pedagogically
useful, since in that case the computation of the main term is considerably simpler.

\begin{definition}
Throughout the following we use the quantities $E_1, E_2 \in (-1,1)$ and
\begin{equation*}
E \;\deq\; \frac{E_1 + E_2}{2}\,, \qquad \omega \;\deq\; E_2 - E_1
\end{equation*}
interchangeably. Without loss of generality we always assume that $\omega \geq 0$.
\end{definition}

For the following we choose and fix a positive constant $\kappa$. We always assume that
\begin{equation} \label{D leq kappa}
E_1, E_2 \;\in\; [-1 + \kappa, 1 - \kappa] \,, \qquad \omega \;\leq\; c_*
\end{equation}
for some small enough positive constant $c_*$ depending on $\kappa$. These restrictions are required since the nature of the correlations changes near the spectral edges $\pm 1$. Throughout the following we regard the constants $\kappa$ and $c_*$ as fixed and do not track the dependence of our estimates on them.

\subsection{Unimodular band matrices} \label{sec:unimodular_results}
Our first theorem gives the leading behaviour of the density-density correlation function in terms of a function $\Theta^\eta_{\phi_1, \phi_2}(E_1, E_2)$, which is explicit but has a complicated form.
In the two subsequent theorems we determine the asymptotics of this function in two physically relevant regimes, where its form simplifies substantially. We remark that Theorems \ref{thm: main result}--\ref{thm: Theta 1} are the same as Theorems 2.2--2.4 in \cite{EK3}. We use the abbreviations
\begin{equation} \label{def_avg}
\avg{X} \;\deq\; \E X\,, \qquad \avg{X \mspace{1mu} ; Y} \;\deq\; \E (XY) - \E X \, \E Y\,.
\end{equation}

\begin{theorem}[Density-density correlations] \label{thm: main result}
Fix $\rho \in (0, 1/3)$ and $d \in \N$, and set $\eta \deq M^{-\rho}$. Suppose that the test functions $\phi_1$ and $\phi_2$ satisfy either both \textbf{(C1)} or both \textbf{(C2)}.
Suppose moreover that
\begin{equation} \label{LW_assump}
W^{1 + d/6} \;\leq\; L \;\leq\; W^C
\end{equation}
for some constant $C$.

Then there exist a constant $c_0 > 0$ and a function $\Theta_{\phi_1,\phi_2}^\eta(E_1, E_2)$ -- which is given explicitly in \eqref{def:theta} and  \eqref{VDdef} 
 below, and whose asymptotic behaviour is derived in Theorems \ref{thm: Theta 2} and \ref{thm: Theta 1} below -- such that, for any $E_1, E_2$ satisfying \eqref{D leq kappa} for small enough $c_* > 0$,  the local density-density correlation satisfies
\begin{equation} \label{EY_result}
\frac{\avg{Y^\eta_{\phi_1}(E_1) \, ; Y^\eta_{\phi_2}(E_2)}}{\avg{Y^\eta_{\phi_1}(E_1)} \avg{Y^\eta_{\phi_2}(E_2)}} \;=\; \frac{1}{(LW)^d} \pB{\Theta_{\phi_1,\phi_2}^\eta(E_1,E_2) + O \pb{M^{-c_0} R_2(\omega + \eta)}} \,,
\end{equation}
where we defined
\begin{equation} \label{def_R2}
R_2(s) \;\deq\; 1 + \ind{d = 1} s^{-1/2} + \ind{d = 2} \abs{\log s}\,.
\end{equation}

Moreover, if $\phi_1$ and $\phi_2$ are analytic in a strip containing the real axis (e.g.\ as in the case \textbf{(C1)}), we may replace the upper bound $L \leq W^C$ in \eqref{LW_assump} $L \leq \exp(W^c)$ for some small constant $c > 0$.
\end{theorem}

We shall prove that the error term in \eqref{EY_result} is smaller than the main term $\Theta$ for all $d \geq 1$. The main term $\Theta$ has a simple, and universal, explicit form only for $d \leq 4$. The two following theorems give the leading behaviour of the function $\Theta$ for $d \leq 4$ in the two regimes $\omega = 0$ and $\omega \gg \eta$. In fact, one may also compute the subleading corrections to $\Theta$. These corrections turn out to be universal for $d \leq 2$ but not for $d \geq 3$; see Theorem \ref{thm: Theta 1} and the remarks following it.

In order to describe the leading behaviour of the variance, i.e.\ the case $\omega = 0$, we introduce the Fourier transform
\begin{equation*}
\phi(E) \;=\; \int_{\R} \dd t \, \ee^{-\ii E t} \, \wh \phi(t)\,, \qquad \wh \phi(t) \;=\; \frac{1}{2 \pi} \int_{\R}
 \dd E \, \ee^{\ii E t} \, \phi(E) \,.
\end{equation*}
For $d \leq 4$ we define the quadratic form $V_d$ through
\begin{equation} \label{def_V_d}
V_d(\phi_1, \phi_2) \;\deq\; \int_\R \dd t \, \abs{t}^{1 - d/2} \, \ol {\wh \phi_1(t)} \, \wh \phi_2(t) \qquad (d \leq 3) \,, \qquad \qquad V_4(\phi_1,\phi_2) \;\deq\; 2 \ol {\wh \phi_1(0)} \, \wh \phi_2(0)\,.
\end{equation}
Note that $V_d(\phi_1, \phi_2)$ is real since both $\phi_1$ and $\phi_2$ are.

\begin{theorem}[The leading term $\Theta$ for $\omega = 0$] \label{thm: Theta 2}
Suppose that the assumptions in the first paragraph of Theorem \ref{thm: main result} hold, and let $\Theta_{\phi_1,\phi_2}^\eta(E_1,E_2)$ be the function from Theorem \ref{thm: main result}. Suppose in addition that $\omega = 0$. Then there exists a constant $c_1 > 0$ such that the following holds for $E = E_1 = E_2$  satisfying \eqref{D leq kappa}.
\begin{enumerate}
\item
For $d =1,2,3$ we have
\begin{equation} \label{Theta3D0}
\Theta_{\phi_1,\phi_2}^\eta(E,E) \;=\; \frac{(d + 2)^{d/2} }{2 \beta \pi^{2 + d} \nu(E)^4} \pbb{\frac{\eta}{\nu(E)}}^{d/2 - 2} \pb{V_d(\phi_1, \phi_2) + O(M^{-c_1})}\,.
\end{equation}
\item
For $d = 4$ we have
\begin{equation} \label{Theta4D0}
\Theta_{\phi_1,\phi_2}^\eta(E,E) \;=\; \frac{36 }{\beta \pi^6 \nu(E)^4}
\pb{   V_4(\phi_1, \phi_2) \abs{\log\eta} +
O(1)}\,.
\end{equation}
\end{enumerate}
\end{theorem}

In order to describe the behaviour of $\Theta$ in the regime $\omega \gg \eta$, for $d = 1,2,3$ we introduce the constants 
\begin{equation}\label{def_K13}
K_d \;\deq\;  2 \re  \int_{\R^d} \frac{\dd x}{(\ii+\abs{x}^2)^2}\,;
\end{equation}
explicitly, 
\begin{equation*} 
K_1 = - \frac{\pi}{\sqrt{2}}\,, \qquad  K_2 = 0 \,,\qquad
K_3 = \sqrt{2} \pi^2\,.
\end{equation*}

\begin{theorem}[The leading term $\Theta$ in the regime $\omega \gg \eta$] \label{thm: Theta 1}
Suppose that the assumptions in the first paragraph of Theorem \ref{thm: main result} hold, and let $\Theta_{\phi_1,\phi_2}^\eta(E_1,E_2)$ be the function from Theorem \ref{thm: main result}.  Suppose in addition that
\begin{equation} \label{eta_Delta_2}
\eta \;\leq\; M^{-\tau} \omega
\end{equation}
for some arbitrary but fixed $\tau > 0$.  Then there exists a constant $c_1 > 0$ such that the following holds for $E_1,E_2$ satisfying \eqref{D leq kappa} for small enough $c_* > 0$. 

\begin{enumerate}
\item
For $d = 1, 2, 3$ we have
\begin{equation} \label{Theta_23}
\Theta_{\phi_1,\phi_2}^\eta(E_1,E_2) \;=\; \frac{(d + 2)^{d/2}}{2 \beta \pi^{2 + 3 d / 2}\nu(E)^4} \pbb{\frac{\omega}{\nu(E)}}^{d/2 - 2} \pB{K_d + O \pb{\sqrt{\omega} + M^{-c_1}}}\,.
\end{equation}
\item
For $d = 2$  \eqref{Theta_23} does not identify 
the leading term since  $K_2=0$. The leading nonzero correction to the
vanishing leading term is 
\begin{equation}\label{C1d2}
\Theta_{\phi_1,\phi_2}^\eta(E_1,E_2) \;=\; \frac{8}{\beta \pi^5 \nu(E)^4}
\pbb{\pi \nu(E) \, \frac{\eta } {\omega^2 + 4 \eta^2}  - \frac{\abs{\log \omega}}{3} + O(1)}
\end{equation} 
 in the  case \textbf{(C1)}
and 
\begin{equation}\label{C2d2}
\Theta_{\phi_1,\phi_2}^\eta(E_1,E_2) \;=\; \frac{8}{\beta \pi^5 \nu(E)^4}
\pbb{  - \frac{\abs{\log \omega}}{3} + O(1)}\, 
\end{equation}
in the case \textbf{(C2)}.
\item
For $d = 4$ we have
\begin{equation} \label{Theta_4}
\Theta_{\phi_1,\phi_2}^\eta(E_1, E_2) \;=\; \frac{36}{\beta \pi^6 \nu(E)^4} \pb{ \abs{\log \omega} + O (1)
}\,.
\end{equation}
\end{enumerate}
\end{theorem}

Note that the leading non-zero terms in the expressions \eqref{Theta3D0}, \eqref{Theta4D0}, \eqref{Theta_23}--\eqref{Theta_4} 
are much larger than the additive error term in \eqref{EY_result}. Hence,
Theorems \ref{thm: main result} and \ref{thm: Theta 2} give a proof of the first Altshuler-Shklovskii formula, \eqref{ASvar2}.
Similarly, Theorems \ref{thm: main result} and \ref{thm: Theta 1} give a proof of the second Altshuler-Shklovskii formula, \eqref{AScorr2}.

The additional term in \eqref{C1d2} as compared to \eqref{C2d2} originates from the heavy Cauchy tail in the test functions $\phi_1, \phi_2$ at large distances. In Theorem \ref{thm: Theta 1} (ii) we give the leading correction, of order $\abs{\log \omega}$, to the vanishing main term for $d = 2$. For $d = 1$ the leading correction (to the nonzero main term of order $\omega^{-3/2}$) is of order $\omega^{-1/2}$; we omit the details.

\begin{remark} \label{rem:regimes}
As explained in the introduction, a phase transition in
the mesoscopic statistics occurs at a specific energy scale, the Thouless energy $\eta_c$.
For random band matrices the Thouless energy
is given by
\begin{equation} \label{def_eta_c}
\eta_c \;=\; W^2/L^2
\end{equation}
(see \cite{EK3}).
The Altshuler-Shklovskii formulas  are expected to hold in the entire \emph{diffusive regime} $\eta \gg \eta_c$.
In the complementary \emph{mean-field} regime, $\eta \ll \eta_c$, the mesoscopic statistics
are no longer given by the Altshuler-Shklovskii formulas.
We emphasize that this  phase transition in the mesoscopic statistics does not coincide with
 the celebrated metal-insulator transition between the Poisson and the Wigner-Dyson-Mehta (WDM) behaviour
 for the microscopic eigenvalue
statistics. It may be tempting to extrapolate the formulas obtained for the microscopic statistics
to mesoscopic scales, but this yields a wrong answer in some regimes.  The relation between the microscopic and mesoscopic statistics is in fact more intricate. We sketch it in the two following paragraphs. 

In the diffusive regime
 the mesoscopic linear statistics are governed by the Altshuler-Shklovskii formulas irrespective of the microscopic statistics. 
 For instance, if $d = 1$ the microscopic eigenvalue statistics are expected to be Poisson for $L \gg W^2$ and WDM for $L \ll W^2$. The condition $\eta \gg \eta_c$ may be satisfied in both regimes, and Theorems \ref{thm: main result}--\ref{thm: Theta 1} show that, in the diffusive regime, the mesoscopic linear statistics are the same for $L \gg W^2$ and $L \ll W^2$.

In the mean-field regime, $\eta \ll \eta_c$, on the other hand, we expect the mesoscopic statistics to be governed by the microscopic statistics. Let $d = 1$ for definiteness. Then for $\eta \ll \eta_c$ and $L \ll W^2$ the behaviour of $\cal N_\eta$ is governed WDM statistics, and the formulas \eqref{ASvar2} and \eqref{AScorr2} hold with $d = 0$ (see \cite[Section 2.3]{EK3} for more details and a proof). On the other hand, for $\eta \ll \eta_c$ and $L \gg W^2$ the behaviour of $\cal N_\eta$ is governed by Poisson statistics, so that \eqref{ASvar2} is replaced with
\begin{equation*}
\var \cal N_\eta(E) \;=\; \avg{\cal N_\eta(E)} \;\sim\; L \eta\,,
\end{equation*}
and the right-hand side of \eqref{AScorr2} is replaced with $0$.  Summarizing, in the mean-field
regime we expect the extrapolation of the microscopic statistics to mesoscopic scales to be valid. 

We expect this behaviour to be representative of general $d$-dimensional disordered Hamiltonians, 
and in particular to hold also for the Anderson model.
\end{remark}

\begin{remark} \label{rem:rho_assump}
Our results hold under the two assumptions
\begin{equation} \label{assumed_bounds1}
L \;\gg\; W^{1 + d/6}
\end{equation}
and
\begin{equation} \label{assumed_bounds2}
\eta \;\gg\; W^{-d/3}\,.
\end{equation}
(The upper bound on $L$ in \eqref{LW_assump} is purely technical and may be relaxed, as explained in the last sentence of Theorem \ref{thm: main result}.)
As mentioned in Remark \ref{rem:regimes}, we expect our results to hold under the sole assumption
\begin{equation} \label{AS_regime}
\eta \;\gg\; \eta_c\, 
\end{equation}
Recalling \eqref{def_eta_c}, this condition is equivalent to $L\gg W\eta^{-1/2}$.
We therefore find that 
\eqref{assumed_bounds1} and \eqref{assumed_bounds2} imply 
\eqref{AS_regime}.

In fact, the assumption \eqref{assumed_bounds1} is not essential;
it is just a simple way to guarantee that we are in the diffusive regime, $L\gg W\eta^{-1/2}$, 
 for all $\eta$ satisfying \eqref{assumed_bounds2}. 
Our results and our proofs (including those in companion paper \cite{EK3}) 
remain valid verbatim if we replace the assumptions \eqref{assumed_bounds1} and \eqref{assumed_bounds2} with the weaker assumptions \eqref{assumed_bounds2} and \eqref{AS_regime}.

Finally, we comment on the two essential  assumptions, 
 \eqref{assumed_bounds2} and \eqref{AS_regime}. The  assumption
 \eqref{assumed_bounds2} is technical but important for our proof;
 it guarantees that only a few terms
in our diagrammatic expansion contribute  in leading order. 
 It is used crucially in the proof of Proposition \ref{prop:large_Sigma_general}; see \cite[Section 4.4]{EK3} for more a detailed explanation. 
Relaxing this assumption will be the subject of future work.

The assumption  \eqref{AS_regime} is physically important as it characterizes 
the diffusive regime. This condition is used when we evaluate the leading order diagrams,
  which give rise to the  Altshuler-Shklovskii formulas. 
We stress, however, that  the essence of our method remains valid even if \eqref{AS_regime} is not satisfied
 (i.e.\ we leave the diffusive regime), under the sole assumption \eqref{assumed_bounds2}.
In that case our expansion technique can still be used to identify the leading behaviour of
the density-density correlation, 
 but the asymptotic behaviour of the leading terms is different. 
See Section \ref{sec:mean-field} below for more details. 
\end{remark}

\subsection{Higher-order correlations} \label{sec:clt_result}

The following result extends Theorem \ref{thm: main result} to arbitrary correlation functions of the mesoscopic densities. It may be interpreted as a Wick theorem, stating that the joint law of the densities is asymptotically Gaussian with covariance matrix $(\Theta^\eta_{\phi_i, \phi_j}(E_i, E_j))_{i,j}$.

\begin{theorem}[The joint law is asymptotically Gaussian] \label{thm:clt}
Fix $\rho \in (0, 1/3)$, $d \in \N$, and $k \in \N$.  Set $\eta \deq M^{-\rho}$.
Let $\phi_1, \dots, \phi_k$ be test functions satisfying either all \textbf{(C1)} or all \textbf{(C2)}. Fix $\kappa>0$,  let
 $E_1, \dots, E_k \in [-1 + \kappa, 1 - \kappa]$, and suppose that \eqref{LW_assump} holds.
Abbreviate
\begin{equation*}
X_i \;\deq\; \frac{Y^\eta_{\phi_i}(E_i) - \E Y^\eta_{\phi_i}(E_i)}{\E Y^\eta_{\phi_i}(E_i)}\,.
\end{equation*}
Then for small enough $c_*$ in \eqref{D leq kappa} the $k$-point correlation function satisfies
\begin{equation} \label{clt_statement}
\E \prod_{i = 1}^k X_i \;=\; \sum_{p \in \fra M(k)} \prod_{\{i,j\} \in p} \E (X_i X_j)
+ O \pBB{ 
\pbb{\frac{W}{L}}^{d/2} \pbb{\frac{R_4(\omega_0 + \eta)}{(LW)^d}}^{k/2}}\,,
\end{equation}
where $\fra M(k)$ denotes the set of pairings of $\{1, \dots, k\}$ and we abbreviated
\begin{equation} \label{def_R4}
R_4(s) \;\deq\; 1 + \ind{d \leq 3} s^{d/2 - 2} + \ind{d = 4} \absb{\log s}
\end{equation}
as well as $\omega_0 \deq \min_{i \neq j} \abs{E_i - E_j}$.  (Note that if $k$ is odd then the leading term of \eqref{clt_statement} is zero by convention.)
\end{theorem}

We remark that the error bound in \eqref{clt_statement} is not optimal and may be easily improved. We chose this form to obtain a simple expression that covers all cases of interest. Since the error term carries an extra power of $(W/L)^{d/2}$ as compared to the main term (see Theorems \ref{thm: Theta 2} and \ref{thm: Theta 1}), it is easy to see that in all regimes studied in Theorems \ref{thm: Theta 2} and \ref{thm: Theta 1} the error term in \eqref{clt_statement} is subleading provided $L \geq W^K$ for some large enough $K$. We also note that some of the energies $E_i$ in the theorem may coincide, in which case $\omega_0=0$.

A concrete corollary of Theorems \ref{thm:clt} and \ref{thm: Theta 2} is the following result. It says that at a fixed energy the rescaled finite-dimensional marginals of the process $(Y^\eta_\phi(E))_\phi$ converge to those of a Gaussian process with covariance $V_d(\cdot, \cdot)$.

\begin{corollary}[Convergence to a Gaussian process at a fixed energy]
Suppose that the assumptions of Theorem \ref{thm:clt} hold, and that in addition $E_1 = \cdots = E_k = E$ for some fixed $k$. 
 Let $d \leq 3$. For $i = 1, \dots, k$ define the random variable 
\begin{equation*}
\wt X_i \;\deq\; (LW)^{d/2} \pBB{\frac{(d + 2)^{d/2}}{2 \beta \pi^{2 + d} \nu(E)^4} \pbb{\frac{\eta}{\nu(E)}}^{d/2 - 2}}^{\!-1/2} \pBB{\frac{Y^\eta_{\phi_i}(E) - \E Y^\eta_{\phi_i}(E)}{\E Y^\eta_{\phi_i}(E)}}\,.
\end{equation*}
Then, as $W \to \infty$, 
 the random vector $(\wt X_1, \dots, \wt X_k)$ converges  in distribution 
to a mean-zero Gaussian vector with covariance matrix $(V_d(\phi_i, \phi_j))_{i,j = 1}^k$.
 A similar result holds for $d = 4$, whose details we omit.
\end{corollary}

\subsection{General band matrices} \label{sec:gen1}

The results of Sections \ref{sec:unimodular_results} and \ref{sec:clt_result} were stated for the unimodular band matrices defined in Section \ref{sec:setup1}. In this section we extend these results to a general class of band matrices. Roughly, we generalize the unimodular band matrices of Section \ref{sec:setup1} in two ways: the variances $S_{xy}$ may be given by an arbitrary profile on the scale $W$ (instead of the uniform profile of \eqref{step S}), and the law of $A_{xy}$ may be an arbitrary symmetric law with sufficient decay.

\subsubsection*{Definition of model}
As in Section \ref{sec:setup1}, we assume that the upper-triangular entries of $H = H^*$ are independent random variables with mean zero. We set
\begin{equation}\label{ST}
S_{xy} \;\deq\; \E \abs{H_{xy}}^2 \,, \qquad T_{xy} \;\deq\; \E H_{xy}^2\,.
\end{equation}
We assume that the law of $H_{xy}$ is symmetric, i.e.\ that $H_{xy}$ and $-H_{xy}$ have the same law. Moreover, we assume that (for nonzero $S_{xy}$) the entries $A_{xy} \deq (S_{xy})^{-1/2} H_{xy}$ have uniform subexponential decay, in the sense that
\begin{equation} \label{subexp_decay}
\P(\abs{A_{xy}} > \xi) \;\leq\; C \ee^{-\xi^c}
\end{equation}
for some constants $c,C > 0$ and for any $\xi>0$.

The remaining assumptions are on the deterministic profile matrices $S$ and $T$; roughly, we assume that $S$ and $T$ are translation invariant (in the sense of \eqref{transl-inv} below) and live on the scale $W$, but allow them to be otherwise arbitrary (up to the trivial constraint $\abs{T_{xy}} \leq S_{xy}$). We shall describe such a general profile using three fixed functions $f,g,h \col \R^d \to \R$. We require that the profile $S$ be given in terms of $f$ according to
\begin{equation} \label{general S}
S_{xy} \;=\; \frac{1}{M - 1} f \pbb{\frac{[x - y]_L}{W}}\,, \qquad M \;\deq\; \sum_{x \in \bb T} f \pbb{\frac{x}{W}}\,,
\end{equation}
where $[x]_L$ denotes the canonical representative of $x \in \Z^d$ in the torus $\bb T$. Similarly, we require that the profile $T$ be given in terms of $f$, $g$, and $h$ according to
\begin{equation} \label{general T}
T_{xy} \;=\;  \frac{1}{M - 1} f(z) \, [1 - \varphi h(z)]\, \ee^{\ii \lambda g(z)} \,, \qquad z \;\deq\; \frac{[x - y]_L}{W}\,.
\end{equation}
Here $\varphi, \lambda \in [0,1]$ are parameters that may depend on $L$. Note that \eqref{general S} and \eqref{general T} are the most general matrices $S$ and $T$ that are translation invariant, are given by a fixed profile on the scale $W$, and satisfy the trivial constraint $\abs{T_{xy}} \leq S_{xy}$.

We say that a function $f \col \R^d \to \R$ is \emph{piecewise $C^1$} if there exists a finite collection of disjoint open sets $U_1, \dots, U_n$ with piecewise $C^1$ boundaries,  whose closures cover $\R^d$, such that $f$ is $C^1$ on each $U_i$.  We also say that a function $f$ is \emph{piecewise $C^1$ with bounded derivatives} if it is piecewise $C^1$ and $\nabla f$ is bounded on each $U_i$. 
We always make the following assumptions on the functions $f$, $g$, and $h$. 
\begin{itemize}
\item[(A$f$)]
We assume that $f \col \R^d \to \R$ is an even, bounded, nonnegative, piecewise $C^1$ function, such that $f$ and $\abs{\nabla f}$ are integrable. We also assume that
\begin{equation} \label{moment of f}
\int_{\R^d} \dd x \, f(x) \, \abs{x}^{4 + c} \;<\; \infty
\end{equation}
for some $c > 0$.

Moreover, we introduce the covariance matrix of $f$,
\begin{equation} \label{definition of D0}
(D_0)_{ij} \;\deq\; \frac{1}{2} \int_{\R^d} x_i x_j f(x) \, \dd x\,,
\end{equation}
and assume that
\begin{equation} \label{D_bounded}
c \;\leq\; D_0 \;\leq\; C 
\end{equation}
in the sense of quadratic forms, for some positive constants $c$ and $C$.

\item[(A$g$)]
We assume that $g \col \R^d \to \R$ is an odd, bounded, piecewise $C^1$ function with bounded derivatives. 
We also assume that $g$ is not equal to a linear function on the support of $f$.
\item[(A$h$)]
We assume that $h \col \R^d \to \R$ is an even, piecewise $C^1$ function with bounded
derivatives, satisfying $0\leq h\leq 1$.
We also assume that $h$ is not identically zero on the support of $f$.
\end{itemize}

\begin{definition}[General band matrix] \label{def:gen_band}
We call the matrix $H$ a \emph{general band matrix} if it satisfies \eqref{ST}--\eqref{general T} for $f$, $g$, and $h$ satisfying (A$f$), (A$g$), and (A$h$) respectively.
\end{definition}

Note that under these assumptions we have \eqref{link between M and W}. Throughout the following we regard $f$, $g$, and $h$ as fixed, and do not track the dependence of the errors on them. The smoothness assumptions on $f$, $g$, and $h$ are technical and made for convenience. All other assumptions are natural: That $f$ and $h$ are even and $g$ is odd is clearly necessary since $H$ is Hermitian. 
The condition \eqref{D_bounded} guarantees that the system exhibits a non-degenerate diffusion.  
The condition \eqref{moment of f} is necessary only for $d=2$; in other dimensions, the
finiteness of $(2+c)$-th moment would be sufficient. (For $d=2$ the leading contribution arises from a fourth order Taylor expansion; hence the higher order moment assumption.) All of these assumptions on the decay of $f$ are made for conveniece. Indeed, our method may easily also handle heavy-tailed $f$, in which case the behaviour of $\Theta$ is different. (See Section \ref{sec: heavy tail} below for more details.)

Finally, the assumption that $g$ is not a linear function on the support of $f$ essentially amounts to excluding a trivial gauge transformation.
 Indeed, if $g$ were linear on the support of $f$, then (neglecting unimportant boundary issues on $\bb T$) the effect of the phase in \eqref{general T} simply amounts to a conjugation of $H$ with a unitary matrix. Hence, the final sentences of (A$g$) and (A$h$) are not restrictive; they simply fix  an ambiguity in the definition of the general band matrices. 
 Note that for $\varphi = 0$ (respectively $\lambda = 0$) the choice of $h$ (respectively $g$) is immaterial.

Note that by definition  $S$ and $T$ are translation invariant, 
 $S$ is real symmetric, and $T$ is Hermitian:
\begin{equation} \label{transl-inv}
S_{xy} \;=\; S_{x-y \, 0} \;=\; S_{yx} \;=\; \ol{S_{xy}}\,, \qquad  T_{xy} \;=\; T_{x - y \,  0} \;=\; \ol{T_{yx}}\,.
\end{equation}
Definition \ref{def:gen_band}  encompasses several important examples:
\begin{itemize}
\item[(a)]
\emph{The  complex Hermitian case ($\beta = 2$)}, where $T = 0$. 
\item[(b)]
\emph{The real symmetric case ($\beta = 1$)}, where $T = S$.
\item[(c)]
\emph{The rotated real symmetric case}, where $h = 0$. 
\end{itemize}
In addition, by varying the parameters $\varphi$ and $\lambda$ we may interpolate between these, and other, models; in particular, we may investigate the transition from $\beta = 1$ to $\beta = 2$ in the behaviour of the mesoscopic density statistics.

\subsubsection*{Results for general band matrices}

In the general band matrix model  of Definition \ref{def:gen_band}, for technical reasons outlined in Section \ref{sec:general_proof} below, we cannot control the errors in \eqref{EY_result} for arbitrary $\rho < 1/3$. Instead, we require the condition $\rho < c$ for some positive universal constant $c > 0$.

\begin{theorem}
If $H$ is the general band matrix model from Definition \ref{def:gen_band},  Theorems \ref{thm: main result} and \ref{thm:clt} are valid provided one replaces the assumption $\rho \in (0,1/3)$ with $\rho \in (0,c)$ for some universal constant $c > 0$. (One can take $c = 1/7$.)
\end{theorem}

The rest of this subsection is devoted to the asymptotics of the leading term $\Theta$, which has a more complicated behaviour than in Section \ref{sec:unimodular_results} since it depends on the parameters $f$, $g$, $h$, $\varphi$, and $\lambda$. This dependence on the functions $f$, $g$, and $h$ is encoded by the two fixed coefficients
\begin{equation} \label{def_Delta_0}
\Delta_0 \;\deq\; \inf_{q \in \R^d} \frac{1}{2} \int \pb{x \cdot q - g(x)}^2 \, f(x) \, \dd x \;=\; \frac{1}{2} \int_{\R^d} g(x)^2 f(x) \, \dd x - \absbb{\frac{1}{2} \int_{\R^d} D_0^{-1/2} x g(x) f(x) \, \dd x}^2
\end{equation}
and
\begin{equation} \label{def_Upsilon_0}
\Upsilon_0 \;\deq\; \int_{\R^d} h(x) f(x) \, \dd x \,.
\end{equation}
By assumption on $g$ and $h$, we have $\Delta_0 > 0$ and $\Upsilon_0 > 0$.
 The dependence of $\Theta$ on all of the quantities $f$, $g$, $h$, $\varphi$, and $\lambda$ takes place via the single quantity
\begin{equation} \label{def_sigma}
\sigma \;\deq\; \Delta_0 \lambda^2 + \Upsilon_0 \varphi\,,
\end{equation}
which may depend on $L$ through $\lambda$ and $\varphi$.

For $d \leq 3$ we generalize the definition of the quadratic form $V_d$ from \eqref{def_V_d} by defining 
\begin{equation} \label{def_V_d_a}
V_d(\phi_1, \phi_2; a) \;\deq\; \int_\R \dd t \, \abs{t}^{1 - d/2} \, \ee^{- a \abs{t}} \, \ol {\wh \phi_1(t)} \, \wh \phi_2(t)
\end{equation}
for $a \geq 0$. Note that $V_d(\phi_1, \phi_2; 0) = V_d(\phi_1, \phi_2)$ and $\lim_{a\to \infty} V_d(\phi_1, \phi_2; a) =0$.
The following result generalizes Theorem \ref{thm: Theta 2} to the general band matrix model.

\begin{theorem}[The leading term $\Theta$ for $\omega = 0$] \label{thm: Theta 2 gen}
Suppose that $H$ satisfies Definition \ref{def:gen_band}.  Suppose in addition that $\omega = 0$. Then there exists a constant $c_1 > 0$ such that the following holds for $E = E_1 = E_2$  satisfying \eqref{D leq kappa}.
\begin{enumerate}
\item
For $d =1,2,3$ we have
\begin{multline} \label{Theta3D0_gen}
\Theta_{\phi_1,\phi_2}^\eta(E,E) \;=\; \frac{1}{2^{2 + d/2} \pi^{2 + d} \nu(E)^4 \sqrt{\det D_0}} \pbb{\frac{\eta}{\nu(E)}}^{d/2 - 2} 
\\
\times \pbb{V_d(\phi_1, \phi_2) + V_d\pbb{\phi_1, \phi_2; \frac{2 \sigma}{\pi \nu(E) \eta}} + O(M^{-c_1})}\,.
\end{multline}
\item
For $d = 4$ we have
\begin{equation} \label{Theta4D0_gen}
\Theta_{\phi_1,\phi_2}^\eta(E,E) \;=\; \frac{1}{16 \pi^6 \nu(E)^4 \sqrt{\det D_0}}
 \pB{V_4(\phi_1, \phi_2) \pb{\abs{\log \eta} + \min\{\abs{\log \eta}, \abs{\log \sigma}\}} + O(1)} \,.
\end{equation}
\end{enumerate}
\end{theorem}

In particular, for $\sigma \ll \eta$ we recover the results of Theorem \ref{thm: Theta 2} for $\beta = 1$, and for $\sigma \gg \eta$ the results of Theorem \ref{thm: Theta 2} for $\beta = 2$. 
 Here we used that in the case \eqref{step S} we have the explicit expression
\begin{equation} \label{D_const}
D_0 \;=\; \frac{1}{2 (d+2)}\,.
\end{equation}

In order to describe the behaviour of $\Theta$ in the regime $\omega \gg \eta$, 
for $d = 1,2,3$ we introduce the constants
\begin{equation} \label{def_B_d}
B_d \;\deq\; \int_{\R^d} \dd x \, \frac{1}{(1 +  \abs{x}^2)^2}\,,
\end{equation}
so that $K_d = 2 B_d \re \ii^{d/2 - 2}$;
explicitly,
\begin{equation*} 
B_1 \;=\; \frac{\pi}{2} \,, \qquad B_2 \;=\; \pi \,, \qquad B_3 \;=\; \pi^2\,.
\end{equation*}
In addition, for $d = 2$ we also introduce the quantity
\begin{equation} \label{def_Q0}
Q_0 \;\deq\; \frac{1}{32} \int_{\R^2} \absb{D_0^{-1/2} x}^4 \, f(x) \, \dd x\,,
\end{equation}
which depends on the fourth moments $f$.

The following result generalizes Theorem \ref{thm: Theta 1} to the general band matrix model of Definition \ref{def:gen_band}.
\begin{theorem}[The leading term $\Theta$ in the regime $\omega \gg \eta$] \label{thm: Theta 1 gen}
Suppose that $H$ satisfies Definition \ref{def:gen_band}, and that \eqref{eta_Delta_2} holds for some $\tau > 0$. Then there exists a constant $c_1 > 0$ such that the following holds for $E_1,E_2$ satisfying \eqref{D leq kappa} for some small enough $c_* > 0$.

\begin{enumerate}
\item
For $d = 1, 2, 3$ we have
\begin{multline} \label{Theta_13_gen}
\Theta_{\phi_1,\phi_2}^\eta(E_1,E_2) \;=\; \frac{1}{2^{2 + d/2} \pi^{2 + 3 d / 2} \nu(E)^4 \sqrt{\det D_0}}
\pbb{\frac{\omega}{\nu(E)}}^{d/2 - 2}
\\
\times \pbb{K_d + 2 B_d \re \pbb{\ii + \frac{\pi \nu(E) \sigma}{2\omega}}^{d/2 - 2} + O \pb{\sqrt{\omega} + M^{-c_1}}}\,,
\end{multline}
where the fractional power is taken to be holomorphic in the right half-plane.
\item
For $d = 2$ and small $\sigma$, \eqref{Theta_13_gen} does not identify 
the leading term since  $K_2=0$. The leading nonzero correction to the
vanishing leading term is
\begin{multline} \label{Theta_2_gen_1}
\Theta_{\phi_1,\phi_2}^\eta(E_1,E_2) \;=\; \frac{1}{2 \pi^5 \nu(E)^4 \sqrt{\det D_0}}
\\
\times \pbb{ \frac{\pi\nu(E) [4 \eta  + \pi \nu(E) \sigma ]} {4 \omega^2 + (4 \eta + \pi \nu(E) \sigma)^2} +
\frac{\pi\eta \nu(E)} {\omega^2 + 4 \eta^2}
+ \p{Q_0 - 1} \pb{\abs{\log \omega} + \min\{\abs{\log \omega}, \abs{\log \sigma}\}} + O(1)}
\end{multline} 
in the  case \textbf{(C1)} 
and 
\begin{multline} \label{Theta_2_gen_2}
\Theta_{\phi_1,\phi_2}^\eta(E_1,E_2) \;=\; \frac{1}{2 \pi^5 \nu(E)^4 \sqrt{\det D_0}}
\\
\times \pbb{\frac{\pi^2 \nu(E)^2 \sigma} {4 \omega^2 + (\pi \nu(E) \sigma)^2} +
 \p{Q_0 - 1} \pb{\abs{\log \omega} + \min\{\abs{\log \omega}, \abs{\log \sigma}\}} + O(1)}
\end{multline} 
in the case \textbf{(C2)}. (Note that \eqref{Theta_2_gen_2} is obtained from \eqref{Theta_2_gen_1} by replacing $\eta$ with $0$.) 
\item
For $d = 4$ we have
\begin{equation} \label{Theta_4_gen}
\Theta_{\phi_1,\phi_2}^\eta(E_1, E_2) \;=\; \frac{1}{8 \pi^6 \nu(E)^4 \sqrt{\det D_0}}
 \pB{\abs{\log \omega} + \min\{\abs{\log \omega}, \abs{\log \sigma}\}  + O(1)}\,.
\end{equation}
\end{enumerate}
\end{theorem}

Similarly to the case $\omega = 0$, we note that in the case $\omega \gg \eta$ and $d = 1,3,4$ we have a transition from the case $\beta = 1$ to the case $\beta = 2$ depending on whether $\sigma \ll \omega$ or $\sigma \gg \omega$. For $d = 4$ this follows easily from \eqref{Theta_4_gen}, and for $d = 1,3$ from \eqref{Theta_13_gen} combined with 
$K_d = 2 B_d \re \ii^{d/2 - 2}$.
Here we used that in the case \eqref{step S} we have the explicit expression
\begin{equation} \label{Q_const}
\qquad Q_0 \;=\; \frac{2}{3}\,.
\end{equation}

Owing to $K_2=0$, the case $d=2$ is special; the correlation is determined by higher-order corrections to the algebraic cancellation in the integral \eqref{def_K13} for $d=2$.
 Similarly to the results in (ii) of Theorem~\ref{thm: Theta 1}, the first nonvanishing terms have a different structure. 
For definiteness, we focus on the case \textbf{(C2)}, i.e.\ \eqref{Theta_2_gen_2}. 
 The transition from $\beta = 1$ (for $\sigma \ll \omega^2 \abs{\log \omega}$)
 to $\beta = 2$ (for $\sigma \gg \abs{\log \omega}^{-1}$)
passes through a region of much stronger correlations, since in the regime  
$\omega^2 \abs{\log \omega} \ll \sigma \ll \abs{\log \omega}^{-1}$  the first term
in  \eqref{Theta_2_gen_2}  dominates over the logarithmic terms.

Interestingly, since all of the results in Theorems \ref{thm: Theta 2 gen} and \ref{thm: Theta 1 gen} depend on the parameters $\varphi$ and $\lambda$ only through their combination $\sigma$, we find that, for the purposes
of  mesoscopic statistics, decreasing the magnitude of $\E A_{xy}^2$ is equivalent to rotating $\E A_{xy}^2$ in the complex plane. In particular, either procedure may be used to probe the transition from $\beta = 1$ to $\beta = 2$.

\subsection{A remark on heavy-tailed band profiles} \label{sec: heavy tail}

The moment assumption \eqref{moment of f} on $f$ is not fundamental for our method. We imposed it to simplify the presentation of our results, since the behaviour of $\Theta$ for heavy-tailed $f$ is different. To illustrate this difference, we consider the case $d = 1$ and $f(x) = \frac{1}{\pi} \frac{1}{x^2 + 1}$. Then \eqref{EY_result} holds, whereby the leading term $\Theta$ is given for $\omega = 0$ by
\begin{equation} \label{heavy_tail}
\Theta_{\phi_1,\phi_2}^\eta(E,E) \;=\; \frac{1}{\beta \pi^4 \nu(E)^3} \frac{1}{\eta} \pb{V_2(\phi_1, \phi_2) + O(M^{-c_1})}\,,
\end{equation}
and for $\omega \gg \eta$ in the case \textbf{(C2)} by
\begin{equation} \label{heavy_tail2}
\Theta_{\phi_1,\phi_2}^\eta(E_1,E_2) \;=\; \frac{1}{2 \beta \pi^5 \nu(E)^4} \pb{- \abs{\log \omega} + O(1)}\,.
\end{equation}
We omit the proofs, which are identical to those of Theorems \ref{thm: Theta 2 gen} and \ref{thm: Theta 1 gen}, up to the explicit
calculation of the leading term.
Similar results hold for higher dimensions and for different heavy-tailed profiles $f$.

We remark that the one-dimensional band matrix model with a variance profile decaying as $f(x)\sim x^{-2}$ is conjectured in the physics literature \cite{MFDQS, ME} to be \emph{critical} in the sense that it describes a disordered quantum system at the Anderson (metal-insulator) transition. The other conjectured critical band matrix model is obtained by setting $d = 2$ and choosing $f$ with rapid decay (i.e.\ with light tail); this model was extensively studied in Sections \ref{sec:unimodular_results}--\ref{sec:gen1}. 

Comparing \eqref{heavy_tail} and \eqref{Theta3D0} for $d=2$, as well as \eqref{heavy_tail2} and \eqref{C2d2}, we note that both the one- and two-dimensional critical band matrix models have the same mesoscopic density fluctuations. 
 Defining $\cal N(I)$ as the (smoothed) number of eigenvalues in the mesoscopic interval $I$, we find in both cases that
\begin{equation} \label{compressibility}
\var(\cal N(I)) \;\approx\; C_d W^{-d} \, \E (\cal N(I))
\end{equation}
for some constants $C_1$ and $C_2$, assuming $\abs{I} \gg W^{- d  /3}$. (See \eqref{EY_result}, \eqref{Theta3D0}, \eqref{heavy_tail}, 
and \eqref{semicir}.) In particular,  the variance of $\cal N(I)$ is proportional to the length of $I$. This suggests weak correlations of $\cal N(I_1)$ and $\cal N(I_2)$ for disjoint $I_1$ and $I_2$, which was indeed established in \eqref{heavy_tail2} and \eqref{C2d2} for the one- and two-dimensional critical band matrices respectively.

The behaviour \eqref{compressibility} had been previously established in the physics literature; see e.g.\ \cite{CKL}. 
 Moreover, it was conjectured in \cite{CKL} that the proportionality in \eqref{compressibility} is equivalent to the \emph{multifractality} of the eigenvectors of $H$, and the proportionality constant $C_d W^{-d}$ (called the \emph{compressibility}) is directly related to the multifractal exponent. See \cite{ME} for a review.

\section{Path expansion and computation of the leading term} \label{sec:3}

We now begin the proof of Theorems \ref{thm: main result}--\ref{thm: Theta 1}. For simplicity, we assume
 throughout the proof that $\beta = 2$; the case $\beta = 1$ is similar and the minor modifications are sketched in 
Section \ref{sec:sym} below.

In this section we review the renormalized path expansion from \cite{EK3} that underlies our proof, and compute the leading term.
We first observe that, since the left-hand side of \eqref{EY_result} is invariant under the scaling $\phi \mapsto \lambda \phi$ for $\lambda \neq 0$, we assume without loss of generality that $\int \dd E \, \phi_i(E) = 2 \pi$ for $i = 1,2$. We shall make this assumption throughout the proof without further mention.

\subsection{Expansion in nonbacktracking powers}

We expand $\phi^\eta(H / 2 - E)$ in nonbacktracking powers $H^{(n)}$ of $H$, defined through
\begin{equation} \label{def: nb}
H^{(n)}_{x_0 x_n} \;\deq\; \sum_{x_1, \dots, x_{n - 1}} H_{x_0 x_1} \cdots H_{x_{n - 1} x_n} \prod_{i = 0}^{n - 2} 
\ind{x_i \neq x_{i + 2}}\,.
\end{equation}
From \cite{EK1}, Section 5, we find that
\begin{equation} \label{H^n and U_n}
H^{(n)} \;=\; U_n(H/2) - \frac{1}{M - 1} U_{n - 2}(H / 2)\,,
\end{equation}
where $U_n$ is the $n$-th Chebyshev polynomial of the second kind, defined through
\begin{equation} \label{definition of Un}
U_n(\cos \theta) \;=\; \frac{\sin (n + 1) \theta}{\sin \theta}\,.
\end{equation}
Note that \eqref{H^n and U_n} requires the deterministic condition $\abs{A_{xy}} = 1$ on the entries of $H$. As stated in Section \ref{sec:gen1}, this condition is not necessary for our proof, but does simplify it considerably. How to relax it is explained in Section \ref{sec:general_proof}.

From \cite{EK1}, Lemmas 5.3 and 7.9, we recall the expansion in nonbacktracking powers of $H$.
\begin{lemma} \label{lemma: Uexp}
For $t \geq 0$ we have
\begin{equation} \label{Chebyshev expansion of propagator}
\ee^{-\ii t H/2} \;=\; \sum_{n \geq 0} a_n(t) H^{(n)}\,,
\end{equation}
where
\begin{equation} \label{def of an}
a_n(t) \;\deq\; \sum_{k \geq 0} \frac{\alpha_{n + 2k}(t)}{(M - 1)^k}\,, \qquad \alpha_k(t) \;\deq\; 2 (- \ii)^k 
\frac{k+1}{t} J_{k + 1}(t)
\end{equation}
and $J_\nu$ denotes the $\nu$-th Bessel function of the first kind.

Moreover, we have
\begin{equation}\label{bounds on a_n}
\sum_{n \geq 0} \abs{a_n(t)}^2 \;=\; 1 + O(M^{-1}) \,, \qquad \abs{a_n(t)} \;\leq\; C \frac{t^n}{n!}\,.
\end{equation}
\end{lemma}

Throughout the following we denote by $\arcsin$ the analytic branch of $\arcsin$ extended to the real axis by continuity from the upper half-plane.
The following coefficients will play a key role in the expansion.
For $n \in \N$ and $E \in \R$ define
\begin{equation*}
\gamma_n(E) \;\deq\; \int_0^\infty \dd t \, \ee^{\ii E t}\, a_n(t)\,.
\end{equation*}
In \cite[Lemma 3.2]{EK3} 
we proved that
\begin{equation} \label{claim about gamma n}
\gamma_n(E) \;=\; \frac{2 (-\ii)^n \ee^{\ii (n+1) \arcsin E}}{1 - (M - 1)^{-1} \ee^{2\ii \arcsin E}}\,.
\end{equation}

Define
\begin{equation} \label{def_F_eta}
F_{\phi_1, \phi_2}^\eta(E_1,E_2) \;\equiv\; F^\eta(E_1,E_2) \;\deq\; \avgb{\tr \phi_1^\eta (H/2 - E_1) \,; \tr \phi_2^\eta (H/2 - E_2)}\,,
\end{equation}
where we used the notation \eqref{def_avg}.
 Note that the left-hand side of \eqref{EY_result} may be written as
\begin{equation}\label{ThetatoF}
\frac{\avg{Y^\eta_{\phi_1}(E_1) \, ; Y^\eta_{\phi_2}(E_2)}}{\avg{Y^\eta_{\phi_1}(E_1)} \avg{Y^\eta_{\phi_2}(E_2)}} \;=\; \frac{1}{N^2} \frac{ F^\eta(E_1,E_2)}{\E Y^\eta_{\phi_1}(E_1) \, \E Y^\eta_{\phi_2}(E_2)}\,.
\end{equation}
The expectations in the denominator are easy to compute using the local semicircle law for band matrices; see Lemma \ref{lem:EY} below.
Our main goal is to compute $ F^\eta(E_1,E_2)$.

Throughout the following we use the abbreviation
\begin{equation}\label{phipsi}
\psi(E) \;\deq\; \phi(-E)\,,
\end{equation}
and define $\psi^\eta$, $\psi_i$, and $\psi_i^\eta$ similarly similarly in terms of $\phi^\eta$, $\phi_i$, and $\phi_i^\eta$.
We also use the notation
\begin{equation} \label{def of convolution}
(\varphi * \chi)(E) \;\deq\; \frac{1}{2 \pi} \int \dd E' \, \varphi(E - E') \, \chi(E')
\end{equation}
to denote convolution.  The normalizing factor $(2 \pi)^{-1}$ is chosen so that $\wh {\varphi * \chi} = \wh \varphi \, \wh \chi$. Observe that
\begin{equation} \label{claim about gamma}
(\psi^\eta * \gamma_n)(E) \;=\; \int_0^\infty \dd t \, \ee^{\ii E t} \, \wh \phi(\eta t)\, a_n(t)\,.
\end{equation}
We note that in the case where $\phi(E) = \frac{2}{E^2 + 1}$, we have $\wh \phi(t) = \ee^{-\abs{t}}$. Hence \eqref{claim about gamma} implies in the case \textbf{(C1)}
\begin{equation} \label{Cauchy case identity}
(\psi^\eta * \gamma_n)(E) \;=\; \int_0^\infty \dd t \, \ee^{\ii (E + \ii \eta) t} \, a_n(t) \;=\; \gamma_n(E + \ii \eta)\,.
\end{equation}
We now return to the case of a general real $\phi$. Since $\phi$ is real, we have $\ol{\wh \phi(t)} = \wh \phi(-t)$.  We may therefore use Lemma \ref{lemma: Uexp} and Fourier transformation to get
\begin{equation} \label{calculation of phi eta}
\phi^\eta (H/2 - E) \;=\;
2 \re  \sum_{n = 0}^\infty H^{(n)} \int_0^\infty \dd t \, \wh \phi(\eta t) \, \ee^{\ii t E} a_n(t)
\;=\; \sum_{n = 0}^\infty H^{(n)} \, 2 \re (\psi^\eta * \gamma_n)(E)\,,
\end{equation}
where $\re$ denotes the Hermitian part of a matrix, i.e.\ $\re A \deq (A + A^*)/2$, and in the last step we used \eqref{claim about gamma} and the fact that 
$H^{(n)}$ is Hermitian.
We conclude that
\begin{equation} \label{expansion without trunction}
F^\eta(E_1, E_2) \;=\; \sum_{n_1, n_2 \geq 0} 2 \re \pb{(\psi_1^\eta * \gamma_{n_1})(E_1)} \, 2 \re \pb{(\psi_2^\eta * \gamma_{n_2})(E_2)} \, \avgb{\tr H^{(n_1)}\,; \tr H^{(n_2)}}\,.
\end{equation}

Because the combinatorial estimates of Section \ref{sec: part 2} deteriorate rapidly for $n \gg \eta^{-1}$, it is essential to cut off the terms $n > M^\mu$ in the expansion \eqref{expansion without trunction}, where $\rho < \mu < 1/3$. Thus, we choose a cutoff exponent $\mu$ satisfying $\rho < \mu < 1/3$. All of the estimates in this paper depend on $\rho, \mu$, and $\phi$; we do not track this dependence. The following result gives the truncated version of \eqref{expansion without trunction}, whereby the truncation is done in $n_i$ and in the support of $\wh \phi_i$.

\begin{proposition}[Path expansion with truncation] \label{prop: expansion with trunction}
Choose $\mu < 1/3$ and $\delta > 0$ satisfying $2 \delta < \mu - \rho < 3 \delta $. Define
\begin{equation} \label{definition of gamma tilde}
\wt \gamma_n(E, \phi) \;\deq\; \int_0^{M^{\rho + \delta}} \dd t \, \ee^{\ii E t} \, \wh \phi(\eta t)\, a_n(t)
\end{equation}
and
\begin{equation} \label{truncated series}
\wt F^\eta(E_1,E_2) \;\deq\; \sum_{n_1 + n_2 \leq M^\mu} 2 \re \pb{\wt \gamma_{n_1}(E_1,\phi_1)} \, 2 \re \pb{\wt \gamma_{n_2}(E_2, \phi_2)} \, \avgb{\tr H^{(n_1)}\,; \tr H^{(n_2)}}\,.
\end{equation}
Let $q > 0$ be arbitrary. Then for any $n\in \N$ and recalling \eqref{phipsi} we have the estimates
\begin{equation} \label{gamma - g}
\abs{(\psi_i^\eta * \gamma_{n})(E_i) - \wt \gamma_{n}(E_i, \phi_i)} \;\leq\; C_q M^{-q} \qquad (i = 1,2)
\end{equation}
and
\begin{equation} \label{F - wt F}
\absb{F^\eta(E_1,E_2) - \wt F^\eta(E_1,E_2)} \;\leq\; C_q N^2 M^{-q}.
\end{equation}
Moreover, for all $q > 0$ we have 
\begin{equation} \label{bound on gamma}
\absb{\wt \gamma_{n}(E_i, \phi_i)} +  \absb{(\psi_i^\eta * \gamma_{n})(E_i)} \;\leq\; \min \hb{C, C_q (\eta n)^{-q}}\,.
\end{equation}

If $\phi_1$ and $\phi_2$ are analytic in a strip containing the real axis, the factors $C_q M^{-q}$ on the right-hand sides of \eqref{gamma - g} and \eqref{F - wt F} may be replaced with $\exp(-M^c)$ for some $c > 0$, and the factor $C_q (\eta n)^{-q}$ on the right-hand side of \eqref{bound on gamma} by $\exp(- (\eta n)^c)$.
\end{proposition}

The proof of Proposition \ref{prop: expansion with trunction} is given in Appendix \ref{app: path expansion}.

\subsection{The behaviour of $\wt F^\eta(E_1, E_2)$}

Our main goal is to compute $\wt F^\eta(E_1, E_2)$ from \eqref{truncated series}. In order to clarify the argument, it is actually helpful to generalize the assumptions on the matrix of variances $S$. (This more general setup is also used in the generalization of Section \ref{sec:gen1}.) We suppose that $S_{xy}$ is given by \eqref{general S} for some $f$ satisfying the assumption (A$f$) from Section \ref{sec:gen1}.
We introduce the covariance matrices of $S_{x0}$ and $f$, defined through
\begin{equation} \label{definition of D}
D_{ij} \;\deq\; \frac{1}{2}\sum_{x \in \bb T} \frac{x_i x_j}{W^2} S_{x0} \,, \qquad (D_0)_{ij} \;\deq\; \frac{1}{2} \int_{\R^d} x_i x_j f(x) \, \dd x\,.
\end{equation}
(Recall also \eqref{definition of D0}.)
It is easy to see that $D = D_0 + O(W^{-1})$.
Note, that since \eqref{D_bounded} holds for $D_0$, it also holds for $D$ for large enough $W$.
In addition, for $d = 2$ we also introduce the quantities
\begin{equation} \label{def_Q}
Q \;\deq\; \frac{1}{32} \sum_{x \in \bb T} S_{x0} \, \absbb{ D^{-1/2} \, \frac{x}{W}}^4\,, \qquad
Q_0 \;\deq\; \frac{1}{32} \int_{\R^2} \absb{D_0^{-1/2} x}^4 \, f(x) \, \dd x\,.
\end{equation}
(Recall also \eqref{def_Q0}.)
As above, it is easy to see that $Q = Q_0 + O(W^{-1})$.

The main result of this section is summarized in the following Proposition \ref{prop: main}, which establishes the leading asymptotics of $\wt F^\eta(E_1,E_2)$, defined in \eqref{truncated series}, for small $\omega = E_2 - E_1$. The basic strategy is an expansion of the expectation on the right-hand side of \eqref{truncated series} in terms of graphs, as explained in Section \ref{sec_41}  below. As it turns out, the leading contribution arises from eight \emph{skeleton graphs}, called the \emph{dumbbell skeletons} in
Section \ref{sec:classification_skeletons} below, whose combined contribution is denoted by $\cal V_{\rm{main}} \equiv (\cal V_{\rm{main}})^\eta_{\phi_1, \phi_2}(E_1,E_2)$. It is given explicitly by
\begin{multline}\label{VDdef}
\cal V_{\rm{main}} \;=\; \sum_{b_1, b_2 = 0}^\infty \sum_{(b_3, b_4) \in \cal A} \indb{b_1 + b_2 + b_3 + b_4 \leq M^\mu / 2}
\\
\times 2 \re \pb{\wt \gamma_{2 b_1 + b_3 + b_4}(E_1,\phi_1)} \, 2 \re \pb{\wt \gamma_{2 b_2 + b_3 + b_4}(E_2, \phi_2)} \, \cal I^{b_1 + b_2} \tr S^{b_3 + b_4}\,, 
\end{multline}
where we defined
\begin{equation} \label{def_cal_A}
\cal A \;\deq\; \pb{\h{1,2,\dots} \times \h{0,1,\dots}} \setminus \hb{(2,0), (1,1)}\,.
\end{equation}
and
\begin{equation} \label{def_I}
\cal I \;\equiv\; \cal I_M \;\deq\; \frac{M}{M - 1}\,.
\end{equation}
(The choice of the symbol $\cal I$ suggests that for most purposes $\cal I$ should be thought of as $1$.)

\begin{proposition} \label{prop: main}
Suppose that the assumptions of the first paragraph of Theorem \ref{thm: main result} hold. Suppose moreover that $S$ is given by \eqref{general S} with a function $f$ satisfying {\rm (A$f$)} and \eqref{D_bounded}. Then there is a constant $c_0 > 0$ such that, for any $E_1,E_2$ satisfying \eqref{D leq kappa} for small enough $c_* > 0$, we have
\begin{equation} \label{main_prop_error}
\wt F^\eta(E_1, E_2) \;=\; \cal V_{\rm{main}} + \frac{N}{M} O \pb{M^{-c_0} R_2(\omega + \eta)} \,,
\end{equation}
where the leading contribution $\cal V_{\rm{main}}$ from \eqref{VDdef} satisfies the following estimates.
\begin{enumerate}
\item
Suppose that \eqref{eta_Delta_2} holds. Then for $d = 1, 2, 3$  we have
\begin{equation} \label{V3C2}
\cal V_{\rm{main}} \;=\;
\frac{ (2/\pi)^{d/2}}{\nu(E)^2 \sqrt{\det D}}\pbb{\frac{L}{2 \pi W}}^d \pbb{\frac{\omega}{\nu(E)}}^{d/2 - 2} \pB{K_d + O \pb{\omega^{1/2} + M^{-\tau/2}}}
\end{equation}
where $K_d$ was defined in \eqref{def_K13}. Moreover, for $d = 4$ we have
\begin{equation} \label{V4C2}
\cal V_{\rm{main}} \;=\; \frac{8}{\nu(E)^2\sqrt{\det D}} \pbb{\frac{L}{2 \pi  W}}^d
 \pb{\abs{\log \omega}  + O (1)}\,.
\end{equation}
\item
Suppose that \eqref{eta_Delta_2} holds and that $d = 2$. If $\phi_1$ and $\phi_2$ satisfy \textbf{(C1)} then
\begin{equation} \label{V2C1}
\cal V_{\rm{main}} \;=\; \frac{8}{\pi\nu(E)^2 \sqrt{\det D}} \pbb{\frac{L}{2 \pi W}}^2  \pBB{\frac{\pi\eta \nu(E)  } {\omega^2 + 4 \eta^2} +
 \p{Q - 1} \abs{\log \omega} + O(1)}\,,
\end{equation}
and if $\phi_1$ and $\phi_2$ satisfy \textbf{(C2)} then
\begin{equation} \label{V2C2}
\cal V_{\rm{main}} \;=\; \frac{8}{\pi \nu(E)^2 \sqrt{\det D}} \pbb{\frac{L}{2 \pi W}}^2 \pb{  (Q - 1) \abs{\log \omega} + O(1)}\,.
\end{equation}
\item
Suppose that $\omega = 0$. Then the exponent $\mu$ from Proposition \ref{prop: expansion with trunction} may be chosen so that there exists an exponent $c_1 > 0$ such that for $d =1,2,3$ we have
\begin{equation} \label{V3C2D0}
\cal V_{\rm{main}} \;=\; \frac{2^{d/2}}{ \nu(E)^2\sqrt{\det D}} \pbb{\frac{L}{2 \pi W}}^d
  \pbb{\frac{\eta}{\nu(E)}}^{d/2 - 2} \pb{ V_d(\phi_1, \phi_2) + O(M^{-c_1})}
\end{equation}
and for $d = 4$ we have
\begin{equation} \label{V4C2D0}
\cal V_{\rm{main}} \;=\; \frac{4}{\nu(E)^2 \sqrt{\det D}} \pbb{\frac{L}{2 \pi W}}^4 \, \pb{V_4(\phi_1, \phi_2) \abs{\log \eta} + O(1)}\,.
\end{equation}
\end{enumerate}
\end{proposition}

The proof consists of two independent parts.
\begin{itemize}
\item[(a)]
The asymptotic analysis of the right-hand side of $\eqref{VDdef}$, which yields \eqref{V3C2}--\eqref{V4C2D0}.
\item[(b)]
The estimate of the error terms, which yields \eqref{main_prop_error}.
\end{itemize}
Of these two, (b) represents the main work and is done in Section \ref{sec: part 2}. The rest of this section is devoted to (a).

We note that Proposition \ref{prop: main} is stated as Proposition 4.1 in \cite{EK3}. 
The asymptotics of the leading term, stated in \eqref{V3C2}--\eqref{V4C2D0}, are established in the current paper. 
The key result \eqref{main_prop_error} was proved in \cite[Section 4]{EK3}, but only under several simplifying assumptions, called {\bf (S1)}--{\bf (S3)} there.
Section \ref{sec: part 2} of the current paper gives the general
proof of \eqref{main_prop_error} by showing that the errors arising from the simplifications
 {\bf (S1)}--{\bf (S3)} in \cite{EK3} are negligible.

\subsection{Computation of the leading term in the case \textbf{(C1)}} \label{sec:dumbbell_C1}
We now perform part (a) of the proof of Proposition \ref{prop: main}, i.e.\ we compute $\cal V_{\rm{main}}$. As it turns out, the computation of the contribution of the dumbbell skeletons in the case where $\phi_1$ and $\phi_2$ satisfy \textbf{(C1)} is different, and somewhat simpler, than in the case where they satisfy \textbf{(C2)}. Hence, in this subsection we focus on the case \textbf{(C1)}, and devote the next one to the case \textbf{(C2)}. 
\begin{proposition}[Dumbbell skeletons in the case \textbf{(C1)}]
\label{prop: leading term}
Suppose that $\phi_1$ and $\phi_2$ satisfy \textbf{(C1)}, that \eqref{LW_assump} holds, and that \eqref{D leq kappa} holds for some small enough $c_* > 0$.
\begin{enumerate}
\item
Suppose that \eqref{eta_Delta_2} holds. Then $\cal V_{\rm{main}}$ satisfies \eqref{V3C2} for $d = 1,2,3$  and \eqref{V4C2} for $d = 4$.
\item
 Suppose that \eqref{eta_Delta_2} holds. Then $\cal V_{\rm{main}}$ satisfies \eqref{V2C1} for $d = 2$. 
\item
Suppose that $\omega = 0$. Then $\cal V_{\rm{main}}$ satisfies \eqref{V3C2D0} for $d = 1,2,3$ and \eqref{V4C2D0} for $d = 4$.
\end{enumerate}
\end{proposition}

The rest of this subsection is devoted to the proof of Proposition \ref{prop: leading term}. We begin by introducing some notation that we shall use throughout this and the following subsection. For $s > 0$ and $k\in \N$ define
\begin{equation} \label{def R}
R_k(s) \;\deq\; 1 + \ind{d \leq k - 1} s^{(d - k)/2} + \ind{d = k} \absb{\log s}\,,
\end{equation}
which generalizes $R_2(s)$ defined in \eqref{def_R2} and 
$R_4(s)$ defined in \eqref{def_R4}. 
The parameter $R_k(s)$ will be used in estimates of the form
\begin{equation} \label{use_of_R}
\int_{\R^d} \dd x \, \ind{\abs{x} \leq \epsilon} \frac{\abs{x}^l}{(\zeta + \abs{x}^2)^{k/2}} \;\leq\; C_\epsilon R_{k - l}(\abs{\zeta})\,,
\end{equation}
where $\zeta$ satisfies $\re \zeta \geq 0$ and $\abs{\zeta} \leq 3$, and $0 \leq l \leq k$ are nonnegative integers.

The computation of the main term will rely on the following asymptotic results on the resolvent of the matrix $S$.
 Its proof is given in Appendix \ref{appendix: S}.

\begin{proposition} \label{prop: bounds on S}
Let $S$ be as in \eqref{general S} and $\alpha \in \C$ satisfy $\abs{\alpha} \leq 1$ and $\abs{1 - \alpha} \geq 4 / M + (W/L)^2$.
\begin{enumerate}
\item
There exists a constant $C > 0$, depending only on $d$ and the profile function $f$, such that
\begin{equation} \label{bound res 1}
\normbb{\frac{1}{1 - \alpha S}}_{\ell^\infty \to \ell^\infty} \;\leq\; \frac{C \log N}{2 - \abs{1 + \alpha}}\,.
\end{equation}
Under the same assumptions we have, for each $k = 1,2, \dots$,
\begin{equation} \label{bound res 2}
\sup_{x,y} \absbb{\pbb{\frac{S}{(1 - \alpha S)^k}}_{xy}} \;\leq\; \frac{C}{M} R_{2k}(\abs{1 - \alpha})\,,
\end{equation}
where the constant $C$ depends only on $d$, $f$, and $k$.

\item
Suppose that $\alpha$ in addition satisfies $\re \alpha \geq 0$ and $\abs{\alpha} \geq 1/2$. 
Abbreviate $u \deq \abs{1 - \alpha}$ and let $\zeta\in \bb S^1$ be defined through $1-\alpha =  u \zeta$.
 Then for $d = 1,2,3$ we have
\begin{multline} \label{asymptotics of trace d leq 3}
\tr \frac{S}{\p{1 - \alpha S}^2}
\\
=\; \frac{u^{d/2 - 2}}{\sqrt{\det D}} \pbb{\frac{L}{2 \pi W}}^d  \pBB{ B_d \, \zeta^{d/2 - 2}  + O \pbb{ \exp \pbb{- \frac{c L \sqrt{u}}{W}}  + \frac{1}{M u} + u + \ind{d = 2} u \abs{\log u} + \ind{d = 3} u^{1/2}}}\,,
\end{multline} 
 where $B_d$ was defined in \eqref{def_B_d}.
Here the power $\zeta^{d/2 - 2}$ of $\zeta$ is taken to be analytic in the right half-plane; note that by assumption on $\alpha$ we have $\re \zeta > 0$. 
Moreover, under the same assumptions we have for $d = 4$
\begin{equation} \label{asymptotics of trace d 4}
\tr \frac{S}{\p{1 - \alpha S}^2} \;=\; \frac{\pi^2}{\sqrt{\det D}} 
 \pbb{\frac{L}{2 \pi  W}}^d  \pb{\abs{\log u} + O(1)}\,.
\end{equation}
\item
For $d=2$, under the assumptions of (ii), we have the more precise two-term asymptotics
\begin{equation} \label{precise asymptotics d=2}
\tr \frac{S}{\p{1 - \alpha S}^2} \;=\; \frac{1}{\sqrt{\det D}} \pbb{\frac{L}{2 \pi W}}^2  \pBB{\frac{\pi}{u \zeta} + \pi \p{Q - 1} \abs{\log u} + O\pbb{1 + \frac{1}{M u^2} + \frac{1}{u} \exp \pbb{-\frac{cL\sqrt{u}}{W}}}}\,,
\end{equation}
where $Q$ was defined in \eqref{def_Q}. 
\end{enumerate}
\end{proposition}

In order to apply Proposition \ref{prop: bounds on S} to the proof of Proposition \ref{prop: leading term}, we
introduce the abbreviations
\begin{equation} \label{def:AA} 
A_i \;\deq\; \arcsin E_i \,, \qquad A_i^\eta \;\deq\; \arcsin (E_i + \ii \eta)\,,
\end{equation}
which we shall tacitly use throughout the following.

\begin{proof}[Proof of Proposition \ref{prop: leading term}]
We begin by rewriting \eqref{VDdef} as
\begin{multline} \label{common expression for V(D)}
\cal V_{\rm{main}} \;=\; \sum_{b_1, b_2 = 0}^{\infty} \sum_{(b_3, b_4) \in \cal A} \, 2 \re \pb{\gamma_{2 b_1 + b_3 + b_4} * \psi_1^\eta}(E_1) \, 2 \re \pb{\gamma_{2 b_2 + b_3 + b_4} * \psi^\eta_2}(E_2) \, \cal I^{b_1 + b_2} \tr S^{b_3 + b_4}
\\
+ O_q(N M^{-q})\,,
\end{multline}
which follows easily 
 using \eqref{bound on gamma} to get rid of the condition on the summation variables $b_1, \dots, b_4$, 
as well as \eqref{gamma - g} to replace $\wt \gamma_{n}(E_i, \phi_i)$ with $(\gamma_{n} * \psi_i^\eta)(E_i)$.

Now we make use of the special form \eqref{Cauchy} of $\phi_1$ and $\phi_2$ from Assumption \textbf{(C1)}: using \eqref{Cauchy case identity} we find
\begin{equation*}
\cal V_{\rm{main}} \;=\; \sum_{b_1, b_2 = 0}^{\infty} \sum_{(b_3, b_4) \in \cal A} \, 2 \re \gamma_{2 b_1 + b_3 + b_4}(E_1 + \ii \eta) \, 2 \re \gamma_{2 b_2 + b_3 + b_4}(E_2 + \ii \eta) \, \cal I^{b_1 + b_2} \tr S^{b_3 + b_4}
+ O_q(N M^{-q})\,.
\end{equation*}
We may now plug in the expression \eqref{claim about gamma n} and sum over $b_1$ and $b_2$. Abbreviating
\begin{equation} \label{def_T}
\qquad T(z) \;\deq\; \frac{2}{1 - (M - 1)^{-1} \ee^{2 \ii \arcsin(z)}}\,,
\end{equation}
and recalling the definition \eqref{def:AA}, we get
\begin{multline} \label{Vmain1}
\cal V_{\rm{main}} \;=\; \sum_{(b_3, b_4) \in \cal A} \,
2 \re \pbb{T(E_1 + \ii \eta) \, \frac{\ee^{\ii A_1^\eta}}{1 + \ee^{2 \ii A_1^\eta} \cal I} \, \, (-\ii \ee^{\ii A_1^\eta}) ^{b_3 + b_4}}
\\
\times
2 \re \pbb{T(E_2 + \ii \eta) \, \frac{\ee^{\ii A_2^\eta}}{1 + \ee^{2 \ii A_2^\eta} \cal I} \, \, (-\ii \ee^{\ii A_2^\eta}) ^{b_3 + b_4}}  \, \tr  S^{b_3 + b_4} + O_q(N M^{-q})\,.
\end{multline}

We now prove part (i) of Proposition \ref{prop: leading term}. Thus, we assume that \eqref{eta_Delta_2} holds. 
Writing out
\begin{equation} \label{2re2re}
(2 \re x_1) (2 \re x_2) = 2  \re (x_1 \ol x_2 + x_1 x_2)
\end{equation}
yields
\begin{equation}\label{vdd}
\cal V_{\rm{main}} \;=\; 2 \re (\cal V_{\rm{main}}' + \cal V_{\rm{main}}'') + O_q(N M^{-q})
\end{equation}
in self-explanatory notation. We focus on $\cal V_{\rm{main}}'$; the analysis of $\cal V_{\rm{main}}''$ is similar. In the summation over $b_3$ and $b_4$ in the definition of $\cal V_{\rm{main}}'$, we replace the set $\cal A$ with $\pb{\h{1,2,\dots} \times \h{0, 1, \dots}}$ and subtract the terms $(b_3,b_4) = (2,0), (1,1)$. This gives $\cal V_{\rm{main}}' = \cal V_{\rm{main},0}' - \cal V_{\rm{main},1}'$, where
\begin{align}
\cal V_{\rm{main},0}' &\;\deq\;
\sum_{b_3 = 0}^\infty \sum_{b_4 = 1}^\infty \,
T(E_1) \, \frac{\ee^{\ii A_1^\eta}}{1 + \ee^{2 \ii A_1^\eta} \cal I} \, (-\ii \ee^{\ii A_1^\eta}) ^{b_3 + b_4}\, 
\ol{T(E_2)} \, \frac{\ee^{-\ii \ol A_2^\eta}}{1 + \ee^{-2 \ii \ol A_2^\eta} \cal I} \, \, (\ii \ee^{- \ii \ol A_2^\eta})^{b_3 + b_4}  \, \tr  S^{b_3 + b_4}
\notag \\ \label{V'_main_computed}
&\;=\;
T(E_1) \ol{T(E_2)} \, \frac{\ee^{\ii A_1^\eta}}{1 + \ee^{2 \ii A_1^\eta} \cal I} \, \frac{\ee^{-\ii \ol A_2^\eta}}{1 + \ee^{-2 \ii \ol A_2^\eta} \cal I}
\tr \frac{\ee^{\ii (A_1^\eta - \ol A_2^\eta)} S}{\pb{1 - \ee^{\ii (A_1^\eta - \ol A_2^\eta)} S}^2}
\end{align}
and
\begin{equation*}
\cal V_{\rm{main},1}' \;\deq\; 2
T(E_1 + \ii \eta) \, \ol{T(E_2  + \ii \eta)} \, \frac{\ee^{\ii A_1^\eta}}{1 + \ee^{2 \ii A_1^\eta} \cal I} \, \frac{\ee^{-\ii \ol A_2^\eta}}{1 + \ee^{-2 \ii \ol A_2^\eta} \cal I} \, \ee^{2 \ii (A_1^\eta - \ol A_2^\eta)} \tr  S^2\,.
\end{equation*}
(Note that the two exceptional terms
 $(b_3,b_4) = (2,0), (1,1)$ actually give the same contribution; this gives rise to the
prefactor 2, since the summands on the right-hand side of \eqref{Vmain1} depend on $b_3$ and $b_4$ only through their sum $b_3 + b_4$.)
We first focus on the easier term, $\re \cal V_{\rm{main},1}'$, which we shall simply estimate in absolute value. An elementary estimate yields
\begin{equation} \label{1/cos}
\frac{\ee^{\ii A_i^\eta}}{1 + \ee^{2 \ii A_i^\eta} \cal I} \;=\; \frac{\ee^{\ii A_i}}{1 + \ee^{2 \ii A_i}} + O(\eta) \;=\; O(1)\,, \qquad
\frac{\ee^{\ii A_i}}{1 + \ee^{2 \ii A_i}} \;=\; \frac{1}{2 \sqrt{1 - E_i^2}} \;=\; \frac{1}{\pi\nu_i}\,,
\end{equation}
where we abbreviated $\nu_i \deq \nu(E_i)$.
Using $T(E_i + \ii \eta) = 2+ O(M^{-1})$ and $\tr S^2 \leq CN/M$ by \eqref{tr_Sn_bound}, we find
\begin{equation} \label{V'1}
\absb{\cal V_{\rm{main},1}'} \;\leq\; \frac{CN}{M}\,.
\end{equation}

Next, we compute $\re \cal V_{\rm{main},0}'$. Writing $\alpha \deq \ee^{\ii (A_1^\eta - \ol A_2^\eta)}$ and $\nu \equiv \nu(E)$, 
and using \eqref{1/cos}, we find
\begin{equation}\label{V'_0_computation}
  \cal V_{\rm{main},0}' \;=\;  \frac{4}{\pi^2\nu_1\nu_2} 
\tr \frac{\alpha S}{\pb{1 - \alpha S}^2} \pb{1 + O (\eta)}
\;=\; \frac{4}{\pi^2\nu^2}
\tr \frac{\alpha S}{\pb{1 -\alpha S}^2} \pb{1 + O (\omega)}\,,
\end{equation}
where we used that $M^{-1} \leq \eta \leq \omega$.
In order to estimate the trace, we invoke \eqref{bound res 2}. 
Expanding $\alpha$ the variable $E_i + \ii \eta - E$ yields
\begin{equation} \label{alpha_Delta_sqrt_kappa}
\alpha \;=\; \ee^{\ii (A_1^\eta - \ol A_2^\eta)} \; = \; 1 + \frac{2\ii}{\pi\nu} (\omega + 2 \ii \eta) + O \p{\omega^2}\,.
\end{equation}
We write $1-\alpha$ in polar form: $1-\alpha = u \zeta$ with $u \deq \abs{1-\alpha}$ and $\zeta \in \bb S^1$. This yields
\begin{equation*}
u \;=\; \frac{2\omega}{\pi\nu} \pb{1 + O(\eta/\omega + \omega)}\,, \qquad \zeta \;=\; - \ii + O(\eta/\omega + \omega)\,.
\end{equation*}
We may plug this into the formula \eqref{asymptotics of trace d leq 3} for the case $d \leq 3$. By assumption on $\eta$ and $\omega$, we have $u \asymp \omega$ and $\alpha = 1 + O(\omega)$. Thus we get from \eqref{asymptotics of trace d leq 3} for $d \leq 3$ that
\begin{equation} \label{calV_prime_d=3}
\cal V_{\rm{main},0}' \;=\; \frac{(2/\pi)^{d/2}}{\nu^2 \sqrt{\det D}} \pbb{\frac{L}{2 \pi W}}^d \pbb{\frac{\omega}{\nu}}^{d/2 - 2}
 \pbb{ B_d (-\ii)^{d/2 - 2}  + O \pbb{\frac{\eta}{\omega} + \omega^{1/2} + \exp \pbb{- \frac{c L \omega^{1/2}}{W}} }}\,,
\end{equation}
where we used that $\eta \geq M^{-1}$.
Similarly, if $d = 4$ we get from \eqref{asymptotics of trace d 4} that
\begin{align}
\cal V_{\rm{main},0}' &\;=\; \frac{4}{\nu^2\sqrt{\det D}}  \pbb{\frac{L}{2 \pi  W}}^d  \pb{\abs{\log u} + O(1)} (1 + O(\omega))
\notag \\ \label{calV_prime_d=4}
&\;=\; \frac{4}{\nu^2\sqrt{\det D}}  \pbb{\frac{L}{2 \pi  W}}^d \abs{\log \omega} \qbb{1 + O \pbb{\frac{1}{\abs{\log \omega}}}}\,.
\end{align}
This concludes the analysis of $\cal V_{\rm{main}}'$.

By a similar analysis, we find $\cal V_{\rm{main}}'' = \cal V_{\rm{main},0}'' - \cal V_{\rm{main},1}''$, where
\begin{equation} \label{V''1}
\absb{\cal V_{\rm{main},1}''} \;\leq\; \frac{CN}{M}
\end{equation}
and
\begin{equation} \label{V''0}
\abs{\cal V_{\rm{main},0}''} \;\leq\; C \absbb{\tr \frac{\ee^{\ii (A_1^\eta + A_2^\eta)} S}{\pb{1 + \ee^{\ii (A_1^\eta + A_2^\eta)} S}^2}}\,.
\end{equation}
We shall estimate the trace using \eqref{bound res 2}. To that end, we use the elementary estimate
\begin{equation} \label{bound e A1 A2}
\absb{1 + \ee^{\ii (A_1^\eta + A_2^\eta)}} \;\geq\; c\,,
\end{equation}
which follows from \eqref{D leq kappa}.
We conclude that
\begin{equation} \label{V''_estimate}
\abs{\cal V_{\rm{main},0}''} \;\leq\; \frac{CN}{M}\,.
\end{equation}
In particular, for $d \leq 3$ we have the weaker bound
\begin{equation} \label{V''01}
\abs{\cal V_{\rm{main},0}''} \;\leq\; C \absBB{\frac{1}{\nu^2\sqrt{\det D}} \pbb{\frac{L}{2 \pi W}}^d \pbb{\frac{\omega}{\nu}}^{d/2 - 2}\omega^{1/2}}\,.
\end{equation}
Similarly, for $d = 4$ we have the weaker bound
\begin{equation} \label{V''02}
\abs{\cal V_{\rm{main},0}''} \;\leq\; C \absBB{\frac{1}{\nu^2\sqrt{\det D}}  \pbb{\frac{L}{2 \pi  W}}^d \log \pbb{\frac{\omega}{\nu}} \frac{1}{\log \omega}}
\end{equation}

In order to conclude the proof of part (i), we observe that
\begin{equation} \label{Kd_V}
K_d \;=\;  2 B_d \re (-\ii)^{d/2 - 2} 
\end{equation}
for $d = 1,2,3$.
Plugging \eqref{V'1}, \eqref{calV_prime_d=3}, \eqref{calV_prime_d=4}, \eqref{V''1}, \eqref{V''01}, and \eqref{V''02} into
\begin{equation} \label{Vmain_computation}
\cal V_{\rm{main}} \;=\; 2 \re \pb{\cal V_{\rm{main},0}' - \cal V_{\rm{main},1}' + \cal V_{\rm{main},0}'' - \cal V_{\rm{main},1}''} + O_q(N M^{-q})
\end{equation}
from \eqref{vdd} completes the proof of part (i) of Proposition \ref{prop: leading term}.

The proof of part (ii) is similar. We find, exactly as in the proof of part (i), that
\begin{equation*}
\cal V_{\rm{main}} \;=\; 2 \re \cal V_{\rm{main},0}' + O \pbb{\frac{N}{M}}
\;=\; 2 \re  \frac{4}{\pi^2\nu^2} \tr \frac{S}{\pb{1 - \alpha S}^2} \pb{1 + O (\omega)} + O \pbb{\frac{N}{M}}
\end{equation*}
with $\alpha = \ee^{\ii (A_1^\eta - \ol A_2^\eta)}$.
Plugging \eqref{alpha_Delta_sqrt_kappa} into \eqref{precise asymptotics d=2} yields
\begin{equation*}
\cal V_{\rm{main},0}' \;=\; \frac{4}{\pi^2\nu^2}
\frac{1}{\sqrt{\det D}} \pbb{\frac{L}{2 \pi W}}^2  \pBB{\frac{\pi}{u \zeta} + \pi \p{Q - 1} \abs{\log u} + O(1)}\,.
\end{equation*}
Using \eqref{alpha_Delta_sqrt_kappa} on $u \zeta = 1 - \alpha$ we therefore easily get \eqref{V2C1}. This concludes the proof of \eqref{V2C1} and hence of part (ii).

What remains is the proof of part (iii) of Proposition~\ref{prop: leading term}.
 Thus, set $\omega = 0$ so that $E_1 = E_2 = E$. The details are similar to the above proof of part (i). The estimates \eqref{V'1}, \eqref{V''1}, and \eqref{V''_estimate} may be taken over verbatim. The only difference is the computation of the main contribution, $\cal V'_{\rm{main}, 0}$. From \eqref{V'_0_computation} we get
\begin{equation*}
\cal V_{\rm{main},0}' \;=\; \frac{4}{\pi^2\nu^2}
\tr \frac{\alpha S}{\p{1 - \alpha S}^2} \pb{1 + O (\eta)}\,, \qquad \alpha \;\deq\; \abs{\ee^{\ii \arcsin (E + \ii \eta)}}^2\,.
\end{equation*}
A simple expansion yields $\alpha = 1 - \frac{4 \eta}{\pi\nu} + O(\eta^2)$. Hence \eqref{asymptotics of trace d leq 3} yields, for $d = 1,2,3$,
\begin{equation*}
\cal V_{\rm{main},0}' \;=\; \frac{(2/\pi)^{d/2}}{\nu^2 \sqrt{\det D}} \pbb{\frac{L}{2 \pi W}}^d \pbb{\frac{2\eta}{\nu}}^{d/2 - 2} 
\qBB{ B_d  + O \pbb{\exp \pbb{- \frac{c L \eta^{1/2}}{W}} + \eta^{1/2}}}\,.
\end{equation*}
Here we used the inequality $\frac{1}{M \eta} \leq \eta^{1/2}$ to absorb the error term $\frac{1}{M \eta}$ 
into the last error term. 
Similarly, for $d = 4$ we get from \eqref{asymptotics of trace d 4}
\begin{equation*}
\cal V_{\rm{main},0}' \;=\; \frac{4}{\nu^2 \sqrt{\det D}} \pbb{\frac{L}{2 \pi  W}}^4 \abs{\log \eta} \qbb{1 + O\pbb{\frac{1}{\abs{\log \eta}}}}\,.
\end{equation*}
Now \eqref{V3C2D0} and \eqref{V4C2D0} in the case \textbf{(C1)} follow from \eqref{Vmain_computation} and a simple computation of 
 $V_d(\phi_1, \phi_2)$ from \eqref{def_V_d} for $\phi_1 = \phi_2$ given by \eqref{Cauchy}. 
\end{proof}

We conclude this subsection with a remark about the assumption on $\omega$ in the  case (i) of Proposition \ref{prop: leading term}. 
The result of Proposition \ref{prop: leading term} is only meaningful in the regime $\omega \to 0$, which is not imposed by our assumptions \eqref{D leq kappa} and \eqref{eta_Delta_2} on $E$, $\omega$, and $\eta$. If $\omega$ is of order one, the contribution of the dumbbell skeletons no longer has a simple universal form as in Proposition \ref{prop: leading term}. However, our results remain valid even if $\omega \asymp 1$. In that case, we need to replace Proposition \ref{prop: leading term} with the following result, which follows from a tedious and unenlightening calculation. 
\begin{proposition}
\label{prop: leading term details}
Suppose that $\phi_1$ and $\phi_2$ satisfy \textbf{(C1)}, and that \eqref{D leq kappa} holds for some small enough $c_* > 0$. Then 
\begin{equation}\label{VDD}
\cal V_{\rm{main}} \;=\; \frac{4}{\pi^2\nu_1\nu_2} \, \qBB{\cal D^\eta(E_1, E_2) + O \pBB{R_4 (\omega) \frac{N}{M} \pbb{\frac{1}{M} + \eta}}}
\end{equation}
with $\nu_i=\nu(E_i)$ and where
\begin{equation} \label{def Theta C1}
\cal D^\eta(E_1, E_2) \;\deq\; 2 \re \tr \frac{\ee^{\ii (A_1^\eta - \ol A_2^\eta)} S}{\pb{1 - \ee^{\ii (A_1^\eta - \ol A_2^\eta)} S}^2} - 2 \re \tr \frac{\ee^{\ii (A_1^\eta + A_2^\eta)} S}{\pb{1 + \ee^{\ii (A_1^\eta + A_2^\eta)} S}^2} - 8 (1 - 2 E_1^2) (1 - 2 E_2^2) \tr S^2\,.
\end{equation}
\end{proposition}

In the regime $\omega \ll 1$, the first summand in $\cal D^\eta(E_1, E_2)$ 
dominates; its leading asymptotics may be explicitly computed, which leads to the formulas in 
Proposition~\ref{prop: leading term}. If $\omega \asymp 1$, all
three terms typically are of the same order, and cannot be brought into a simpler form. In this case, barring a coincidental cancellation in $\cal D^\eta(E_1, E_2)$, these terms are all of order $M/N$. Hence, the error term in \eqref{VDD} is of subleading order. In the regime $\omega \asymp 1$, we may compute the traces in \eqref{def Theta C1} using a Riemann sum approximation  provided that $f$ is piecewise $C^\infty$ or that, for some $k \in \N$, all derivatives of order $k$ of $\wh f$ are integrable.
The result is
\begin{multline} \label{D_largeD}
\cal D^\eta(E_1, E_2) \;=\; \pbb{\frac{L}{2 \pi W}}^d 2 \re \int \qBB{\frac{\ee^{\ii (A_1 - A_2)} \wh f_*(q)}{\pb{1 - \ee^{\ii (A_1 - A_2)} \wh f_*(q)}^2} - \frac{\ee^{\ii (A_1 + A_2)} \wh f_*(q)}{\pb{1 + \ee^{\ii (A_1 + A_2)} \wh f_*(q)}^2}} \dd q
\\
- 8 \pbb{\frac{L}{W}}^d (1 - 2 E_1^2) (1 - 2 E_2^2) \int f_*(x)^2 \, \dd x + o_{\omega}((L/W)^d)\,,
\end{multline}
where $f_*(x) \deq f(x) / \int f(y) \, \dd y$ is the probability density associated with $f$.
This formula immediately shows that the critical dimension for
the universality of the correlation decay is $d=4$. Noting that  $\wh f_*(q) \sim 1- (q,Dq) + O(q^4) $
and $A_2-A_1 \sim \omega$,
the first integral is approximately
$$
 \int \frac{\dd q}{\omega^2 (q \cdot Dq)^2} \;\approx\; \frac{1}{\omega^2} \int \frac{\dd q}{ q^4}\,.
$$
For $d \leq 4$ the main contribution comes from the very
small $q$ regime and the details of $\wh f_*$ are irrelevant:
only the covariance matrix $D$ matters (which is essentially a constant
in the case of \eqref{step S}). For $d>4$, however, the
integral is not concentrated on the infrared regime  $q \approx 0$  and the specific form of $\wh f_*$,
in particular its decay properties, is essential and  influences the asymptotics of $\cal V_{\rm{main}}$.

\subsection{Computation of the leading term in the case \textbf{(C2)}} \label{sec:dumbbell_C2}

We now prove the analogue of Proposition \ref{prop: leading term} for the case that $\phi_1$ and $\phi_2$ satisfy \textbf{(C2)} instead of \textbf{(C1)}.
The calculation reveals that in the regime \eqref{eta_Delta_2} the explicit form of $\phi_i$ is not important, and only its integral
$\int \phi_i =2\pi$ matters. On the other hand, in the regime $\omega = 0$ the answer depends on $\phi_i$ via the quadratic form $V_d$ defined in \eqref{def_V_d}.  Throughout this section we use the notation of Section \ref{sec:dumbbell_C1}, and in particular \eqref{def R} and \eqref{def:AA}.

\begin{proposition}[Dumbbell skeletons in the case \textbf{(C2)}]
\label{prop: leading term 2}
Suppose that $\phi_1$ and $\phi_2$ satisfy \textbf{(C2)}, that \eqref{LW_assump} holds, and that \eqref{D leq kappa} holds for some small enough $c_* > 0$.
\begin{enumerate}
\item
Suppose that \eqref{eta_Delta_2} holds. Then $\cal V_{\rm{main}}$ satisfies \eqref{V3C2} for $d = 1,2,3$ and \eqref{V4C2} for $d = 4$.
\item
 Suppose that \eqref{eta_Delta_2} holds. Then $\cal V_{\rm{main}}$ satisfies \eqref{V2C2} for $d = 2$. 
\item
Suppose that $\omega = 0$. Then the exponent $\mu$ in Proposition \ref{prop: expansion with trunction} may be chosen so that $\cal V_{\rm{main}}$ satisfies \eqref{V3C2D0} for $d \leq 3$ and \eqref{V4C2D0} for $d = 4$.
\end{enumerate}
\end{proposition}

\begin{proof}[Proof of Proposition \ref{prop: leading term 2} \rm{(i)}]
Similarly to \eqref{common expression for V(D)}, we get
\begin{multline} \label{common expression for V(D) 2}
\cal V_{\rm{main}} \;=\; \sum_{b_1, b_2 = 0}^{[M^\mu] - 1} \sum_{(b_3, b_4) \in \cal A_\mu} \, 2 \re \pb{\gamma_{2 b_1 + b_3 + b_4} * \psi_1^\eta}(E_1) \, 2 \re \pb{\gamma_{2 b_2 + b_3 + b_4} * \psi^\eta_2}(E_2) \, \cal I^{b_1 + b_2} \tr S^{b_3 + b_4}
\\
+ O_q(N M^{-q})\,,
\end{multline}
where we defined
\begin{equation*}
\cal A_\mu \;\deq\; \pb{\h{1,2,\dots, [M^\mu]} \times \h{0,1, \dots, [M^\mu] - 1}}  \setminus \hb{(2,0), (1,1)}\,.
\end{equation*}
The only difference between \eqref{common expression for V(D)} and \eqref{common expression for V(D) 2} is that in \eqref{common expression for V(D) 2} we use the index set $\cal A_\mu$ instead of $\cal A$, thus dropping any terms with a summation index larger than $[M^\mu]$ (or $[M^\mu] - 1$).

Next, we split the integration domain of each convolution using a smooth, nonnegative, symmetric function $\chi$ satisfying $\chi(E) = 1$ for $\abs{E} \leq 1$ and $\chi(E) = 0$ for $\abs{E} \geq 2$. We split $\psi_i = \psi^{\leq}_i + \psi^{>}_i$, where
\begin{equation} \label{phi_cutoff1}
\psi^{\leq}_i(E) \;\deq\; \psi_i(E) \chi(M^{-\theta/2} E)\,, \qquad
\psi^{>}_i(E) \;\deq\; \psi_i(E) \pb{1 - \chi(M^{-\theta/2} E)}\,,
\end{equation}
for some positive constant $\theta > 0$. Here we take $\theta \deq \tau$, where $\tau$ is the constant from \eqref{eta_Delta_2}.

This yields the splitting $\psi^\eta_i = \psi^{\leq, \eta}_i + \psi^{>, \eta}_i$ of the rescaled test function $\psi^\eta(E) = \eta^{-1}\psi(\eta^{-1}E)$. This splitting is done on the scale $\eta M^{\theta/2}$, and we have
\begin{equation} \label{supp_phi_trunc}
\supp \psi_i^{\leq, \eta} \;\subset\; [-2 \eta M^{\theta/2}, 2 \eta M^{\theta/2}]\,. 
\end{equation}
Moreover, recalling \eqref{non_Cauchy} and using the trivial bound $\abs{\gamma_n(E)} \leq C$ we find
\begin{equation} \label{phi_geq_bound}
\absb{(\psi_i^{\leq,\eta} * \gamma_{n})(E_i)} \;\leq\; C \,, \qquad \absb{(\psi_i^{>,\eta} * \gamma_{n})(E_i)} \;\leq\; C_q M^{-q}
\end{equation}
for any $q > 0$.
Plugging the splitting $\psi^\eta_i = \psi^{\leq, \eta}_i + \psi^{>, \eta}_i$ into \eqref{common expression for V(D) 2} and using \eqref{phi_geq_bound} yields
\begin{multline} \label{VD_step2}
\cal V_{\rm{main}}
\;=\; \sum_{b_1, b_2 = 0}^{[M^\mu] - 1} \sum_{(b_3, b_4) \in \cal A_\mu} \, 2 \re \pb{\gamma_{2 b_1 + b_3 + b_4} * \psi_1^{\leq, \eta}}(E_1) \, 2 \re \pb{\gamma_{2 b_2 + b_3 + b_4} * \psi^{\leq, \eta}_2}(E_2) \, \cal I^{b_1 + b_2} \tr S^{b_3 + b_4}
\\
+ O_q(N M^{-q})\,.
\end{multline}
We use the same splitting \eqref{2re2re} as in \eqref{vdd}, and focus only on the term $\cal V'_{\rm{main}}$. As after \eqref{vdd}, we split $\cal V_{\rm{main}}'= \cal V_{\rm{main},0}' - \cal V_{\rm{main},1}'$ and focus only on leading term $\cal V_{\rm{main},0}'$. Using the same method as the one leading to \eqref{V'1},  one can easily show that
\begin{equation} \label{V_prime_1_estimate}
\cal V_{\rm{main},1}' = O\pbb{\frac{N}{M}}\,.
\end{equation}
Thus, we have to compute
\begin{align*}
\cal V'_{\rm{main},0} &\;=\;
\sum_{b_1, b_2 = 0}^{[M^\mu] - 1} \sum_{b_3 = 1}^{[M^\mu]} \sum_{b_4 = 0}^{[M^\mu] - 1} \, \pb{\gamma_{2 b_1 + b_3 + b_4} * \psi_1^{\leq, \eta}}(E_1) \, \pb{\ol{\gamma_{2 b_2 + b_3 + b_4}} * \psi^{\leq, \eta}_2}(E_2) \, \cal I^{b_1 + b_2} \tr S^{b_3 + b_4}
\\
&\;=\; \Biggl[T(E_1) \ol{T(E_2)} \, \frac{\ee^{\ii A_1}}{1 + \ee^{2 \ii A_1} \cal I} \, \frac{\ee^{-\ii A_2}}{1 + \ee^{-2 \ii A_2} \cal I}
\tr \Biggl(\frac{\ee^{\ii (A_1 - A_2)} S}{\pb{1 - \ee^{\ii (A_1 - A_2)} S}^2}
\\
&\qquad \times \pb{1 - (-\ee^{2 \ii A_1} \cal I)^{[M^\mu]}} \pb{1 - (-\ee^{-2 \ii A_2} \cal I)^{[M^\mu]}} \pb{1 - (\ee^{\ii (A_1 - A_2)} S)^{[M^\mu]}}^2 \Biggr) \Biggl]
{}*{} \psi_1^{\leq, \eta}(E_1) * \psi_2^{\leq, \eta}(E_2)\,.
\end{align*}
Next, we get rid of all terms with an exponent $[M^\mu]$. The basic idea is that any such term oscillates in its energy variable on a much smaller scale than the scale $\eta$ of the convolution with $\psi_i^{\leq, \eta}$. More precisely, we multiply out the three parentheses on the second line, and prove that any of the seven terms that is not $1$ yields a negligible contribution. All of these terms are treated in the same way. For definiteness, we focus on the term $(-\ee^{2 \ii A_1} \cal I)^{[M^\mu]}$. Thus, we have to estimate
\begin{equation} \label{conv_expression}
\qBB{T(E_1) \ol{T(E_2)} \, \frac{\ee^{\ii A_1}}{1 + \ee^{2 \ii A_1} \cal I} \, \frac{\ee^{-\ii A_2}}{1 + \ee^{-2 \ii A_2} \cal I}
(-\ee^{2 \ii A_1} \cal I)^{[M^\mu]} \, \tr \frac{\ee^{\ii (A_1 - A_2)} S}{\pb{1 - \ee^{\ii (A_1 - A_2)} S}^2}}
* \psi_1^{\leq, \eta}(E_1) * \psi_2^{\leq, \eta}(E_2)\,.
\end{equation}
In order to estimate this, we observe that, since $\theta = \tau$, we get from \eqref{supp_phi_trunc} and \eqref{eta_Delta_2} that
\begin{equation} \label{supp_phi_trunc 2}
\supp \psi_i^{\leq, \eta} \;\subset\; [-2 M^{-\theta/2} \omega, 2 M^{-\theta/2}\omega]\,.
\end{equation}
We denote the argument of $\psi_i^{\leq, \eta}$ in the convolution integral 
by $v_i$, and the corresponding argument in the bracket in \eqref{conv_expression} by $E_i^v : = E_i -v_i$.
By \eqref{supp_phi_trunc 2}, on the domain of integration we have $\abs{v_i} \leq 2 M^{-\theta/2} \omega$. Therefore we may estimate the first two denominators of \eqref{conv_expression}, for the integration variable $v_i$ in the support of $\psi_i^{\leq, \eta}$, as
\begin{equation} \label{1+Ie_geq}
\absb{1 + \cal I^{-1} \ee^{-2 \ii \arcsin E_i^v }} \;\geq\; \frac{1}{2} \absb{1 + \ee^{-2 \ii \arcsin E_i^v)}} \;=\; \sqrt{1 - (E_i^v)^2} \;\geq\; c
\end{equation}
for small enough $c_*$ in \eqref{D leq kappa}. Hence we may write \eqref{conv_expression} in the form
\begin{equation*}
\int \dd v_1 \, \psi_1^{\leq, \eta}(v_1) \, \pb{- \ee^{2 \ii \arcsin E_1^v }}^{[M^\mu]} \, \tr h(v_1)\,,
\end{equation*}
where $h$ is a smooth matrix-valued function on $\R$ with derivatives satisfying
\begin{equation*}
\norm{h^{(k)}(v)} \;\leq\; C_k  \eta^{-k - 1}
\end{equation*}
for all $k \in \N$. Since the phase factor $\phi(v_1) \deq 2 \arcsin(E_1 - v_1)$ is regular and $\abs{\phi'} \asymp 1$
on the support of $\psi_1^{\leq, \eta}$, a standard stationary phase argument using a $k$-fold integration by parts
implies that
 \eqref{conv_expression} is bounded by $N C_k \eta^{-k-1} M^{-k\mu} \leq C N M^{-1}$, where the second bound follows by choosing $k$ large enough.

We conclude that
\begin{equation} \label{Wprime0}
\cal V'_{\rm{main},0} \;=\; \qBB{T(E_1) \ol{T(E_2)} \, \frac{\ee^{\ii A_1}}{1 + \ee^{2 \ii A_1} \cal I} \, \frac{\ee^{-\ii A_2}}{1 + \ee^{-2 \ii A_2} \cal I}
\tr \frac{\ee^{\ii (A_1 - A_2)} S}{\pb{1 - \ee^{\ii (A_1 - A_2)} S}^2}}
* \psi_1^{\leq, \eta}(E_1) * \psi_2^{\leq, \eta}(E_2) + O\pbb{\frac{N}{M}}
\end{equation}
for all dimensions $d$. 
We have the elementary estimate $T(E_i^v) = 2 + O(M^{-1})$ and, using \eqref{1+Ie_geq},
\begin{equation*}
\frac{\ee^{\ii \arcsin E_i^v}}{1 + \cal I^{-1} \ee^{2 \ii \arcsin E_i^v }}
\;=\;
\frac{1}{2 \sqrt{1 - E^2}} \pbb{1 + O \pbb{\frac{1}{M} + \omega}} =\frac{1}{\pi\nu} \pbb{1 + O \pbb{\frac{1}{M} + \omega}}\,, \qquad \nu \equiv \nu(E)\,.
\end{equation*}
In order to compute the trace in \eqref{Wprime0}, we proceed exactly as in the proof of Proposition \ref{prop: leading term}; here we have $\alpha \deq \ee^{\ii (\arcsin E_1^v - \arcsin E_2^v))} = 1 - u \zeta$, where (on the domain of integration)
\begin{equation*}
u \;=\; \frac{2\omega}{\pi\nu} \pb{1 + O(M^{-\theta/2} + \omega)}\,, \qquad \zeta \;=\; - \ii + O(M^{-\theta/2} + \omega)\,.
\end{equation*}
Notice that after replacing $E_i - v_i$ with $E$, at the cost of a negligible error, and
thus removing the $v$-dependence in the first factor of the convolutions, the precise
form of $\psi_i$ becomes irrelevant and only $\int \psi_i(E) \, \dd E = 2 \pi$ matters.
Hence we get, for $d \leq 3$,
\begin{equation} \label{calW2}
\cal V'_{\rm{main},0} \;=\; \frac{(2/\pi)^{d/2}}{\nu^2 \sqrt{\det D}}\pbb{\frac{L}{2 \pi W}}^d \pbb{\frac{\omega}{\nu}}^{d/2 - 2} 
\pbb{B_d (-\ii)^{d/2 - 2}  + O \pbb{\exp \pbb{- \frac{c L \sqrt{u}}{W}} + \omega^{1/2} + M^{-\theta/2}}}\,,
\end{equation}
and, for $d = 4$,
\begin{equation} \label{calW3}
\cal V'_{\rm{main},0} \;=\; \frac{4}{\nu^2 \sqrt{\det D}} \pbb{\frac{L}{2 \pi  W}}^d \abs{\log \omega} \pbb{1 + O \pbb{\frac{1}{\abs{\log \omega}} + M^{-\theta/2} + \omega}}\,.
\end{equation}
For $\cal V''_{\rm{main}}$ we perform a similar estimate, using \eqref{bound e A1 A2} with $\eta = 0$ and the fact that ${\cal V''_{\rm{main}}}$ contains $\tr \frac{\alpha S}{(1-\alpha S)^2}$ with $\abs{1 - \alpha} \geq c$ (see \eqref{V''0}--\eqref{bound e A1 A2}), which may be estimated by \eqref{bound res 2}. This gives
\begin{equation} \label{calW4}
\abs{\cal V''_{\rm{main}}} \;\leq\; \frac{C N}{M}\,.
\end{equation}
The estimates \eqref{V_prime_1_estimate} and \eqref{calW2}--\eqref{calW4} conclude the analysis of $\cal V_{\rm{main}}$, and hence the proof of Proposition \ref{prop: leading term 2} (i).
\end{proof}

\begin{proof}[Proof of Proposition \ref{prop: leading term} \rm{(ii)}]
The argument is similar to the proof of part (i), and we only focus on what is different.
We start from \eqref{Wprime0} for $d = 2$, which we write as
\begin{equation}
\cal V'_{\rm{main},0} \;=\; \qBB{\frac{4 + O(\omega)}{\pi^2\nu^2} \tr \frac{S}{\pb{1 - \ee^{\ii (A_1 - A_2)} S}^2}}
* \psi_1^{\leq, \eta}(E_1) * \psi_2^{\leq, \eta}(E_2) + O\pbb{\frac{N}{M}}\,.
\end{equation}
Invoking \eqref{precise asymptotics d=2} with $u = \omega(1 + O(M^{-\tau}))$ yields
\begin{multline*}
\cal V'_{\rm{main},0} \;=\; \frac{4}{\pi^2\nu^2\sqrt{\det D}} \pbb{\frac{L}{2 \pi W}}^2
\\
\times \qBB{\pbb{\frac{\pi}{1 - \ee^{\ii (A_1 - A_2)}}} * \psi_1^{\leq, \eta}(E_1) * \psi_2^{\leq, \eta}(E_2) + \pi (Q - 1) \abs{\log \omega} + O(1)}
+ O\pbb{\frac{N}{M}}\,.
\end{multline*}
We compute the convolution integral using \eqref{alpha_Delta_sqrt_kappa}:
\begin{align*}
&\mspace{-40mu} \pbb{\frac{\pi}{1 - \ee^{\ii (A_1 - A_2)}}} * \psi_1^{\leq, \eta}(E_1) * \psi_2^{\leq, \eta}(E_2)
\\
&\;=\; \ii \frac{\pi^2\nu }{2}  \pb{1 + O(\omega)}
\int \frac{\dd v_1}{2 \pi} \, \frac{\dd v_2}{2 \pi} \, \frac{1}{\omega + v_1 - v_2} \, \psi_1^{\leq, \eta}(v_1) \, \psi_2^{\leq, \eta}(v_2)
\\
&\;=\; \ii  \frac{\pi^2\nu }{2}  \pbb{\frac{1}{\omega} + \frac{\eta}{\omega^2} \int \frac{\dd v_1}{2 \pi} \, \frac{d v_2}{2 \pi} \, (v_1 - v_2) \psi^{\leq}(v_1) \, \psi^\leq(v_2) + O(1)}\,.
\end{align*}
Here we used that, by \eqref{supp_phi_trunc 2}, we always have $\abs{\omega + v_1 - v_2} \geq \omega/2$ on the support of the function $\psi_1^{\leq, \eta}(v_1) \psi_2^{\leq, \eta}(v_2)$. 
Taking the real part of this expression yields simply $O(1)$. Hence we conclude that
\begin{equation*}
\cal V_{\rm{main}} \;=\; 2 \re \cal V'_{\rm{main},0} + O\pbb{\frac{N}{M}}
\;=\; \frac{4}{\pi^2\nu^2 \sqrt{\det D}} \pbb{\frac{L}{2 \pi W}}^2 \pb{2 \pi (Q - 1) \abs{\log \omega} + O(1)}\,.
\end{equation*}
This concludes the proof of part (ii) of Proposition \ref{prop: leading term 2}.
\end{proof}

In order to prove part (iii) of Proposition \ref{prop: leading term 2}, we shall need the following 
local decay bound.
\begin{lemma} \label{lem: lct}
For all $b \in \N$ we have
\begin{equation*}
(\cal I^{-1} S^b)_{yz} \;\leq\; \frac{C}{M b^{d/2}} + \frac{C}{N}
\end{equation*}
for some constant $C$ depending only on $f$.
\end{lemma}
\begin{proof}
This follows from a standard local central limit theorem; see for instance the proof in \cite{So1}. 
\end{proof}
In particular, for $1 \leq b \leq (L / W)^2$ we have
\begin{equation} \label{lclt}
(S^b)_{yz} \;\leq\; \frac{C}{M b^{d/2}}\,.
\end{equation}

\begin{proof}[Proof of Proposition \ref{prop: leading term 2} \rm{(iii)}] 
We assume that the exponent $\mu$ from Proposition \ref{prop: expansion with trunction} has been chosen so that
\begin{equation} \label{cond_mu_rho}
11 \mu \;<\; 12 \rho\,, \qquad 5 \mu \;<\; 1 + 2 \rho\,.
\end{equation}
(Recall that the most critical case is when both $\rho$ and $\mu$ are just slightly below 1/3.)
We use the truncated functions $\psi_i^{\leq, \eta}$ from \eqref{phi_cutoff1}, where $\theta$ is an exponent that satisfies
\begin{equation} \label{cond_theta}
6 \mu - 6 \rho \;<\; \theta \;<\; 2 \rho - \mu\,.
\end{equation}
We start from \eqref{VD_step2} with $E_1 = E_2 = E$, from which we get
\begin{multline} \label{VD_step5}
\cal V_{\rm{main}} \;=\; \sum_{b_1, b_2 = 0}^{[M^\mu] - 1} \sum_{b_3 = 1}^{[M^\mu]} \sum_{b_4 = 0}^{[M^\mu] - 1} \, 2 \re \pb{\gamma_{2 b_1 + b_3 + b_4} * \psi_1^{\leq, \eta}}(E) \, 2 \re \pb{\gamma_{2 b_2 + b_3 + b_4} * \psi^{\leq, \eta}_2}(E) \, \cal I^{b_1 + b_2} \tr S^{b_3 + b_4}
\\
+ O\pbb{\frac{N}{M}}\,;
\end{multline}
here we estimated the contribution of the two terms $(b_3,b_4) = (2,0), (1,1)$ excluded from $\cal A_\mu$ by $C N/M$. For the following we need the bound
\begin{equation} \label{no_osc}
\sum_{b_3 = 1}^{[M^\mu]} \sum_{b_4 = 0}^{[M^\mu]} \tr S^{b_3 + b_4} \;\leq\;
C N \sum_{b_3 = 1}^{[M^\mu]} \sum_{b_4 = 0}^{[M^\mu]} \frac{1}{M (b_3 + b_4)^{d/2}}  \;\leq\; \frac{C N}{M} R_4(M^{-\mu})\,,
\end{equation}
where in the first step we used \eqref{lclt}. Inserting the splitting $1 = \cal I^{-2 (b_1 + b_2)} + (1 - \cal I^{-2 (b_1 + b_2)})$ into the right-hand side of \eqref{VD_step5} and using \eqref{phi_geq_bound} as well as \eqref{no_osc} to estimate the second resulting term  yields
\begin{multline} \label{VD_step6}
\cal V_{\rm{main}} \;=\; \sum_{b_1, b_2 = 0}^{\infty} \sum_{b_3 = 1}^{[M^\mu]} \sum_{b_4 = 0}^{[M^\mu] - 1} \, 2 \re \pb{\gamma_{2 b_1 + b_3 + b_4} * \psi_1^{\leq, \eta}}(E) \, 2 \re \pb{\gamma_{2 b_2 + b_3 + b_4} * \psi^{\leq, \eta}_2}(E) \, \cal I^{-(b_1 + b_2)} \tr S^{b_3 + b_4}
\\
+ O\pbb{\frac{N}{M} M^{3 \mu - 1} R_4(M^{-\mu})}\,.
\end{multline}
Here we also used \eqref{phi_geq_bound} to extend the $b_1,b_2$-summation to $\infty$.

Next, we split
\begin{equation} \label{Vmain_D0}
\cal V_{\rm{main}} \;=\; 2 \re \pb{\cal V_{\rm{main}}' + \cal V_{\rm{main}}''} + O\pbb{\frac{N}{M} M^{3 \mu - 1} R_4(M^{-\mu})}
\end{equation}
using \eqref{2re2re} as in \eqref{vdd}. The error term $\cal V_{\rm{main}}''$ may be estimated exactly as in the proof of Proposition \ref{prop: leading term 2} (i), using the fact that $\phi_i^{\leq, \eta}$ has compact support; see \eqref{calW4}. The result is
\begin{equation} \label{calV4D0}
\abs{\cal V''_{\rm{main}}} \;\leq\; \frac{C N}{M}\,.
\end{equation}

What remains is the computation of
\begin{equation*}
\cal V_{\rm{main}}' \;=\; \sum_{b_1, b_2 = 0}^{\infty} \sum_{b_3 = 1}^{[M^\mu]} \sum_{b_4 = 0}^{[M^\mu] - 1} \, \pb{\gamma_{2 b_1 + b_3 + b_4} * \psi_1^{\leq, \eta}}(E) \, \pb{\ol{\gamma_{2 b_2 + b_3 + b_4}} * \psi^{\leq, \eta}_2}(E) \, \cal I^{-(b_1 + b_2)} \tr S^{b_3 + b_4}.
\end{equation*}
We begin by enforcing exponential convergence with a sufficient rate in the summation over $b_3$ and $b_4$. To that end, let $\xi$ be a constant satisfying
\begin{equation} \label{cond_epsilon}
6 \mu - 6 \rho \;<\; 3 \xi \;<\; \min \hb{2 \rho - \mu - \theta \,,\, \mu/2}
\end{equation}
(see \eqref{cond_mu_rho} and \eqref{cond_theta}),
and set
\begin{equation*}
J \;\deq\; 1 - M^{- \mu - \xi}
\end{equation*}
We introduce the splitting $1 = J^{b_3 + b_4} + (1 - J^{b_3 + b_4})$ into the summation in $\cal V'_{\rm{main}}$. In line with the abuse of notation $(\varphi * \chi)(E) \equiv \varphi(E) * \chi(E)$, in the following we use the notation
\begin{equation} \label{EAEi}
E_i \;=\; E - v_i \,, \qquad A_i \;=\; \arcsin E_i\,, 
\end{equation}
and abbreviate
\begin{equation} \label{conv_integral}
\varphi(E_1, E_2) * \psi_1^{\leq, \eta}(E) * \psi_2^{\leq, \eta}(E) \;\equiv\; \int \dd v_1\, \dd v_2 \, \varphi(E - v_1, E - v_2) \, \psi_1^{\leq, \eta}(v_1) \,\psi_2^{\leq, \eta}(v_2)\,. 
\end{equation}
Now the error term, resulting from the term $1 - J^{b_3 + b_4}$ in the above splitting, is
\begin{multline*}
\sum_{b_3 = 1}^{[M^\mu]} \sum_{b_4 = 0}^{[M^\mu] - 1} \,
\pbb{T(E_1) \ol{T(E_2)} \frac{\ee^{\ii A_1}}{1 + \cal I^{-1} \ee^{2 \ii A_1}} \frac{\ee^{-\ii A_2}}{1 + \cal I^{-1} \ee^{-2 \ii A_2}} (\ee^{\ii (A_1 - A_2)})^{b_3 + b_4} }
* \psi_1^{\leq, \eta}(E) * \psi_2^{\leq, \eta}(E)
\\
\times (1 - J^{b_3 + b_4}) \tr S^{b_3 + b_4} \;\leq\; CM^{-\xi} \sum_{b_3, b_4 = 0}^{[M^\mu]} \tr S^{b_3 + b_4} \;\leq\; \frac{C N}{M} M^{-\xi} R_4(M^{-\mu})\,;
\end{multline*}
here we obtained the first expression by using \eqref{claim about gamma n}
 and summing up the geometric series in the indices $b_1,b_2$; the first inequality follows from the estimate
\begin{equation} \label{1/cos expanded}
\frac{\ee^{\ii A_1}}{1 + \cal I^{-1} \ee^{2 \ii A_1}} \;=\; \frac{1}{2 \sqrt{1 - E^2}} + O\pbb{\frac{1}{M} + \eta M^{\theta/2}}
\;=\; \frac{1}{\pi\nu(E)} + O\pbb{\frac{1}{M} + \eta M^{\theta/2}}
\,,
\end{equation}
valid on the support of the convolution integral (see \eqref{supp_phi_trunc 2}),
and from  $\abs{1 - J^{b_3 + b_4}} \leq CM^{-\xi}$; the second inequality follows from \eqref{no_osc}. We therefore conclude that
\begin{multline} \label{Vmain7}
\cal V_{\rm{main}}' \;=\;
\sum_{b_3 = 1}^{\infty} \sum_{b_4 = 0}^{\infty} \,
\pbb{T(E_1) \ol{T(E_2)} \frac{\ee^{\ii A_1}}{1 + \cal I^{-1} \ee^{2 \ii A_1}} \frac{\ee^{-\ii A_2}}{1 + \cal I^{-1} \ee^{-2 \ii A_2}} (\ee^{\ii (A_1 - A_2)})^{b_3 + b_4} }
* \psi_1^{\leq, \eta}(E) * \psi_2^{\leq, \eta}(E)
\\
\times \tr (JS)^{b_3 + b_4} + O \pbb{\frac{N}{M} M^{-\xi} R_4(M^{-\mu})}\,,
\end{multline}
where we used \eqref{claim about gamma n} as well as the definitions \eqref{def_T} and \eqref{EAEi} 
 followed by an explicit summation over $b_1$ and $b_2$.  In order to simplify the right-hand side 
of \eqref{Vmain7}, we use \eqref{1/cos expanded} and $T(E - v_i) = 2 + O(M^{-1})$ to replace 
 $\ee^{\ii A_i} / (1 + \cal I^{-1} \ee^{2 \ii A_i})$ and $T(E_i)$  with $1/(2\sqrt{1-E^2})$ and 2, respectively.   The contribution of the error terms in both  replacements can be estimated by using the following general estimate for any fixed $k \in \N$ (here we use it for $k=2$):
\begin{equation} \label{no_osc2}
\tr \frac{S}{(1 - J S)^k} \;=\; \sum_{d_1 = 1}^\infty \sum_{d_2, \dots, d_k = 0}^\infty \tr (J S)^{d_1 + \cdots + d_k} \;\leq\; \frac{C_kN}{M} R_{2k}(M^{-\mu - \xi})\,,
\end{equation}
which plays a role similar to that of \eqref{no_osc} in providing a robust a-priori bound 
 which does not make use of oscillations. 
The estimate in \eqref{no_osc2} follows from
 \eqref{bound res 2} together with \eqref{cond_epsilon}. 
We find
\begin{multline} \label{Vmain6}
\cal V_{\rm{main}}' \;=\;
\frac{4}{\pi^2\nu^2} \sum_{b_3 = 1}^{\infty} \sum_{b_4 = 0}^{\infty} \,
\pb{(\ee^{\ii (A_1 - A_2)})^{b_3 + b_4} }
* \psi_1^{\leq,\eta}(E) * \psi_2^{\leq,\eta}(E) \, \tr (JS)^{b_3 + b_4}
\\
+ O \pbb{\frac{N}{M} R_4(M^{-\mu - \xi}) \pbb{\frac{1}{M} + \eta M^{\theta/2}} + \frac{N}{M} R_4(M^{-\mu}) M^{-\xi} }\,.
\end{multline}
Next, we sum up the geometric series on the right-hand side of \eqref{Vmain6}. We use the estimate, valid on the support of the convolution integral,
\begin{align*}
\tr \frac{ \ee^{\ii (A_1 - A_2)} JS}{\pb{1 - \ee^{\ii (A_1 - A_2)} JS}^2} &\;=\; \tr \frac{S}{\pb{\ee^{- \ii (A_1 - A_2)} - JS}^2} + O \pbb{\frac{N}{M} R_4(M^{-\mu - \xi}) M^{-\mu - \xi}}
\\
&\;=\; \tr \frac{S}{\pb{1 + \ii (1 - E^2)^{-1/2} (v_1 - v_2) - JS}^2}
\\
& \qquad + O \pbb{\frac{N}{M} R_4(M^{-\mu - \xi}) M^{-\mu - \xi} + \frac{N}{M} R_6(M^{-\mu - \xi}) M^\theta \eta^2}\,;
\end{align*}
here the last step follows from \eqref{alpha_Delta_sqrt_kappa} and  a short argument using a resolvent expansion together with \eqref{no_osc2} and the bound
\begin{equation}
M^\theta \eta^2 \;\leq\; M^{-\mu - \xi} M^{-3 \xi}\,,
\end{equation}
as follows from \eqref{cond_epsilon}. We omit further details. Plugging this into the right-hand side of \eqref{Vmain6} and recalling \eqref{cond_epsilon} as well as the definition \eqref{def R} of $R_k$, we get
\begin{multline} \label{Vmain8}
\cal V_{\rm{main}}' \;=\;
\frac{4}{\pi^2\nu^2} \, \frac{1}{(2 \pi)^2}\int \dd v_1 \, \dd v_2 \, \psi_1^\eta(v_1) \, \psi_2^\eta(v_2) \, \tr \frac{S}{\pb{1 + \ii (1 - E^2)^{-1/2} (v_1 - v_2) - JS}^2} 
\\
+ O \pbb{\frac{N}{M} R_4(M^{-\mu}) M^{-\xi}}\,.
\end{multline}
Recalling \eqref{cond_epsilon} and $\eta = M^{\rho}$, we find that the error term may be estimated by $\frac{N}{M} R_4(M^{-\mu}) M^{-\xi} \leq R_4(\eta) M^{-c_1}$ for some constant $c_1 > 0$. Going back to \eqref{Vmain_D0} and recalling \eqref{calV4D0} as well as the second inequality of \eqref{cond_mu_rho}, we find
\begin{equation*}
\cal V_{\rm{main}} \;=\;
\frac{8}{\pi^2\nu^2} \, \re \frac{1}{2 \pi} \int \dd v \, (\phi_1^\eta * \psi_2^\eta)(v) \, \tr \frac{S}{\pb{1 + \ii (1 - E^2)^{-1/2} v - JS}^2} + O \pbb{\frac{N}{M} (1 + R_4(\eta) M^{-c_1})}
\end{equation*}
for some constant $c_1 > 0$ (recall from \eqref{phipsi} that $\phi_1^\eta (E) = \psi_1^\eta(-E)$).
 Using Proposition \ref{prop:S_int} below, with
$e \deq \psi_1 * \phi_2$, $b = (1 - E^2)^{-1/2}$, and the observation that
\begin{equation*}
2 \re \int_0^\infty \dd t \, t^{1-d/2} \, \ol{\wh e(t)} \;=\; \int_\R \dd t \, \abs{t}^{1 - d/2} \, \ol {\wh \phi_1(t)} \, \wh \phi_2(t) \;=\; V_d(\phi_1, \phi_2)\,, \qquad 2 \, \ol {\wh e(0)} \;=\; 2 \, \ol {\wh \phi_1(0)} \, \wh \phi_2(0) \;=\; V_4(\phi_1, \phi_2)\,,
\end{equation*}
we find for $d \leq 3$ that
\begin{equation*}
\cal V_{\rm{main}} \;=\; \frac{4}{\pi^2\nu^2\sqrt{\det D}} \pbb{\frac{L}{2 \sqrt{\pi} W}}^d \pbb{\frac{2\eta}{\pi\nu}}^{d/2 - 2} V_d(\phi_1, \phi_2) + O \pbb{\frac{N}{M} (1 + R_4(\eta) M^{-c_1})}
\end{equation*}
and for $d = 4$ that
\begin{equation*}
\cal V_{\rm{main}} \;=\; \frac{4\abs{\log \eta}}{\pi^2\nu^2 \sqrt{\det D}} \pbb{\frac{L}{2 \sqrt{\pi} W}}^4 \, V_4(\phi_1, \phi_2)  + O \pbb{\frac{N}{M}}\,.
\end{equation*}
This concludes the proof of Proposition \ref{prop: leading term 2} (iii).
\end{proof}

The proof of the following result is given in Appendix \ref{appendix: S}.
\begin{proposition} \label{prop:S_int}
Suppose that \eqref{LW_assump} holds. Let $b > 0$ be fixed and $\cal J \deq 1 - M^{-c_2} \eta$ for some $c_2 > 0$.  Fix a smooth real function $e \in L^1(\R)$ satisfying the condition \textbf{(C2)} (see \eqref{non_Cauchy}), and recall the notation  $e^\eta(v) = \eta^{-1}e(\eta^{-1}v)$.
Then for $d \leq 3$ we have
\begin{equation} \label{intR3}
\frac{1}{2 \pi} \int \dd v \, e^\eta(v) \tr \frac{S}{\p{1 + \ii b v -\cal JS}^2} \;=\; \frac{(b \eta)^{d/2 - 2}}{\sqrt{\det D}} \pbb{\frac{L}{2 \sqrt{\pi} W}}^d \int_0^\infty \dd t \, t^{1-d/2} \, \ol{\wh e(t)} + O \pbb{\frac{N}{M} R_4(\eta) M^{-c_0}}
\end{equation}
for some constant $c_0 > 0$, and for $d = 4$ we have
\begin{equation} \label{intR4}
\frac{1}{2 \pi} \int \dd v \, e^\eta(v) \tr \frac{S}{\p{1 + \ii b v -\cal JS}^2} \;=\; \frac{\abs{\log \eta}}{\sqrt{\det D}} \pbb{\frac{L}{2 \sqrt{\pi} W}}^4 \, \ol{\wh e(0)}  + O \pbb{\frac{N}{M}}\,.
\end{equation}
\end{proposition}

\section{Extraction of the leading term and estimate of the error terms } \label{sec: part 2}

In this section we give the core of the proof: the estimate of the error in \eqref{main_prop_error}, hence completing the proof of Proposition \ref{prop: main}. Once Proposition \ref{prop: main} is proved, Theorems \ref{thm: main result}--\ref{thm: Theta 1} will follow easily (see Section~\ref{sec:conclusion_proof} below).
 We recall that, as in Section \ref{sec:3}, we set $\beta = 2$ throughout this section. How to modify the arguments for $\beta = 1$ is sketched in Section \ref{sec:sym}  below.

The basic strategy is to compute the expectation in \eqref{truncated series} by plugging \eqref{def: nb} into it and classifying all possible label configurations $(x_i)$ according to the partition induced by coincidences among the labels. After an appropriate resummation, we shall be able to identify the leading terms which give rise to $\cal V_{\rm{main}}$ and the error terms, which are to be estimated. In \cite{EK3}, this strategy was carried out under various simplifying assumptions, denoted by {\bf (S1)}--{\bf (S3)} there.
Roughly, these simplifications stated that among all possible partitions only the pairings matter, that coincidences of labels that are not imposed by the pairings may be neglected, and that the nonbacktracking condition from \eqref{def: nb} may be neglected beyond some fundamental restrictions it imposes on the class of admissible partitions.

In the following, we deal with all scenarios that were ignored under these simplifications.  In particular, we deal with
blocks of partitions that are larger than two. This requires to introduce more involved structures, which are accompanied by heavier notation. In addition, unlike the pairings from \cite{EK3}, partitions with larger blocks do not have an intuitive graphical representation in terms of bridges joining edges. 
Our analysis builds on that developed in \cite{EK3}; for the convenience of the reader,  we summarize some concepts from Sections 4.1 and 4.2 of \cite{EK3}
in the beginnings of Sections \ref{sec_41} and \ref{sec_42} respectively. 

Although the contributions of all exceptional scenarios are ultimately negligible, they cannot be estimated brutally by absolute value. The reason is the strongly oscillatory character of the expressions we have to estimate. Even if two summation labels coincide, and hence result in a gain in the form of a small prefactor, we cannot afford to estimate all remaining summations by the sum of absolute values of summands. Instead, we have to introduce a more involved, local, bookkeeping of various index coincidences, in which some parts of the summation are estimated by taking the absolute value inside while others are estimated by exploiting oscillations among the summands. Thus, the contribution of any ladder that is not affected by the simplifications of \cite{EK3} still has to be summed up explicitly (i.e.\ exploiting oscillations).

\subsection{Graphs and partitions of edges} \label{sec_41}

We have to compute $\wt F^\eta(E_1, E_2)$ defined in \eqref{truncated series}.
In order to express the nonbacktracking powers of $H$ in terms of the entries of $H$, it is convenient to index the two multiple summations arising from \eqref{def: nb} (when plugged into \eqref{truncated series}) using a graph. This graphical language was introduced in \cite[Section 4.1]{EK3}, and we summarize it here for the convenience of the reader. We introduce a directed graph $\cal C(n_1, n_2) \deq \cal C_1(n_1) \sqcup \cal C_2(n_2)$ defined as the disjoint union of a directed chain $\cal C_1(n_1)$ with $n_1$ edges and a directed chain $\cal C_2(n_2)$ with $n_2$ edges. Throughout the following, to simplify notation we shall often omit the arguments $n_1$ and $n_2$ from the graphs $\cal C$, $\cal C_1$, and $\cal C_2$. For an edge $e \in E(\cal C)$, we denote by $a(e)$ and $b(e)$ the initial and final vertices of $e$. Similarly, we denote by $a(\cal C_i)$ and $b(\cal C_i)$ the initial and final vertices of the chain $\cal C_i$. We call vertices of degree two \emph{black} and vertices of degree one \emph{white}. See Figure \ref{fig: cal C} for an illustration of $\cal C$ and for the convention of the orientation.

\begin{figure}[ht!]
\begin{center}
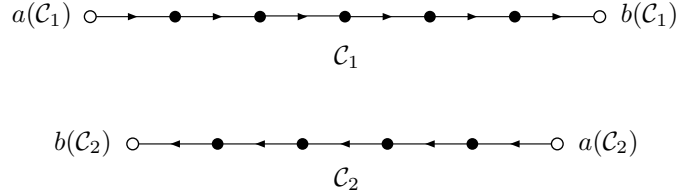
\end{center}
\caption{The graph $\cal C = \cal C_1 \sqcup \cal C_2$. Here we chose $n_1 = 6$ and $n_2 = 5$. We indicate the orientation of the chains $\cal C_1$ and $\cal C_2$ using arrows. In subsequent pictures, we systematically drop the arrows to avoid clutter, but we consistently use this orientation when drawing graphs. \label{fig: cal C}} 
\end{figure}

We assign a label $x_i \in \bb T$ to each vertex $i \in V(\cal C)$, and write $\f x = (x_i)_{i \in V(\cal C)}$.
For an edge $e \in E(\cal C)$ we define the associated pairs of ordered and unordered labels
\begin{equation*}
x_{e} \;\deq\; (x_{a(e)}, x_{b(e)}) \,, \qquad [x_e] \;\deq\; \{x_{a(e)}, x_{b(e)}\}\,.
\end{equation*}
Using the graph $\cal C = \cal C(n_1, n_2)$ we may now write the covariance
\begin{equation} \label{exp_HH_sumx}
\avgb{\tr H^{(n_1)}\,; \tr H^{(n_2)}} \;=\; \E \qb{\pb{\tr H^{(n_1)}}\, \pb{\tr H^{(n_2)}}} - \E \pb{\tr H^{(n_1)}} \E \pb{\tr H^{(n_2)}} \;=\; \sum_{\f x \in {\bb T}^{V(\cal C)}} I(\f x) A(\f x)\,,
\end{equation}
where we introduced
\begin{equation} \label{definition of A}
A(\f x) \;\deq\;
\E \pBB{\prod_{e \in E(\cal C)} H_{x_e}} - \E \pBB{\prod_{e \in E(\cal C_1)} H_{x_e}} \E \pBB{\prod_{e \in E(\cal C_2)} H_{x_e}}\,,
\end{equation}
and the indicator function
\begin{equation} \label{def_I_ind}
I(\f x) \;\deq\; I_0(\f x) \prod_{\substack{i,j \in V(\cal C) \col \\ \dist(i,j) = 2}} \ind{x_i \neq x_j}\,, \qquad
I_0(\f x) \;\deq\; \ind{x_{a(\cal C_1)} = x_{b(\cal C_1)}} \ind{x_{a(\cal C_2)} = x_{b(\cal C_2)}}\,.
\end{equation}
The indicator function $I_0(\f x)$ implements the fact that the final and initial vertices of each chain have the same label, while $I(\f x)$ in addition implements the nonbacktracking condition. When drawing $\cal C$ as in Figure \ref{fig: cal C}, we draw vertices of $\cal C$ with degree two using black dots, and vertices of $\cal C$ with degree one using white dots. The use of two different colours also reminds us that each black vertex $i$ gives rise to a nonbacktracking condition in $I(\f x)$, constraining the labels of the two neighbours of $i$ to be distinct.

In order to compute the expectation in \eqref{definition of A}, we decompose the label configurations $\f x$ according to partitions of $E(\cal C)$.  The following definition was given, in a reduced form, in \cite[Definition 4.2]{EK3}. 
\begin{definition} \label{def:partitions}
\begin{enumerate}
\item
We denote by $\fra P(U)$ for the set of partitions of a set $U$ and by $\fra M(U) \subset \fra P(U)$ the set of pairings (or matchings) of $U$.  (In the applications below the set $U$ will be either $E(\cal C)$ or $V(\cal C)$.) We call blocks of a pairing \emph{bridges}.
\item
We introduce the usual partial order $\leq$ on $\fra P(E(\cal C))$, where $\Pi \leq \Gamma$ means that $\Pi$ is a refinement of $\Gamma$.
\item
For a label configuration $\f x \in \bb T^{V(\cal C)}$ we define the partition $P(\f x) \in \fra P(E(\cal C))$ as the partition of $E(\cal C)$ generated by the equivalence relation $e \sim e'$ if and only if $[x_e] = [x_{e'}]$. Similarly, $P_o(\f x) \in \fra P(E(\cal C))$ is the partition 
of $E(\cal C)$ generated by the equivalence relation $e \sim e'$ if and only if $x_e = x_{e'}$. (Here the subscript ``o'' stands for ``ordered'').
\end{enumerate}
\end{definition}

Here ends the summary of the material from \cite[Section 4.1]{EK3}. Next, we introduce the new concept  of a \emph{halving partition}.
For a partition $\Xi \in \fra P(E(\cal C))$ and a subset $\gamma \subset E(\cal C)$ we define the restriction
\begin{equation*}
\Xi |_\gamma \;\deq\; \hb{\xi \cap \gamma \col \xi \in \Xi \,,\, \xi \cap \gamma \neq \emptyset}\,.
\end{equation*}
Moreover, for $\gamma \subset E(\cal C)$ and $\Xi \in \fra P(E(\cal C))$ we define
\begin{equation} \label{def_mu_Xi}
\mu_\gamma(\Xi) \;\deq\; \indb{\Xi |_\gamma \text{ consists of two blocks of equal size}}\,.
\end{equation}
We also set $\mu_\emptyset(\Xi) \deq 1$.
For a partition $\Gamma \in \fra P(E(\cal C))$ we define the subset
\begin{equation} \label{def_fra_H}
\fra H (\Gamma) \;\deq\; \hb{\Xi \leq \Gamma \col \mu_\gamma(\Xi) = 1 \text{ for each } \gamma \in \Gamma}\,.
\end{equation}
Thus, $\fra H(\Gamma)$ is the subset of partitions $\Xi$ that refine each block of $\Gamma$ into two pieces of equal size.
We call any $\Xi \in\fra H(\Gamma)$ a \emph{halving partition of $\Gamma$}.  If $\Gamma$ is a pairing, $\Xi \in \fra H(\Gamma)$ is simply the atomic partition.

Armed with these definitions, we return to the computation of \eqref{definition of A} to be plugged into \eqref{exp_HH_sumx}.
We first focus on the first term of \eqref{definition of A}. The idea is to partition all edges $e \in E(\cal C)$ first using the unordered labels $[x_e]$, yielding a partition $\Gamma$, and second using the ordered labels $x_e$, yielding a finer partition $\Xi \in \fra H(\Gamma)$. 
In other words, the blocks of $\Gamma$ collect those edges that have the same unordered labels, while
each such block is further subdivided into two smaller blocks according to the two possible
orderings of the same unordered label. We express this constraint on $\f x$ using the indicator function 
\begin{equation}\label{def:B}
B_{\Gamma, \Xi}(\f x) \;\deq\; \ind{P(\f x) = \Gamma} \, \ind{P_o(\f x) = \Xi}\,.
\end{equation}
Notice that the partitions $\Gamma$ and $\Xi$ yield a nonzero contribution only if each block of $\Gamma$ is subdivided by $\Xi$ into two blocks of equal size,  because $ \E A_{xy}^k A_{yx}^l = 0$ unless $k=l$ (recall from \eqref{HA} that $H_{xy} = \sqrt{S_{xy}} A_{xy}$).
This justifies the restriction $\Xi\in \fra H(\Gamma)$.
Thus, we write
\begin{align}
\E \prod_{e \in E(\cal C)} H_{x_e}
&\;=\; \sum_{\Gamma \in \fra P(E(\cal C))} \ind{P(\f x) = \Gamma} \prod_{\gamma \in \Gamma} \E \prod_{e \in \gamma} H_{x_e}
\notag \\ \label{A1}
&\;=\; \sum_{\Gamma \in \fra P(E(\cal C))} \sum_{\Xi \in \fra H(\Gamma)}  B_{\Gamma, \Xi}(\f x) \, \prod_{\gamma \in \Gamma} \pbb{\mu_\gamma(\Xi) \prod_{e \in \gamma} \sqrt{S_{x_e}}}\,,
\end{align}
Here in the first step we used that $H_{x_e}$ and $H_{x_{e'}}$ are independent if $[x_e] \neq [x_{e'}]$, and in the second step that $\E \prod_{e \in \gamma} H_{x_e}$ vanishes unless the partition $P_o(\f x) \vert_\gamma$ consists of two blocks of equal size.

Similarly, we find for the second term of \eqref{definition of A} that
\begin{equation} \label{A2}
\E \pBB{\prod_{e \in E(\cal C_1)} H_{x_e}} \E \pBB{\prod_{e \in E(\cal C_2)} H_{x_e}}
\;=\; \sum_{\Gamma \in \fra P(E(\cal C))} \sum_{\Xi \in \fra H(\Gamma)} B_{\Gamma, \Xi}(\f x) \, \prod_{\gamma \in \Gamma} \pbb{\mu_{\gamma \cap E(\cal C_1)}(\Xi) \mu_{\gamma \cap E(\cal C_2)}(\Xi) \prod_{e \in \gamma} \sqrt{S_{x_e}}}\,.
\end{equation}
In order to express $A(\f x)$ (see \eqref{definition of A}) we subtract \eqref{A2} from \eqref{A1}, which yields
\begin{equation} \label{A3}
A(\f x) \;=\; \prod_{e \in E(\cal C)} \sqrt{S_{x_e}} \sum_{\Gamma \in \fra P_c(E(\cal C))}
\sum_{\Xi \in \fra H(\Gamma)} B_{\Gamma, \Xi}(\f x) \, D(\Gamma, \Xi)\,,
\end{equation}
where we defined
\begin{equation} \label{def_D_ind}
D(\Gamma, \Xi) \;\deq\; \prod_{\gamma \in \Gamma} \mu_\gamma(\Xi) - \prod_{\gamma \in \Gamma} \pbb{\mu_{\gamma \cap E(\cal C_1)}(\Xi) \mu_{\gamma \cap E(\cal C_2)}(\Xi)}
\end{equation}
as well as the set of \emph{connected partitions}
\begin{equation} \label{conn_part}
\fra P_c(E(\cal C)) \;\deq\; \hb{\Gamma \in \fra P(E(\cal C)) \col \text{there is a $\gamma \in \Gamma$ such that } \gamma \cap E(\cal C_1) \neq \emptyset \text{ and } \gamma \cap E(\cal C_2) \neq \emptyset}\,.
\end{equation}
The restriction of the summation in \eqref{A3} from $\Gamma \in \fra P(E (\cal C))$ to $\Gamma \in \fra P_c(E (\cal C))$ follows from the observation that if $\Gamma$ is such that each $\gamma \in \Gamma$ satisfies $\gamma \subset E(\cal C_1)$ or $\gamma \subset E(\cal C_2)$, then $D(\Gamma,\Xi) = 0$ for all $\Xi \in \fra H(\Gamma)$.
We also record that $D(\Gamma, \Xi)$ is either 0 or 1. In analogy to \eqref{conn_part}, we also define the subset $\fra M_c(E(\cal C)) \deq \fra P_c(E(\cal C)) \cap \fra M(E(\cal C))$ of connected pairings.

\subsection{The refining pairing of a partition} \label{sec:refining_partition}
Next, we break up larger blocks of the partition $\Gamma \in \fra P(E(\cal C))$ into pairings.  The blocks of $\Gamma$ were
defined by the label coincidences, which remain unchanged by the breaking up of $\Gamma$; hence this breaking up of $\Gamma$ may seem artificial. It is however a very convenient technical device, since it allows us to reduce the estimates on partitions with arbitrarily large blocks to estimates on pairings. In particular, it allows us to use the machinery developed in \cite{EK3} to control the oscillations. 

Thus, given a  partition $\Gamma \in \fra P(E(\cal C))$ and one of its halving partitions $\Xi\in \fra H(\Gamma)$,
we introduce a rule for breaking up  $\Gamma$
 into a \emph{refining pairing} $\Pi = \Phi(\Gamma, \Xi) \in \fra M(E(\cal C))$. Although there is much arbitrariness in the choice of such a pairing, we define it precisely so as to ensure that, apart from the obvious condition $\Pi\leq\Gamma$, it fulfils the three  following properties which will be important for the rest of the argument.
\begin{enumerate}
\item[(a)]  The two edges of each bridge of $\Pi$ belong
 to different blocks of $\Xi$.
\item[(b)] If $\Gamma \in \fra P_c(E(\cal C))$ is a connected partition and $\Xi \in \fra H(\Gamma)$ then $\Phi(\Gamma, \Xi) \in \fra M_c(E(\cal C))$ is a connected pairing. In other words, connectedness is maintained after
the refinement.
\item[(c)] How a block $\gamma \in \Gamma$ is broken up into pairings does not depend on the other blocks of $\Gamma$.
\end{enumerate}
The refinement  operation $\Phi$  can be easily defined using a
greedy algorithm that successively  breaks up large blocks of $\Gamma$ into bridges,
 such that these three properties are satisfied at each step. 
 In particular, when constructing $\Pi$ from $\Gamma$ and $\Xi$, 
we always break up blocks of $\Gamma$ into smaller blocks in such 
a way that the restriction of $\Xi$ to these smaller blocks is 
again a halving partition.

The precise definition of the operation $\Phi$ is the following.
We introduce a total order on the edges $E(\cal C)$ as follows. The edges of $\cal C_1$ are increasing from left to right (see Figure \ref{fig: cal C}); the edges of $\cal C_2$ are increasing from right to left; any edge of $\cal C_1$ is smaller than any edge of $\cal C_2$. Now suppose that $\Xi \in \fra H(\Gamma)$. We construct $\Pi = \Phi(\Gamma, \Xi)$ recursively as follows. Set $ \Gamma_0  \deq \Gamma$ and $k \deq 0$.
\begin{enumerate}
\item
If $\Gamma_k$ is a pairing then set $\Pi \deq \Gamma_k$ and stop the recursion. Otherwise let $\gamma \in \Gamma_k$ satisfy $\abs{\gamma} > 2$.
\item
Let $e$ be the first edge of $\gamma$ and $e'$ the last edge of $\gamma$ that does not belong to the same block of $\Xi$ as $e$. Set $\Gamma_{k+1}$ to be $\Gamma_k$ in which the block $\gamma$ has been split into the pieces $\{e,e'\}$ and $\gamma \setminus \{e,e'\}$.
\item
Increment $k$ by one and go to step (i).
\end{enumerate}
It is immediate that this algorithm terminates after a finite number of steps and that $\Pi = \Phi(\Gamma, \Xi)$ is a pairing. In fact, $\Phi(\Gamma,\Pi) \in \fra M_c(E(\cal C))$ is a connected pairing. To see this, note that there exists a $\gamma \in \Gamma$ 
such that $\gamma \cap E(\cal C_1) \neq \emptyset$ and $\gamma \cap E(\cal C_2) \neq \emptyset$. It is not hard to see that in at least one step (ii) of the algorithm operating on $\gamma$ we have $e \in E(\cal C_1)$ and $e' \in E(\cal C_2)$, so that the bridge $\{e,e'\} \in \Phi (\Gamma,\Xi)$ connects the two components of $\cal C$.

Hence we may plug the trivial identity $1 = \sum_{\Pi \in \fra M_c(E(\cal C))} \indb{\Phi(\Gamma, \Xi) = \Pi}$ into \eqref{A3} and use \eqref{exp_HH_sumx} to get
\begin{equation} \label{A4}
\avgb{\tr H^{(n_1)}\,; \tr H^{(n_2)}}
\;=\; \sum_{\Pi \in \fra M_c(E(\cal C))} \sum_{\Gamma \geq \Pi} \sum_{\Xi \in \fra H(\Gamma)} \indb{\Phi(\Gamma, \Xi) = \Pi} \, D(\Gamma, \Xi) \sum_{\f x \in \bb T^{V(\cal C)}} I(\f x) \, B_{\Gamma,\Xi}(\f x) \, \prod_{\{e,e'\} \in \Pi} S_{x_e}\,.
\end{equation}
 (Recall that here $\cal C \equiv \cal C(n_1, n_2)$.)
It is convenient to introduce the set of all connected pairings,
\begin{equation*}
\fra M_c \;\deq\; \bigsqcup_{\substack{n_1, n_2 \geq 0 \col \\ n_1 + n_2 \text{ even}}} \fra M_c \pb{E(\cal C(n_1, n_2))}\,,
\end{equation*}
with which we associate the following definitions.
\begin{definition} \label{def:C_Gamma}
With each pairing $\Gamma \in \fra M_c$ we associate its underlying graph $\cal C(\Gamma)$, and regard $n_1$ and $n_2$ as functions on $\fra M_c$ in self-explanatory notation. We also frequently abbreviate $V(\Gamma) \equiv V(\cal C(\Gamma))$, and refer to $V(\Gamma)$ as the vertices of $\Gamma$.
\end{definition}
 Next, suppose that $\Xi \in \fra H(\Gamma)$ and $\Pi = \Phi(\Gamma, \Xi)$. Then $B_{\Gamma, \Xi}(\f x) = 1$ implies $\prod_{\pi \in \Pi} J_{\pi}(\f x) = 1$,
where we defined the indicator function
\begin{equation} \label{def_I_sigma}
J_{\{e,e'\}}(\f x) \;\deq\; \ind{[x_e] = [x_{e'}]} \ind{x_e \neq x_{e'}} \;=\; \ind{x_{a(e)} = x_{b(e')}} \ind{x_{a(e')} = x_{b(e)}}\,.
\end{equation}
To understand this implication, we first recall, from the definition of $B$
in \eqref{def:B}, that $B_{\Gamma, \Xi}(\f x) = 1$ ensures that the edges 
receive the same unordered labels within each block of $\Gamma$
and the halving partition $\Xi$ divides each block  in two 
according to whether the labels are the same as ordered labels.
In particular, if $\Gamma$ is a pairing, i.e.\ if $\Xi$ is atomic
and $\Gamma =\Pi$, then $J_\pi(\f x)=1$ directly follows for 
each pair $\pi\in \Gamma$.  If $\Gamma$ has larger blocks,
then the definition of $\Phi$ guarantees
that every  bridge in $\Pi= \Phi(\Gamma, \Xi)$ is halved by
$\Xi$, i.e.\ their labels, under $B_{\Gamma, \Xi}(\f x) = 1$, are the same as unordered but not as ordered labels.

Therefore we may restrict the summation over $\Pi$ in \eqref{A4} to the set
\begin{equation*}
\fra R \;\deq\; \hbb{\Pi \in \fra M_c \,\col\, \text{there is an $\f x \in \bb T^{V(\Pi)}$ such that }  I(\f x) \prod_{\pi \in \Pi} J_{\pi}(\f x) \neq 0}\,.
\end{equation*}
This set of pairings was also given in \cite[Equation (4.26)]{EK3}, using a more combinatorial definition which we shall not need here. To interpret the set $\fra R$, observe that the right-hand side of \eqref{def_I_sigma} induces a series of constraints among the labels $\f x$; in $\fra R$ we simply require that these constraints hold and be compatible with the nonbacktracking condition from \eqref{def_I_ind}, which requires certain labels to be distinct. As explained in \cite[Sections 4.1 and 4.2]{EK3}, the restriction from $\fra M_c$ to $\fra R$ is significant, and plays an essential role in estimating the contribution of pairings.

Thus we write \eqref{A4} as
\begin{multline} \label{A5}
\avgb{\tr H^{(n_1)}\,; \tr H^{(n_2)}}
\\
=\; \sum_{\Pi \in \fra R} \ind{n_1(\Pi) = n_1} \ind{n_2(\Pi) = n_2} \sum_{\Gamma \geq \Pi} \sum_{\Xi \in \fra H(\Gamma)} \indb{\Phi(\Gamma, \Xi) = \Pi} \, D(\Gamma, \Xi) \sum_{\f x \in \bb T^{V(\cal C)}} I(\f x) \, B_{\Gamma,\Xi}(\f x) \, \prod_{\{e,e'\} \in \Pi} S_{x_e}\,.
\end{multline}
This is the appropriate decomposition of the left-hand side in terms of pairings $\Pi$ of all possible graphs $\cal C$. It is the correct analogue of \cite[Equation (4.27)]{EK3} without making any simplifications.

\subsection{Skeletons} \label{sec_42} 
The summation in \eqref{truncated series} is highly oscillatory,
which requires a careful resummation of graphs of different order. 
We perform a local resummation procedure of the so-called \emph{ladder} subdiagrams,
which are subdiagrams with a pairing structure that consists only of parallel bridges.
This is the second resummation procedure mentioned in the introduction.
Concretely, we regroup pairings $\Pi$ into families that have a similar structure, differing only in the number of parallel bridges per ladder  subdiagram. Their common structure is represented by the simplest element of the family, the \emph{skeleton},  whose ladders consist of a single bridge. 
The concept of skeleton pairing was introduced in \cite[Section 4.2]{EK3}. Here we merely recall the main ideas for the convenience of the reader, and refer to \cite[Section 4.2]{EK3} for full details.

The basic idea is to construct the \emph{skeleton} $\Sigma$ of a pairing $\Pi \in \fra M_c$ 
by collapsing parallel bridges of $\Pi$. By definition, the bridges $\{e_1, e_1'\}$ and $\{e_2, e_2'\}$ are \emph{parallel} if $b(e_1) = a(e_2)$ and $b(e_2') = a(e_1')$. With each $\Pi \in \fra M_c$ we associate a couple $(\Sigma, \f b)$, where $\Sigma \in \fra M_c$ has no parallel bridges, and $\f b = (b_\sigma)_{\sigma \in \Sigma} \in \N^\Sigma$. The pairing $\Sigma$ is obtained from $\Pi$ by successively collapsing parallel bridges until no parallel bridges remain. The integer $b_\sigma$ denotes the number of parallel bridges of $\Pi$ that were collapsed into the bridge $\sigma$. Conversely, for any given couple $(\Sigma, \f b)$, where $\Sigma \in \fra M_c$ has no parallel bridges and $\f b \in \N^{\Sigma}$, we define $\Pi = \cal G(\Sigma, \f b)$ as the pairing obtained from $\Sigma$ by replacing, for each $\sigma \in \Sigma$, the bridge $\sigma$ with $b_\sigma$ parallel bridges. Thus we have a one-to-one correspondence between pairings $\Pi$ and couples $(\Sigma, \f b)$.
Instead of burdening the reader with formal definitions of this correspondence,  we refer to Figure \ref{fig: skeleton} for an illustration. When no confusion is possible, in order to streamline notation we shall identify $\Pi$ with $(\Sigma, \f b)$.
\begin{figure}[ht!]
\begin{center}
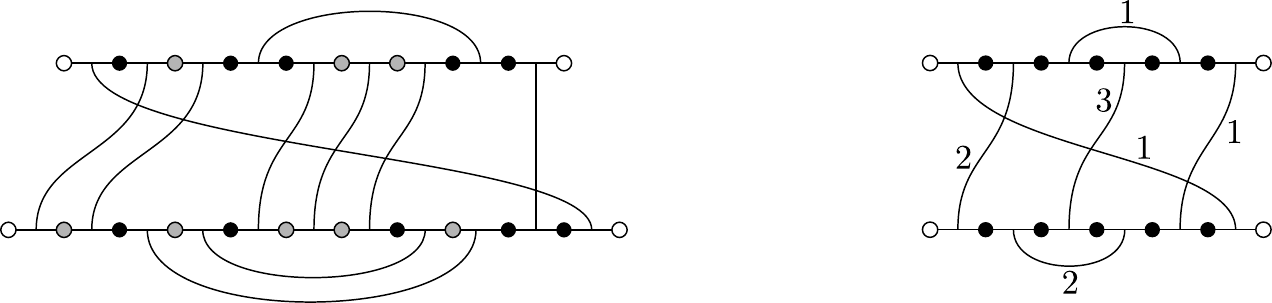
\end{center}
\caption{A pairing $\Pi$ (left) and its skeleton $(\Sigma, \f b)$ (right). As in \cite[Section 4.1]{EK3}, we draw a pairing $\Pi$ by drawing a line connecting the edges $e$ and $e'$ for each bridge $\{e,e'\} \in \Pi$. Next to each skeleton bridge $\sigma \in \Sigma$ we indicate the multiplicity $b_\sigma$ describing how many bridges of $\Pi$ were collapsed into $\sigma$.
We draw the ladder vertices $V_l(\Pi)$ in grey and the vertices $V_s(\Pi)$ are drawn in black or white (see Definition \ref{def:ladders}).
\label{fig: skeleton}} 
\end{figure}

\begin{definition} \label{def:ladders}
Fix $\Sigma \in \fra M_c$ and  $\f b\in  \N^\Sigma$. As above, abbreviate $\Pi \deq \cal G(\Sigma, \f b)$.
\begin{enumerate}
\item
For $\sigma \in \Sigma$ we introduce the \emph{ladder encoded by $\sigma$}, denoted by $L_\sigma(\Sigma, \f b) \subset \Pi$ and defined as the set of bridges of $\Pi$ that are collapsed into the skeleton bridge $\sigma$. Note that $L_\sigma(\Sigma,\f b)$ consists of $\abs{L_\sigma(\Sigma,\f b)} = b_\sigma$ parallel bridges.

\item
We say that a vertex $i \in V(\Pi)$ \emph{touches} the bridge $\{e,e'\} \in \Pi$ if $i$ is incident to $e$ or $e'$. We call a vertex $i$ a \emph{ladder vertex} of $L_\sigma(\Sigma, \f b)$ if it touches two bridges of $L_\sigma(\Sigma, \f b)$.
Note that a ladder consisting of $b$ parallel bridges gives rise to $2(b-1)$ ladder vertices.

\item
We say that $i \in V(\Pi)$ is a \emph{ladder vertex} of $\Pi$ if it is a ladder vertex of $L_\sigma(\Sigma, \f b)$ for some $\sigma \in \Sigma$. We decompose the vertices $V(\Pi) = V_s(\Pi) \sqcup V_l(\Pi)$, where $V_l(\Pi)$ denotes the set of ladder vertices of $\Pi$.
\end{enumerate}
\end{definition}
See Figure \ref{fig: skeleton} for an illustration of Definition \ref{def:ladders}.

Next, for a pairing $\Sigma \in \fra M_c$ we define the set of admissible multiplicities
\begin{equation} \label{def_B_Sigma}
B(\Sigma) \;\deq\; \hb{\f b \in \N^\Sigma \col \cal G(\Sigma, \f b) \in \fra R}\,.
\end{equation}
We define the set of admissible skeletons as
\begin{equation*}
\fra S \;\deq\; \h{\Sigma \in \fra M_c \col B(\Sigma) \neq \emptyset}\,.
\end{equation*}
Due to the nonbacktracking condition and the requirement that parallel bridges are collapsed, not every pairing is an admissible skeleton, and not every family of multiplicities $\f b$ of a skeleton $\Sigma \in \fra S$ is admissible. A combinatorial characterization of $\fra S$ is given in \cite[Lemma 4.6]{EK3}, which we shall however not need here.
All that we shall need about $B(\Sigma)$ is the following simple result, proved in \cite[Lemma 4.6]{EK3}.

\begin{lemma} \label{lem: skeletons}
For any $\Sigma \in \fra S$ the set $\N^{\Sigma} \setminus B(\Sigma)$ is finite.
\end{lemma}

Here ends the summary of the material from \cite[Section 4.2]{EK3}. 
We may now obtain the desired decomposition of $\wt F^\eta(E_1, E_2)$ as a sum over skeletons.
Plugging \eqref{A5} into \eqref{truncated series} yields
\begin{align}
\wt F^\eta(E_1,E_2) &\;=\;
\sum_{\Pi \in \fra R} \ind{2 \abs{\Pi} \leq M^\mu} \, 2 \re \pb{\wt \gamma_{n_1(\Pi)}(E_1,\phi_1)} \, 2 \re \pb{\wt \gamma_{n_2(\Pi)}(E_2, \phi_2)}
\notag \\ 
&\qquad \times \sum_{\Gamma \geq \Pi} \sum_{\Xi \in \fra H(\Gamma)} \indb{\Phi(\Gamma, \Xi) = \Pi} \, D(\Gamma, \Xi) \, \sum_{\f x \in \bb T^{V(\Pi)}} I(\f x) \, B_{\Gamma,\Xi}(\f x) \, \prod_{\{e,e'\} \in \Pi} S_{x_e}
\notag \\ \label{Fskel}
&\;=\;
\sum_{\Sigma \in \fra S} \wt{\cal V}(\Sigma) \,,
\end{align}
where we defined the \emph{value of the skeleton} $\Sigma \in \fra S$ as
\begin{multline} \label{def_V_general}
\wt{\cal V}(\Sigma) \;\deq\; \sum_{\f b \in B(\Sigma)} \indBB{2 \sum_{\sigma \in \Sigma} b_\sigma \leq M^\mu} \, 2 \re \pb{\wt \gamma_{n_1(\Sigma, \f b)}(E_1,\phi_1)} \, 2 \re \pb{\wt \gamma_{n_2(\Sigma, \f b)}(E_2, \phi_2)}
\\
\times \sum_{\Gamma \geq \cal G(\Sigma, \f b)} \sum_{\Xi \in \fra H(\Gamma)} \indb{\Phi(\Gamma, \Xi) = \cal G(\Sigma, \f b)} \, D(\Gamma, \Xi) \, \sum_{\f x \in \bb T^{V(\Sigma, \f b)}} I(\f x) \, B_{\Gamma,\Xi}(\f x) \, \prod_{\{e,e'\} \in \cal G(\Sigma, \f b)} S_{x_e}\,.
\end{multline}
This definition is the correct generalization of $\cal V(\Sigma)$ in \cite[Equation (4.31)]{EK3} without assuming any simplifications. In particular, $\cal V(\Sigma)$ and $\wt{\cal V}(\Sigma)$  differ by a term $\cal E$ in the terminology of \cite{EK3}.

\subsection{Classification of skeletons} \label{sec:classification_skeletons}

In this short subsection we recall the splitting, introduced in \cite[Section 4.4]{EK3}, of the skeletons $\Sigma \in \fra S$ into three classes: the dumbbell skeletons, the small error skeletons, and the large error skeletons. The \emph{dumbbell skeletons} are the eight simplest skeletons, denoted by $D_1, \dots, D_8$, whose contribution is of leading order and yields $\cal V_{\rm{main}}$ from \eqref{VDdef}. They are defined in Figure \ref{fig: dumbbell}.
\begin{figure}[ht!]
\begin{center}
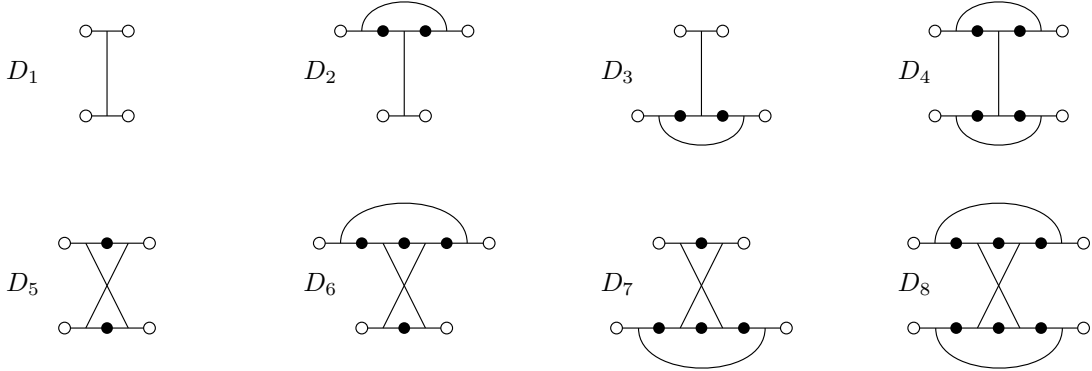
\end{center}
\caption{The eight dumbbell skeletons $D_1, \dots, D_8$. \label{fig: dumbbell}} 
\end{figure}

The contributions of the other skeletons has to be estimated and shown
to be much smaller. As pointed out in \cite[Section 4.4]{EK3},
 it turns out that when estimating $\wt {\cal V}(\Sigma)$ we are faced with two independent difficulties. First, strong oscillations  in the
$\f b$-summations in the definition of  
$\wt {\cal V}(\Sigma)$  \eqref{def_V_general} give rise to cancellations which have to be exploited carefully. Second, due to the combinatorial complexity of the skeletons, the size of $\fra S$ grows exponentially with $M$, which means that we have to deal with combinatorial estimates. It turns out that these two difficulties may be effectively decoupled: if $\abs{\Sigma}$ is small then only the first difficulty matters, and if $\abs{\Sigma}$ is large then only the second one matters. (As usual, $\abs{\Sigma}$ denotes the cardinality of $\Sigma$, i.e.\ the number of bridges in $\Sigma$.) Hence we split the set $\fra S$ in the dumbbell skeletons, the small error skeletons, and the large error skeletons. The splitting into small and large skeletons is done using a fixed large cutoff $K \in \N$. In other words, we write
\begin{equation*}
\fra S \;=\; \fra S_D \sqcup \fra S_K^\leq \sqcup \fra S_K^>\,,
\end{equation*}
where we abbreviated $\fra S_D \deq \{D_1, \dots, D_8\}$ for the set of dumbbell skeletons, and defined the set of small skeletons ${\fra S}_K^\leq \deq \h{\Sigma \in \fra S \setminus \fra S_D \col \abs{\Sigma} \leq K}$ 
 as well as the set of large skeletons ${\fra S}_K^> 
\deq \h{\Sigma \in \fra S \setminus \fra S_D \col \abs{\Sigma} > K}$. 
The constant $K$ will be chosen large enough in Proposition \ref{prop:large_Sigma_general} below.

\subsection{Large skeletons}
We first estimate the contribution of the large skeletons $\fra S^>_K$. As in \cite[Section 4.4]{EK3}, the estimate of the large skeletons is rather simple, since for large enough $K$ their contribution is small even if we estimate them by taking the absolute value inside the summations in \eqref{def_V_general}. Hence, the oscillatory effects present in the prefactors $\wt \gamma$ are not exploited.
 The precise estimate is the following, which is analogous to \cite[Proposition 4.8]{EK3}.

\begin{proposition} \label{prop:large_Sigma_general}
For large enough $K$, depending on $\mu$, we have
\begin{equation} \label{result_large skeletons}
\sum_{\Sigma \in \fra S_K^>} \abs{\wt{\cal V}(\Sigma)} \;\leq\; C_K N M^{-2}\,.
\end{equation}
\end{proposition}

\begin{proof}
Recall the definition \eqref{def_I_sigma} of $J_\pi(\f x)$ for a bridge $\pi = \{e,e'\}$, which encodes the constraints of coincidences of the labels associated with the vertices $a(e)$, $b(e)$, $a(e')$, and $b(e')$. The basic estimate is
\begin{equation} \label{estimate for breaking lumps}
\sum_{\Gamma \geq \cal G(\Sigma, \f b)} \sum_{\Xi \in \fra H(\Gamma)} \indb{\Phi(\Gamma, \Xi) = \cal G(\Sigma, \f b)} \, B_{\Gamma,\Xi}(\f x) \;\leq\; \prod_{\pi \in \cal G(\Sigma, \f b)} J_\pi(\f x)
\end{equation}
for any $\Sigma \in \fra S$, $\f b \in \N^\Sigma$, and $\f x \in \bb T^{V(\Sigma, \f b)}$. In order to prove \eqref{estimate for breaking lumps}, we make the two following observations: (i) each side of $\eqref{estimate for breaking lumps}$ is either zero or one, and (ii) if the right-hand side of \eqref{estimate for breaking lumps} is zero then so is its left-hand side. Observation (i) follows form the trivial fact that for any $\f x$ there are unique partitions $\Gamma, \Xi \in \fra P(E(\cal C))$ such that $B_{\Gamma, \Xi}(\f x) = 1$. Observation (ii) follows form the fact that if $\pi \in \Phi(\Gamma, \Xi)$ then $B_{\Gamma,\Xi}(\f x) = B_{\Gamma,\Xi}(\f x) J_\pi(\f x)$, since, by definition of $\Phi$, the event $J_\pi(\f x) = 1$ is a consequence of the event $B_{\Gamma, \Xi}(\f x) = 1$. This concludes the proof of \eqref{estimate for breaking lumps}

For the following we may drop the nonbacktracking condition encoded by $I$
(see \eqref{def_I_ind} for the definition), but we still keep the condition on the labels imposed by
two traces and implemented by $I_0$. 
From \eqref{def_V_general} and \eqref{estimate for breaking lumps} we therefore get, using \eqref{bound on gamma} and the trivial bound $D(\Gamma,\Xi) \leq 1$, that
\begin{equation*}
\abs{\wt{\cal V}(\Sigma)} \;\leq\; C \sum_{\f b \in \N^\Sigma} \indBB{2 \sum_{\sigma \in \Sigma} b_\sigma \leq M^\mu} \sum_{\f x \in \bb T^{V(\Sigma, \f b)}} I_0(\f x)  \prod_{\{e,e'\} \in \cal G(\Sigma, \f b)} J_{\{e,e'\}}(\f x) \, S_{x_e}\,.
\end{equation*}
Next, we may perform the summation over all labels of $\f x$ indexed by the ladder vertices of $\cal G(\Sigma, \f b)$ (see Definition \ref{def:ladders}). After this summation, each ladder $L_\sigma(\Sigma, \f b)$ of $\cal G(\Sigma, \f b)$ is replaced with a single bridge that encodes an entry of $S^{b_\sigma}$ instead of $S$. Thus we get the bound
\begin{equation}\label{trivvest}
\abs{\wt{\cal V}(\Sigma)} \;\leq\; C \sum_{\f b \in \N^\Sigma} \indBB{2 \sum_{\sigma \in \Sigma} b_\sigma \leq M^\mu} \sum_{\f x \in \bb T^{V(\Sigma)}}  I_0(\f x) \prod_{\{e,e'\} \in \Sigma} J_{\{e,e'\}}(\f x) \, \pb{S^{b_{\{e,e'\}}}}_{x_e}\,.
\end{equation}
See \cite[Section 4.2]{EK3} for the full details of the summation over the labels of the ladder vertices. The right-hand side of \eqref{trivvest} is equal to \cite[Equation (4.42)]{EK3}, which was estimated in \cite[Section 4.4]{EK3}. The result from \cite[Section 4.4]{EK3} is \eqref{result_large skeletons}. This concludes the proof of Proposition \ref{prop:large_Sigma_general}. 
\end{proof}

\subsection{Small skeletons} \label{sec_53}

For the following we fix $K$ to be the constant from Proposition \ref{prop:large_Sigma_general}. In order to handle the small skeletons in $\fra S_K^\leq \cup \fra S_D$, we have to exploit carefully the oscillations in \eqref{def_V_general}. Since the set of small skeletons $\fra S_K^\leq \cup \fra S_D$ is finite and  independent of $M$, it suffices to compute (in the case of $\fra S_D$) or estimate (in the case of $\fra S_K^\leq$) the contribution of each such skeleton individually. The details of the following estimates will be somewhat different for the two cases \textbf{(C1)} and \textbf{(C2)}; for definiteness, we focus on the (harder) case \textbf{(C2)}, i.e.\ we assume that $\phi_1$ and $\phi_2$ both satisfy \eqref{non_Cauchy}. The case \textbf{(C1)} is handled using a similar argument whose details we omit.

Before proceeding further, we want to replace the summation over $\f b \in B(\Sigma)$ with $\f b \in \N^\Sigma$ (recall that $B(\Sigma)$ from \eqref{def_B_Sigma} encodes a restriction on the ladder multiplicities imposed by the nonbacktracking condition). To this end, we split $\wt{\cal V}(\Sigma) = \wt{\cal V}_0(\Sigma) - \wt{\cal V}_1(\Sigma)$ arising from the splitting $\ind{\f b \in B(\Sigma)} = 1 - \ind{\f b \notin B(\Sigma)}$ plugged into \eqref{def_V_general}.  The goal of this subsection is to prove the following result.

\begin{proposition} \label{prop:V_general_estimate}
Suppose that $\phi_1$ and $\phi_2$ satisfy \eqref{non_Cauchy}. Suppose moreover that \eqref{D leq kappa} holds for some small enough $c_* > 0$. Then for any fixed $K \in \N$ and small enough $\delta > 0$ in Proposition \ref{prop: expansion with trunction} there exists a constant $c_0 > 0$ such that 
\begin{equation} \label{V_general_estimate}
\abs{\wt{\cal V}(\Sigma)} \;\leq\; \abs{\wt{\cal V}_0(\Sigma)} + \abs{\wt{\cal V}_1(\Sigma)} \;\leq\; \frac{C_\Sigma N}{M} R_2(\omega + \eta)  M^{-c_0}\,.
\end{equation}
for all $\Sigma \in \fra S_K^{\leq}$.
\end{proposition}
This bound is much smaller than the true size
of the leading term,
\begin{equation*}
\abs{\cal V_{\rm main}} \;\asymp\; \frac{N}{M} R_4(\omega + \eta)\,,
\end{equation*}
which is obtained from the dumbbell skeletons computed in Section \ref{sec:dumbbell_C2} (see Proposition \ref{prop: main}).

The rest of this subsection is devoted to the proof of Proposition \ref{prop:V_general_estimate}. We focus on the main term, $\wt {\cal V}_0(\Sigma)$. The estimate of $\wt {\cal V}_1(\Sigma)$ is much easier and is sketched at the end of this subsection.
To guide the reader, we split the somewhat lengthy argument into eight steps. 

\subsubsection*{Step 1. Conditioning on ladders}
A key idea behind the estimate of $\wt{\cal V}_0(\Sigma)$ is that the ladders in $\cal G (\Sigma, \f b)$ that remain ladders in $\Gamma\geq \cal G (\Sigma, \f b)$ (see the summation over $\Gamma$ in 
\eqref{def_V_general})  will be summed up explicitly. (Recall
 that $L_\sigma(\Sigma, \f b)$ denotes the ladder in $\cal G (\Sigma, \f b)$ encoded by the skeleton bridge $\sigma \in \Sigma$ and consisting of $b_\sigma$ parallel bridges.)
To make this classification of ladders precise, we say that a 
partition $\Gamma$ (with $\Gamma \geq \cal G(\Sigma, \f b)$) \emph{disrupts} the ladder $L_\sigma(\Sigma, \f b)$ if there is a $\pi \in L_\sigma(\Sigma, \f b)$ such that $\pi \notin \Gamma$,  i.e.\ a bridge of the ladder is 
contained in a strictly larger block of $\Gamma$.
 We shall classify all partitions $\Gamma \geq \cal G(\Sigma, \f b)$ according to the set $\zeta \subset \Sigma$ of skeleton bridges whose ladders $\{L_\sigma(\Sigma, \f b) \col \sigma \in \zeta\}$ the partition $\Gamma$ disrupts. The other ladders $L_\sigma(\Sigma, \f b)$, where $\sigma$ lies the complementary set $\bar \zeta \deq \Sigma \setminus \zeta$, are not disrupted by $\Gamma$, i.e.\ for $\sigma \in \bar \zeta$ and $\pi \in L_\sigma(\Sigma, \f b)$ we have $\pi \in \Gamma$. We denote the set of partitions $\Gamma$ that disrupt precisely the ladders encoded by the skeleton bridges $\zeta$ by $\fra F_\zeta(\Sigma, \f b)$; explicitly,
\begin{equation*}
\fra F_\zeta(\Sigma, \f b) \;\deq\; \hb{\Gamma \geq \cal G(\Sigma, \f b) \col \text{$\Gamma$ disrupts $L_\sigma(\Sigma, \f b)$ for $\sigma \in \zeta$ and does not disrupt $L_\sigma(\Sigma, \f b)$ for $\sigma \in \bar \zeta$}}\,.
\end{equation*}
We shall sometimes refer to $\zeta$ as the set of \emph{disrupted skeleton bridges} and to the $\bar \zeta$ as the set of \emph{not disrupted skeleton bridges}.

Our strategy will be to sum carefully, by making use of oscillations, the labels associated with ladders that are not disrupted by $\Gamma$, while using a more brutal approach for the ladders that are disrupted by $\Gamma$. The latter summation will be estimated by taking the absolute value inside the summation. The resulting loss is compensated by the fact that the summation variables are subject to additional constraints owing to their being disrupted by $\Gamma$.
 As it turns out, these constraints lead to a reduction in summation that is sufficient to compensate the loss resulting from ignoring the phases of the summands. We note that an even more refined strategy could be applied, in which we would subdivide ladders disrupted by $\Gamma$ into subladders that are not disrupted by $\Gamma$. We would then sum these not disrupted subladders explicitly (making use of oscillations), and sum the remaining labels of the ladder more brutally, again making use of constraints imposed by the disrupting by $\Gamma$. As it turns out, however, such a refined approach is mercifully not needed, and we choose to characterize  a ladder as ``disrupted by $\Gamma$'' and to sum its labels brutally as soon as even one of its bridges is not contained in $\Gamma$.

For any given skeleton $\Sigma \in \fra S$, we split the summation over $\Gamma$ and $\Xi$ in \eqref{def_V_general} according to the set $\zeta$ of disrupted skeleton bridges: 
\begin{equation} \label{V_o_V_zeta}
\wt{\cal V}_0(\Sigma) \;=\; \sum_{\zeta \subset \Sigma} \cal V_\zeta(\Sigma)\,,
\end{equation}
where we defined
\begin{multline} \label{def_V_T}
\cal V_\zeta(\Sigma) \;\deq\; \sum_{\f b \in \N^\Sigma} \indBB{2 \sum_{\sigma \in \Sigma} b_\sigma \leq M^\mu} \, 2 \re \pb{\wt \gamma_{n_1(\Sigma, \f b)}(E_1,\phi_1)} \, 2 \re \pb{\wt \gamma_{n_2(\Sigma, \f b)}(E_2, \phi_2)}
\\
\times \sum_{\Gamma \in \fra F_\zeta(\Sigma, \f b)} \sum_{\Xi \in \fra H(\Gamma)} \indb{\Phi(\Gamma, \Xi) = \cal G(\Sigma, \f b)} \, D(\Gamma, \Xi) \,
\, \sum_{\f x \in \bb T^{V(\Sigma, \f b)}} I(\f x) \, B_{\Gamma,\Xi}(\f x) \, \prod_{\{e,e'\} \in \cal G(\Sigma, \f b)} S_{x_e}\,.
\end{multline}
(To unburden notation, we omit a tilde in the definition of $\cal V_\zeta$.) 

Next, in \eqref{def_V_T} we first replace $\wt \gamma_{n_i(\Sigma, \f b)}(E_i,\phi_i)$ with $(\psi_i^\eta * \gamma_{n_i(\Sigma, \f b)})(E_i)$ using \eqref{gamma - g}, and then introduce a cutoff in the tail of $\psi_i^\eta$ by replacing $\psi_i^\eta$ with $\psi_i^{\leq, \eta}$ from \eqref{phi_cutoff1} where $\theta \deq \delta$ (here we used the estimate \eqref{phi_geq_bound}), and by replacing the indicator function $\indb{2 \sum_{\sigma \in \Sigma} b_\sigma \leq M^\mu}$ with $\prod_{\sigma \in \Sigma} \ind{b_\sigma \leq M^\mu}$ using \eqref{phi_geq_bound}. We omit the details of these brutal estimates, which are similar to those
used to obtain \eqref{VD_step2} from \eqref{common expression for V(D) 2} and \eqref{VDdef}; 
the only additional complication is the summation over $\Gamma$ and $\Xi$, which may be easily dealt with using \eqref{estimate for breaking lumps}. Using the splitting \eqref{2re2re} (with $x_i = \wt \gamma_{n_i(\Sigma, \f b)}(E_i, \phi_i)$) we therefore get
\begin{equation} \label{V_split2}
\cal V_\zeta(\Sigma) \;=\; 2 \re \pb{\cal V_\zeta'(\Sigma) + \cal V_\zeta''(\Sigma)} + O_{q,\Sigma}(N M^{-q})\,,
\end{equation}
where we defined
\begin{multline} \label{def_V_T_prime}
\cal V_\zeta'(\Sigma) \;\deq\; \sum_{\f b \in \{1, \dots, [M^\mu]\}^\Sigma} \pb{\gamma_{n_1(\Sigma, \f b)} * \psi_1^{\leq, \eta}}(E_1) \, \pb{\ol {\gamma_{n_2(\Sigma, \f b)}} * \psi_2^{\leq, \eta}}(E_2)
\\
\times \sum_{\Gamma \in \fra F_\zeta(\Sigma, \f b)} \sum_{\Xi \in \fra H(\Gamma)} \indb{\Phi(\Gamma, \Xi) = \cal G(\Sigma, \f b)} \, D(\Gamma, \Xi) \,
\, \sum_{\f x \in \bb T^{V(\Sigma, \f b)}} I(\f x) \, B_{\Gamma,\Xi}(\f x) \, \prod_{\{e,e'\} \in \cal G(\Sigma, \f b)} S_{x_e}\,,
\end{multline}
and $\cal V_\zeta''(\Sigma)$ is defined similarly but without the complex conjugation on $\gamma_{n_2(\Sigma, \f b)}$. We give the details of the estimate for the (harder) term $\cal V_\zeta'(\Sigma)$. Thus, throughout the following, we only consider $\cal V_\zeta'(\Sigma)$; the estimates for $\cal V_\zeta''(\Sigma)$ are almost identical, up to minor differences that are explained in Step 8 at the end of this subsection. 

Next, we use \eqref{claim about gamma n} to rewrite the factors $\gamma$.
To that end, we have to classify the bridges of $\Sigma$ into three classes according 
to the following definition, which is the same as Definition 4.13 in \cite{EK3}.
\begin{definition} \label{def_split_Sigma}
For $i = 1,2$ we define
\begin{equation*}
\Sigma_i \;\deq\; \hb{\sigma \in \Sigma \col \sigma \subset E(\cal C_i)}\,,
\end{equation*}
the set of bridges consisting only of edges of $\cal C_i$. We abbreviate $\Sigma_d \deq \Sigma_1 \cup \Sigma_2$ (the set of ``domestic bridges''). We also define $\Sigma_c \deq \Sigma \setminus \Sigma_d$, the set of bridges connecting the two components of $\cal C$.
\end{definition} 
Since each $\sigma \in \Sigma_c$ contains one edge of $\cal C_1$ and one edge of $\cal C_2$, and each $\sigma \in \Sigma_i$ contains two edges of $\cal C_i$, we find that the number of edges in the $i$-th chain $\cal C_i(n_i)$ of the graph
$\cal C(n_1, n_2)$ with pairing $\cal G(\Sigma, \f b)$ is
\begin{equation*}
n_i(\Sigma, \f b) \;=\; \sum_{\sigma \in \Sigma_c} b_\sigma + 2 \sum_{\sigma \in \Sigma_i} b_\sigma\,.
\end{equation*}
Recalling the notation \eqref{def:AA}, we therefore get
\begin{multline} \label{V_T_prime_2}
\cal V_\zeta'(\Sigma) \;=\; \sum_{\f b \in \{1, \dots, [M^\mu]\}^\Sigma}
\pbb{s \, \prod_{\sigma \in \Sigma} \chi_\sigma^{b_\sigma}} * \psi_1^{\leq, \eta}(E_1) * \psi_2^{\leq, \eta}(E_2)
\\
\times \sum_{\Gamma \in \fra F_\zeta(\Sigma, \f b)} \sum_{\Xi \in \fra H(\Gamma)} \indb{\Phi(\Gamma, \Xi) = \cal G(\Sigma, \f b)} \, D(\Gamma, \Xi) \,
\, \sum_{\f x \in \bb T^{V(\Sigma, \f b)}} I(\f x) \, B_{\Gamma,\Xi}(\f x) \, \prod_{\{e,e'\} \in \cal G(\Sigma, \f b)} S_{x_e}\,,
\end{multline}
where we defined the shorthand
\begin{equation*}
s \;\equiv\; s(E_1, E_2) \;\deq\; T(E_1) \ol{T(E_2)} \, \ee^{\ii (A_1 - A_2)}
\end{equation*}
and the phases
\begin{equation} \label{def_chi}
\chi_\sigma \;\equiv\; \chi_\sigma(E_1, E_2) \;\deq\;
\begin{cases}
-\ee^{2 \ii A_1} & \text{if } \sigma \in \Sigma_1
\\
-\ee^{-2 \ii A_2} & \text{if } \sigma \in \Sigma_2
\\
\ee^{\ii (A_1 - A_2)} & \text{if } \sigma \in \Sigma_c\,.
\end{cases}
\end{equation}
(On the first line of the right-hand side of \eqref{V_T_prime_2}, the
stars denote the convolutions with respect to the variables $E_1$ and $E_2$ on which $s$ and $\chi_\sigma$ depend.)
The precise form of the function $s$ is irrelevant; we only need the bound $\abs{s} \leq 5$.
From now on we drop the explicit mention of the domain $\{1, \dots, [M^\mu]\}^\Sigma$ of the $\f b$-summation in our expressions.

\subsubsection*{Step 2. Decoupling of the $\zeta$- and $\bar \zeta$-variables}
We fix $\Sigma \in \fra S_K^\leq$ and $\zeta \subset \Sigma$, where $\zeta$ is the set of the disrupted skeleton
bridges. In order to  perform the ladder summations for the complementary set $\bar \zeta$, we shall have to split the $\f x$ and $\f b$ variables according to the splitting $\Sigma = \zeta \sqcup \bar \zeta$. 
Recalling Definition \ref{def:ladders}, we split the labels according to
$\f x = (\f x_s, \f x_l)$, where $\f x_s = (x_i)_{i \in V_s(\Sigma, \f b)}$ and $\f x_l = (x_i)_{i \in V_l(\Sigma, \f b)}$.
We further split $\f x_l = (\f x_\sigma)_{\sigma \in \Sigma}$, where $\f x_\sigma$ contains the labels $x_i$ indexed by ladder vertices $i$ of $L_\sigma(\Sigma, \f b)$ (see Definition \ref{def:ladders}). We also introduce the set
  $L_\zeta(\Sigma, \f b) \deq \bigcup_{\sigma \in \zeta} L_\sigma(\Sigma, \f b)$, which yields the splitting $\f x_l = (\f x_\zeta, \f x_{\bar \zeta})$. Finally, we split the multiplicities $\f b = (b_\sigma)_{\sigma \in \Sigma} = (\f b_\zeta, \f b_{\bar \zeta})$, where $\f b_\zeta = (b_\sigma)_{\sigma \in \zeta}$.

Our next goal is to fix the multiplicities $\f b_{\zeta}$ and the labels $\f x_s, \f x_\zeta$, and to sum over $\f b_{\bar \zeta}$ and $\f x_{\bar \zeta}$. 
In order to perform the summation over $\f b_{\bar \zeta}$ and $\f x_{\bar \zeta}$, we shall first have to decouple them from the other summation variables, i.e.\ remove the constraints on the summation variables that involve the fixed variables.
The key idea behind the decoupling is to parametrize a partition $\Gamma \in \fra F_\zeta(\Sigma, \f b_\zeta, \f b_{\bar \zeta})$ using a partition $\wt \Gamma \in \fra F_\zeta(\Sigma, \f b_\zeta, \f 1)$ combined with the missing multiplicities $\f b_{\bar \zeta}$. Here $\f b_{\bar \zeta} = \f 1$ denotes the trivial multiplicities $(b_\sigma)_{\sigma \in \bar \zeta}$ where $b_\sigma = 1$ for all $\sigma \in \bar \zeta$. Informally, we collapse all ladders encoded by the set of not disrupted skeleton bridges $\bar \zeta$ into single bridges.

To make this idea precise, we first observe that there is a canonical bijection between the skeleton vertices $V_s$ of $\Gamma \in \fra F_\zeta(\Sigma, \f b_\zeta, \f b_{\bar \zeta})$ and of $\wt \Gamma \in \fra F_\zeta(\Sigma, \f b_\zeta, \f 1)$ (since the length of each ladder can be varied independently without changing the skeleton structure). Similarly, for each $\sigma \in \zeta$ there is a canonical bijection between the ladders  $L_\sigma$ in
 $\Gamma$ and $\wt \Gamma$ (since the length of each ladder encoded by a disrupted skeleton bridge in $\zeta$ is the same in $\Gamma$ and $\wt \Gamma$). 
Instead of a cumbersome formal definition, we refer to Figure \ref{fig: skeleton}, where the skeleton vertices $V_s$ are drawn in black or white. We shall use these bijections tacitly throughout the following, in particular identifying the sets $V_s$ of $\Gamma$ and $\wt \Gamma$ as  well as the labels $\f x_\zeta$ of $\Gamma$ and $\wt \Gamma$.

Next, given a $\Gamma \in \fra F_\zeta(\Sigma, \f b_\zeta, \f b_{\bar \zeta})$, the above bijection uniquely defines $\wt \Gamma \in \fra F_\zeta(\Sigma, \f b_\zeta, \f 1)$.  Moreover, we may recover the partition $\Gamma \in \fra F_\zeta(\Sigma, \f b_\zeta, \f b_{\bar \zeta})$ from $\wt \Gamma \in \fra F_\zeta(\Sigma, \f b_\zeta, \f 1)$ and $\f b_{\bar \zeta}$. To that end, we note that $\Gamma$ coincides with $\wt \Gamma$ on the set of edges $\bigcup \hb{\pi \col \pi \in L_\zeta(\Sigma, \f b)}$ (which is common to both $\Gamma$ and $\wt \Gamma$), and that on the complementary set $\bigcup \hb{\pi \col \pi \in L_{\bar \zeta}(\Sigma, \f b)}$ the partition $\Gamma$ coincides with the pairing $\cal G(\Sigma, \f b)$. We have an analogous parametrization of $\Xi \in \fra H(\Gamma)$ using $\wt \Xi \in \fra H(\wt \Gamma)$: on the set of edges $\bigcup \hb{\pi \col \pi \in L_\zeta(\Sigma, \f b)}$ the partition $\Xi$ coincides with $\wt \Xi$, and that on the complementary set $\bigcup \hb{\pi \col \pi \in L_{\bar \zeta}(\Sigma, \f b)}$ the partition $\Xi$ is atomic. We shall sometimes use the notations $\Gamma = \Gamma(\wt \Gamma, \f b_{\bar \zeta})$ and $\Xi = \Xi(\wt \Xi, \f b_{\bar \zeta})$ to denote these parametrizations.

The following lemma says roughly that the operation $\Phi$ (introduced in Section~\ref{sec:refining_partition}) 
for refining partitions into pairings commutes with the parametrization $(\wt \Gamma, \wt \Xi) \mapsto (\Gamma, \Xi)$, and that the indicator function $D$ (see \eqref{def_D_ind}) is invariant under this parametrization. Note that it was precisely because of this first property that the algorithm $\Phi$ had to be defined carefully in Section \ref{sec:refining_partition}, so that how it breaks up a block of $\Gamma$ is independent of the other blocks of $\Gamma$.

\begin{lemma} \label{lem:lift}
Let $\wt \Gamma \in \fra F_\zeta(\Sigma, \f b_\zeta, \f 1)$ and $\wt \Xi \in \fra H(\wt \Gamma)$. Let $\f b_{\bar \zeta}$ be arbitrary, and let $\Gamma = \Gamma(\wt \Gamma, \f b_{\bar \zeta})$ and $\Xi = \Xi(\wt \Xi, \f b_{\bar \zeta})$. Then the following hold.
\begin{enumerate}
\item
$D(\wt \Gamma, \wt \Xi) = D(\Gamma, \Xi)$.
\item
$\Phi(\wt \Gamma, \wt \Xi) = \cal G(\Sigma, \f b_\zeta, \f 1)$ if and only if $\Phi(\Gamma, \Xi) = \cal G(\Sigma, \f b_\zeta, \f b_{\bar \zeta})$.
\end{enumerate}
\end{lemma}
\begin{proof}
To prove (i), we note that $D = 1$ if and only if the first term on the right-hand side of \eqref{def_D_ind} is one and the second term is zero. It is easy to check that each term remains invariant under the change $(\wt \Gamma, \wt \Xi) \mapsto (\Gamma, \Xi)$. The claim (ii) is an easy consequence of the definition of the algorithm $\Phi$. 
\end{proof}

Using Lemma \ref{lem:lift} and the parametrization defined above, we may rewrite \eqref{V_T_prime_2} as
\begin{multline} \label{V_T_prime_3}
\cal V_\zeta'(\Sigma) \;=\; \sum_{\f b_\zeta}
\sum_{\wt \Gamma \in \fra F_\zeta(\Sigma, \f b_\zeta, \f 1)} \sum_{\wt \Xi \in \fra H(\wt \Gamma)} \indb{\Phi(\wt \Gamma, \wt \Xi) = \cal G(\Sigma, \f b_\zeta, \f 1)} \, D(\wt \Gamma, \wt \Xi)
\sum_{\f x_s} \sum_{\f x_\zeta} \pBB{\prod_{\{e,e'\} \in L_\zeta(\Sigma, \f b_\zeta, \f 1)} S_{x_e}}
\\
\times 
\sum_{\f b_{\bar \zeta}} \sum_{\f x_{\bar \zeta}} I(\f x) \, B_{\Gamma(\wt \Gamma, \f b_{\bar \zeta}),\Xi (\wt \Xi, \f b_{\bar \zeta})}(\f x) \,
\pbb{s \, \prod_{\sigma \in \Sigma} \chi_\sigma^{b_\sigma}} * \psi_1^{\leq, \eta}(E_1) * \psi_2^{\leq, \eta}(E_2)
\,
\pBB{\prod_{\{e,e'\} \in L_{\bar \zeta}(\Sigma, \f b)} S_{x_e}}\,.
\end{multline}

In \eqref{V_T_prime_3} we separated the summations over $\f b_{\bar \zeta}$ and $\f x_{\bar \zeta}$ from
their counterparts without bars, but 
the variables $\f b_{\bar \zeta}$ and $\f x_{\bar \zeta}$ are still coupled to the variables $\f b_\zeta$, $\f x_s$, and $\f x_\zeta$ through the indicator functions $I$ and $B$. We therefore have to rewrite these indicator functions in a more amenable form. Define the set of unordered edge labels
\begin{equation*}
E_{\f b_\zeta}\p{\f x_s , \f x_\zeta} \;\deq\; \hb{[x_e] \col \{e,e'\} \in L_\zeta(\Sigma, \f b_\zeta, \f 1)}\,,
\end{equation*}
which collects all unordered edge labels associated with edges belonging  to the $\zeta$-bridges.  As the notation implies, this set does not depend on the labels $\f x_{\bar \zeta}$. Moreover, we have the trivial bound
\begin{equation*}
\absb{E_{\f b_\zeta}\p{\f x_s , \f x_\zeta}} \;\leq\; C_\Sigma \, M^\mu\,.
\end{equation*}
We also recall the indicator function $J_\pi(\f x)$ from \eqref{def_I_sigma}, 
which enforces the unordered edge labels to coincide within a bridge $\pi$. Introduce the indicator function
\begin{equation*}
J_{\{e,\tilde e\}, \{e',\tilde e'\}}(\f x) \;\deq\; \indb{[x_e] = [x_{e'}] = [x_{\tilde e}] = [x_{\tilde e'}]}\,
\end{equation*}
associated with the event that two bridges, $\pi= \{e,\tilde e\}$ and $\pi'= \{e',\tilde e'\}$, have the same
unordered edge labels. Notice that this is a very atypical situation; the leading contribution
comes from the case when all bridges have distinct edge labels.

Using these notations, we may rewrite the above indicator function $B$ as 
\begin{multline} \label{B_split}
B_{\Gamma(\wt \Gamma, \f b_{\bar \zeta}),\Xi (\wt \Xi, \f b_{\bar \zeta})}(\f x) \;=\; B_{\wt \Gamma, \wt \Xi}\p{\f x_s, \f x_\zeta} \pBB{\prod_{\{e,e'\} \in L_{\bar \zeta}(\Sigma, \f b_\zeta, \f b_{\bar \zeta})} J_{\{e,e'\}}(\f x_s, \f x_{\bar \zeta}) \indb{[x_e] \notin E_{\f b_\zeta}\p{\f x_s , \f x_\zeta}}}
\\
\times \pBB{\prod_{\pi \neq  \pi' \in L_{\bar \zeta}(\Sigma, \f b_\zeta, \f b_{\bar \zeta})} \pb{1 - J_{\pi,\pi'}(\f x_s, \f x_{\bar \zeta})}}\,.
\end{multline}
The first factor implements all constraints among the $\zeta$-bridges of $L_\zeta(\Sigma, \f b_\zeta, \f b_{\bar \zeta}) \simeq L_\zeta(\Sigma, \f b_\zeta, \f 1)$. The remaining factors provide an explicit form of the constraints among the remaining 
$\bar\zeta$-bridges, which will be needed for the subsequent summation over the $\bar \zeta$-variables: the second factor implements the condition that edges of $L_{\bar \zeta}(\Sigma, \f b_\zeta, \f b_{\bar \zeta})$ belonging to the same bridge have compatible labels (as defined by $J_{\{e,e'\}}$) and that their labels be distinct from the labels associated with the $\zeta$-bridges; 
 the final factor implements the condition that distinct  $\bar\zeta$-bridges
have distinct unordered labels.

Plugging \eqref{B_split} into \eqref{V_T_prime_3} yields
\begin{multline} \label{V_T_prime_4}
\cal V_\zeta'(\Sigma) \;=\; \sum_{\f b_\zeta}
\sum_{\Gamma \in \fra F_\zeta(\Sigma, \f b_\zeta, \f 1)} \sum_{\Xi \in \fra H(\Gamma)} \indb{\Phi(\Gamma, \Xi) = \cal G(\Sigma, \f b_\zeta, \f 1)} \, D(\Gamma, \Xi)
\sum_{\f x_s} \sum_{\f x_\zeta}
B_{\Gamma, \Xi}\p{\f x_s, \f x_\zeta}
\pBB{\prod_{\{e,e'\} \in L_\zeta(\Sigma, \f b_\zeta, \f 1)} S_{x_e}}
\\
\times 
\sum_{\f b_{\bar \zeta}} \sum_{\f x_{\bar \zeta}} I(\f x)
\pBB{\prod_{\{e,e'\} \in L_{\bar \zeta}(\Sigma, \f b_\zeta, \f b_{\bar \zeta})} J_{\{e,e'\}}(\f x_s, \f x_{\bar \zeta}) \indb{[x_e] \notin E_{\f b_\zeta}\p{\f x_s , \f x_\zeta}} S_{x_e}}
\\
\times \pBB{\prod_{\pi \neq \pi' \in L_{\bar \zeta}(\Sigma, \f b_\zeta, \f b_{\bar \zeta})} \pb{1 - J_{\pi,\pi'}(\f x_s, \f x_{\bar \zeta})}}
\, \pbb{s \, \prod_{\sigma \in \Sigma} \chi_\sigma^{b_\sigma}} * \psi_1^{\leq, \eta}(E_1) * \psi_2^{\leq, \eta}(E_2)\,,
\end{multline}
where we dropped the tildes from the partitions $\wt \Gamma$ and $\wt \Xi$ to unclutter notation.

\subsubsection*{Step 3. Resolution of coincidences among the $\bar \zeta$-labels}
Our current goal is to estimate the last two lines of \eqref{V_T_prime_4} by making use of the oscillations in the phases
 $\chi_\sigma$ for $\sigma \in \bar \zeta$. To unclutter notation, in the following we frequently omit  the arguments $\Sigma$, $\f b = (\f b_\zeta, \f b_{\bar \zeta})$, and $\f x$, unless they are needed to avoid confusion. We have to factorize the right-hand side of \eqref{V_T_prime_4} into a product over $\sigma$. The main obstacle is the profusion of indicator functions introducing constraints among the $\bar \zeta$-labels $\f x_{\bar \zeta}$. We split
\begin{equation*}
\prod_{\pi \neq \pi' \in L_{\bar \zeta}} \pb{1 - J_{\pi,\pi'}(\f x_s, \f x_{\bar \zeta})} \;=\; \pBB{\prod_{\sigma \in \bar \zeta} \pb{1 - U_\sigma(\f x_s, \f x_\sigma)}} \pBB{\prod_{\sigma \neq \sigma' \in \bar \zeta} \pb{1 - U_{\sigma, \sigma'}(\f x_s, \f x_\sigma, \f x_{\sigma'})}}\,,
\end{equation*}
where we defined the indicator functions
\begin{equation} \label{def_U_sigma}
1 - U_\sigma \;\deq\; \prod_{\pi\neq \pi' \in L_\sigma} \p{1 - J_{\pi,\pi'}}
\,,\qquad
1 - U_{\sigma,\sigma'} \;\deq\; \prod_{\pi \in L_\sigma} \prod_{\pi \in L_{\sigma'}} \p{1 - J_{\pi,\pi'}}\,.
\end{equation}
The first expression excludes that two different bridges within the same ladder $\sigma$ have 
the same labels, while the second excludes that two bridges from different ladders have the same labels.
We adopt the convention that $U_{\sigma,\sigma} \deq 0$.
We also introduce the indicator function
\begin{equation} \label{def_V_sigma}
1 - \wt U_\sigma \;\deq\; \prod_{\{e,e'\} \in L_{\sigma}} \pb{1 - \indb{[x_e] \in E_{\f b_\zeta}}}\,,
\end{equation}
which excludes that any label in a $\bar\zeta$-ladder $\sigma$ coincide with a $\zeta$-label.
(Note that the right-hand side depends on the choice of ordering of the pair $\{e,e'\}$; it may be chosen arbitrarily, since we shall always use $1 - \wt U_\sigma$ within an expression in which $[x_e] = [x_{e'}]$ for $\{e,e'\} \in L_\sigma$.)

The idea behind these definitions is that the term $1$ is typical (and hence yields a leading contribution); the error terms $U_\sigma$, $U_{\sigma, \sigma'}$, and $\wt U_\sigma$ yield smaller contributions resulting from coinciding summation labels. We expand
\begin{equation} \label{splitting_cal_J}
\cal J \;\deq\; \pBB{\prod_{\sigma \in \bar \zeta} (1 - U_\sigma) (1 - \wt U_\sigma)} \pBB{\prod_{\sigma , \sigma' \in \bar \zeta} \p{1 - U_{\sigma, \sigma'}}} \;=\; \sum_{\alpha, \beta \subset \bar \zeta} \sum_{\gamma \subset \bar \zeta^2} \cal J_{\alpha, \beta, \gamma}\,,
\end{equation}
where we defined
\begin{equation} \label{def_cal_J_2}
\cal J_{\alpha, \beta, \gamma} \;\deq\; \prod_{\sigma \in \alpha} (- U_\sigma) \prod_{\sigma \in \beta} (-\wt U_\sigma) \prod_{\{\sigma, \sigma'\} \in \gamma} (- U_{\sigma,\sigma'})\,.
\end{equation}
Thus we have the splitting $\cal J = \sum_{\xi \subset \bar \zeta} \cal J_\xi$, where
\begin{equation}\label{Jxi}
\cal J_\xi \;\deq\; \sum_{\alpha, \beta \subset \bar \zeta} \sum_{\gamma \subset \bar \zeta^2} \ind{\xi = \alpha \cup \beta \cup [\gamma]} \, \cal J_{\alpha, \beta, \gamma}
\end{equation}
and we defined $[\gamma] \deq \bigcup \{\{\sigma, \sigma'\} \col \{\sigma, \sigma'\} \in \gamma\}$. The interpretation of $\cal J_\xi$ is that it imposes constraints precisely on the ladders $L_\sigma$ indexed by $\sigma \in \xi$.
Hence, the set $\xi$ consists of all $\sigma \in \bar \zeta$ such that $\cal J_\xi$ introduces a constraint among the labels of $L_\sigma$. Abbreviate the complementary set by $\bar \xi \deq \bar \zeta \setminus \xi$;
these are the $\bar\zeta$-ladders whose labels are not subject to restrictions.
 We split the variables $\f b_{\bar \zeta} = (\f b_\xi, \f b_{\bar \xi})$ and $\f x_{\bar \zeta} = (\f x_\xi, \f x_{\bar \xi})$ in self-explanatory notation. Thus, we have $\cal J_\xi(\f x_s, \f x_{\bar \zeta}) \equiv \cal J_\xi(\f x_s, \f x_\xi)$, so that $\cal J_\xi$ does not depend on the $\bar \xi$-labels $\f x_{\bar \xi}$.

Hence we may write the two last lines of \eqref{V_T_prime_4} as $\sum_{\xi \subset \bar \zeta} R_\xi$, where
\begin{equation} \label{def_R_xi}
R_\xi \;\deq\;
\sum_{\f b_{\bar \zeta}} \sum_{\f x_{\bar \zeta}}
\pBB{\prod_{\{e,e'\} \in L_{\bar \zeta}} J_{\{e,e'\}} S_{x_e}} I (\f x) \, \cal J_\xi (\f x_s, \f x_{\xi}) \, \pbb{s \, \prod_{\sigma \in \Sigma} \chi_\sigma^{b_\sigma}} * \psi_1^{\leq, \eta}(E_1) * \psi_2^{\leq, \eta}(E_2)\,.
\end{equation}

We first fix the $\xi$-variables and sum over the $\bar \xi$-variables, using the fact that $\cal J_\xi$ does not depend on the $\bar \xi$-variables. Before we may do this summation, we have to deal with a final complication arising from the indicator function $I(\f x)$ implementing the nonbacktracking condition. Since $I(\f x)$ couples labels associated with vertices at distance two from each other, $I(\f x)$ does not factorize over the ladders $\sigma \in \bar \xi$. However, provided we treat the first and last summation label of each ladder separately, such a factorization is possible. For each $\sigma \in \bar \xi$ we split $\f x_\sigma = (\f y_\sigma, \f z_\sigma)$ into its \emph{inner labels} $\f y_\sigma$ and its \emph{edge labels} $\f z_\sigma$.  By definition, $\f z_\sigma$ consists of the labels associated with all ladder vertices of $L_\sigma(\Sigma, \f b)$ that are adjacent to a vertex that is not a ladder vertex of $L_\sigma(\Sigma, \f b)$; note that for $b_\sigma = 1$ there are no such vertices, for $b_\sigma = 2$ there are two such vertices, and for $b_\sigma \geq 3$ there are four such vertices.  See Figure \ref{fig: zx} for an illustration of this splitting. 
 We also abbreviate $\f y_{\bar \xi} = (\f y_\sigma)_{\sigma \in \bar \xi}$ and $\f z_{\bar \xi} = (\f z_\sigma)_{\sigma \in \bar \xi}$. Thus we get the factorization
\begin{equation*}
I(\f x) \;=\; \wt I(\f x_s, \f x_\zeta, \f x_\xi, \f z_{\bar \xi}) \prod_{\sigma \in \bar \xi} \pb{1 - W_\sigma(\f x_s, \f x_{\zeta}, \f z_{\bar \xi}, \f y_\sigma)}\,,
\end{equation*}
where the first factor $\wt I$ includes all terms in the product on the right-hand side of \eqref{def_I_ind}
 that do not depend on the inner labels $\f y_\sigma$ of any $\sigma \in \bar \xi$; the remaining terms on the right-hand side of \eqref{def_I_ind} depend on precisely one $\f y_\sigma$, so that they may be factorized over $\sigma \in \bar \xi$ and written in the form $1-W_\sigma$. The interpretation of $W_\sigma = 0$ is that, given $(\f x_s, \f x_\zeta, \f x_\xi, \f z_{\bar \xi})$, the summation over $\f y_\sigma$ is unrestricted. A glance at Figure \ref{fig: zx} should clarify this splitting.  Plugging this into \eqref{def_R_xi} yields
\begin{multline} \label{R_xi_1}
R_\xi \;=\;
\sum_{\f b_\xi} \sum_{\f x_\xi} \cal J_\xi(\f x_s, \f x_\xi) \,
\pBB{\prod_{\{e,e'\} \in L_{\xi}} J_{\{e,e'\}} S_{x_e}}
\sum_{\f b_{\bar \xi}} \sum_{\f z_{\bar \xi}}
\wt I(\f x_s, \f x_\zeta, \f x_\xi, \f z_{\bar \xi})
\sum_{\f y_{\bar \xi}}
\pBB{\prod_{\sigma \in \bar \xi} \pb{1 - W_\sigma(\f x_s, \f x_{\zeta}, \f z_{\bar \xi}, \f y_\sigma)}}
\\
\times \pBB{\prod_{\{e,e'\} \in L_{\bar \xi}} J_{\{e,e'\}} S_{x_e}} \, \pbb{s \, \prod_{\sigma \in \Sigma} \chi_\sigma^{b_\sigma}} * \psi_1^{\leq, \eta}(E_1) * \psi_2^{\leq, \eta}(E_2)\,.
\end{multline}
Similarly to the above splitting of $\bar \zeta = \xi \sqcup \bar \xi$, we perform a final splitting $R_\xi \;=\; \sum_{\vartheta \subset \bar \xi} R_{\xi, \vartheta}$, obtained by applying to the right-hand side of \eqref{R_xi_1} the splitting
\begin{equation} \label{intr_vartheta}
\prod_{\sigma \in \bar \xi} \p{1 - W_\sigma} \;=\; \prod_{\sigma \in \bar \xi} \p{1 - W_\sigma} \pb{\ind{b_\sigma \geq 3} + \ind{b_\sigma \leq 2}} \;=\; \sum_{\vartheta \subset \bar \xi} \cal K_{\xi, \vartheta}\,,
\end{equation}
where $\cal K_{\xi, \vartheta}$ collects all terms containing a factor $W_\sigma$ or $\ind{b_\sigma \leq 2}$ for some
 $\sigma \in \vartheta$. The leading contribution is the term $\cal K_{\emptyset, \xi}$, corresponding to $\vartheta=\emptyset$.
Note that $\cal K_{\xi, \vartheta}$ does not depend on $\f y_\sigma$ for all $\sigma \in \bar \vartheta \deq \bar \xi \setminus \vartheta$. Moreover, note that for each $\sigma \in \bar \vartheta$ we have $b_\sigma \geq 3$ for nonzero summands in $R_{\xi, \vartheta}$.

\subsubsection*{Step 4. Summing over the $\bar \xi$-variables using oscillations}
We now explicitly perform the $\f y_\sigma$-summations for each $\sigma \in \bar \vartheta$. To that end, for $\sigma \in \Sigma$ we introduce the labels $x_1$, $x_2$, $x_1'$, $x_2'$, $z_1$, and $z_2$ associated with the vertices at the ends of the ladder $L_{\sigma}$, as defined in Figure \ref{fig: zx}.
\begin{figure}[ht!]
\begin{center}
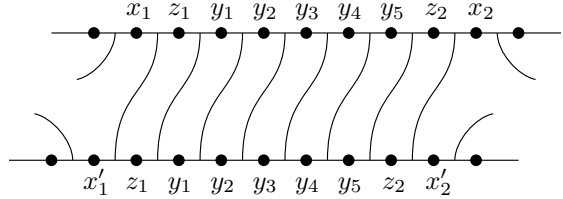
\end{center}
\caption{A ladder $L_\sigma$ with its labels.
The skeleton $\f x_s$-labels are denoted by $x_1$, $x_2$, $x_1'$, and $x_2'$,
the  edge labels by $z_1$ and $z_2$ and the inner labels 
 by $y_1, \dots, y_5$. We display each label next to the vertex it is associated with and omit the argument $\sigma$. The choice of names associated with the ordering of the vertices is immaterial. Since the omnipresent indicator functions $J_{\{e,e'\}}$ force two labels on the same side of a bridge to coincide, we use the same letter for them. For later purposes, we distinguish the labels $x_1$ and $x_1'$ as well as $x_2$ and $x_2'$, although the same indicator functions always impose that $x_1 = x_1'$ and $x_2 = x_2'$. 
\label{fig: zx}} 
\end{figure}
We may sum up the $\f y_\sigma$-labels for $\sigma \in \bar \vartheta$. The result is
\begin{multline*}
R_{\xi, \vartheta} \;=\;
\sum_{\f b_\xi} \sum_{\f x_\xi} \cal J_\xi \,
\pBB{\prod_{\{e,e'\} \in L_{\xi}} J_{\{e,e'\}} S_{x_e}}
\sum_{\f b_{\bar \xi}} \sum_{\f z_{\bar \xi}}
\wt I
\sum_{\f y_{\vartheta}}
\cal K_{\xi, \vartheta} \pBB{\prod_{\{e,e'\} \in L_\vartheta} J_{\{e,e'\}} S_{x_e}} 
\\
\times  \pBB{s \, \pBB{\prod_{\sigma \in \Sigma} \chi_\sigma^{b_\sigma}} \pBB{\prod_{\sigma \in \bar \vartheta} S_{x_1(\sigma) z_1(\sigma)} \pb{S^{b_\sigma - 2}}_{z_1(\sigma) z_2(\sigma)} S_{z_2(\sigma) x_2(\sigma)} \delta_{x_1(\sigma) x_1'(\sigma)} \delta_{x_2(\sigma) x_2'(\sigma)}}} * \psi_1^{\leq, \eta}(E_1) * \psi_2^{\leq, \eta}(E_2)\,.
\end{multline*}
We may now explicitly sum up the geometric series arising from the summation over $b_\sigma = 3,4, \dots, [M^\mu]$ for each $\sigma \in \bar \vartheta$. This summation exploits the required cancellations, and, having done it, we may take the absolute value inside the summation. Estimating $\wt I \leq I_0$ (recall \eqref{def_I_ind}),  we get
\begin{multline} \label{R_xi_theta_estimate}
\abs{R_{\xi, \vartheta}} \;\leq\; C
\sum_{\f b_\xi} \sum_{\f x_\xi} \abs{\cal J_\xi} I_0 \,
\pBB{\prod_{\{e,e'\} \in L_{\xi}} J_{\{e,e'\}} S_{x_e}}
\sum_{\f b_\vartheta} \sum_{\f x_\vartheta}
\pBB{\prod_{\sigma \in \vartheta} \pb{W_\sigma + \ind{b_\sigma \leq 2}}}
\pBB{\prod_{\{e,e'\} \in L_\vartheta} J_{\{e,e'\}} S_{x_e}} 
\\
\times \prod_{\sigma \in \bar \vartheta} \pBB{ \wt Z_{x_1(\sigma) x_2(\sigma)}(\sigma) \, \delta_{x_1(\sigma) x_1'(\sigma)} \delta_{x_2(\sigma) x_2'(\sigma)} }\,,
\end{multline}
where we defined 
\begin{equation} \label{def_wt_Z}
\wt Z_{xy}(\sigma) \;\deq\; \sum_{u,w} S_{xu} \sup \hB{\absb{\pb{Z(\chi_\sigma(E_1 - v_1, E_2 - v_2) S)}_{uw}} \col \abs{v_1}, \abs{v_2} \leq 2 \eta M^{\delta/2}} \, S_{wy}
\end{equation}
and
\begin{equation} \label{def_Z}
Z(x) \;\deq\; \sum_{b = 1}^{[M^\mu] - 2} x^b \;=\; \frac{x (1 - x^{[M^\mu] - 2})}{1 - x}\,.
\end{equation}
Note that the condition on $v_1$ and $v_2$ in \eqref{def_wt_Z} amounts to constraining
them to lie in the support of $\psi_1^{\leq, \eta}(v_1) \psi_2^{\leq, \eta}(v_2)$ in the convolution integral; see \eqref{supp_phi_trunc} and recall that $\theta = \delta$.

Bearing later applications in mind, we introduce a general class of matrices $\cal Z(\sigma)$, parametrized by the bridges $\sigma$ of $\Sigma$, which satisfy a set of simple bounds that are sufficient to conclude the proof.

\begin{definition} \label{def_admissible_Z}
Recall the splitting $\Sigma = \Sigma_1 \sqcup \Sigma_2 \sqcup \Sigma_c$ from Definition \ref{def_split_Sigma}.
We call the family of matrices $\cal Z(\sigma, E_1, E_2, L) \equiv \cal Z (\sigma)$ parametrized by $\sigma \in \Sigma$ \emph{admissible} if
\begin{equation} \label{domestic_bound2}
\abs{\cal Z_{xy}(\sigma)} \;\leq\; \frac{C}{M}\,, \qquad
\sum_{y} \abs{\cal Z_{xy}(\sigma)} \;\leq\; C \log N
\end{equation}
for $\sigma \in \Sigma_1 \cup \Sigma_2$ and
\begin{equation} \label{connecting_bound2}
\abs{\cal Z_{xy}(\sigma)} \;\leq\; \frac{C}{M} M^{2 \delta} R_2(\omega + \eta)\,, \qquad
\sum_y \abs{\cal Z_{xy}(\sigma)} \;\leq\; C M^\mu\,
\end{equation}
for $\sigma \in \Sigma_c$.
\end{definition}

As advertised, the matrix $\wt Z$ from \eqref{def_wt_Z}
satisfies the estimates in
Definition \ref{def_admissible_Z}.
\begin{lemma}
For small enough $\delta$ the matrices $\wt Z$ from \eqref{def_wt_Z} are admissible
in the sense of Definition \ref{def_admissible_Z}.
\end{lemma}
\begin{proof}
The claim is a trivial corollary of \cite[Lemma 4.16]{EK3}. We remark that the proof of \cite[Lemma 4.16]{EK3} relies on the estimates \eqref{bound res 1} and \eqref{bound res 2}. It makes essential use of the oscillations in the sum \eqref{def_Z}. See \cite{EK3} for the full details.
\end{proof}

In this step we have achieved the main goal: to exploit the oscillations
from an appropriate subset of the ladders. In the remainder of the argument we have to perform the remaining
summations. We shall not use oscillations any longer: only the size of the terms and their combinatorics will matter.

\subsubsection*{Step 5. Summing over the $\bar \xi$-variables without oscillations}
We now sum over the remaining $\bar \xi$-variables, i.e.\ the $\vartheta$-variables $\f b_\vartheta$ and $\f x_\vartheta$ in \eqref{R_xi_theta_estimate}. This is a straightforward estimate, since the right-hand side of \eqref{R_xi_theta_estimate} factorizes over all ladders encoded by $\sigma \in \vartheta$. The contribution to the right-hand side \eqref{R_xi_theta_estimate} of an individual ladder $\sigma \in \vartheta$ incident to skeleton vertices with labels  $x_1=x_1'$ and $x_2=x_2'$ 
 from $\f x_s$  (see Figure \ref{fig: zx}) is estimated by
\begin{equation} \label{backtracking_coincidences}
\sum_{b = 1}^{[M^\mu]} \sum_{y_1, \dots, y_{b - 1}} S_{x_1 y_1} S_{y_1 y_2} \cdots S_{y_{b - 2} y_{b - 1}} S_{y_{b - 1} x_2}\qBB{\pBB{1 - \prod_{i = 1}^{b - 4} \ind{y_i \neq y_{i + 2}}} + \ind{b \leq 2}} \;\leq\; \cal Z_{x_1 x_2}(\sigma)\,,
\end{equation}
where $\cal Z$ satisfies Definition \ref{def_admissible_Z} 
(in fact, it satisfies the stronger bound \eqref{domestic_bound2} for all $\sigma$). To prove \eqref{backtracking_coincidences}, we note that the estimate of the term proportional to $\ind{b \leq 2}$ is trivial. We deal with the other term by writing $\ind{y_i \neq y_{i + 2}} = 1 - \ind{y_i = y_{i + 2}}$ and using the elementary bound 
\begin{equation} \label{product_exp}
0 \;\leq\; 1 - \prod_{i = 1}^n (1 - a_i) \;\leq\; \sum_{i=1}^n a_i \qquad (a_i\in [0,1]) 
\end{equation}
(with  $a_i = \ind{y_i = y_{i + 2}}$ and
$n = b-4$ in this case).
This yields $b - 4 \leq M^\mu$ terms, each of which is bounded by $M^{-1} (S^k)_{x_1 x_2}$ for some $k \in \N$.
The factor $M^{-1}$ comes from the 
\begin{equation}\label{gainM}
\sum_{y_{i+1}} S_{y_i y_{i+1}}S_{y_{i+1}y_{i}} \;=\; (S^2)_{y_i y_i} \;\leq\; \frac{C}{M}
\end{equation}
summations in the event $a_i = \ind{y_i = y_{i + 2}}$; the rest of the $S$-factors are summed up freely.
 This concludes the proof of \eqref{backtracking_coincidences}.
Recall that $R_\xi = \sum_{\vartheta \subset \bar \xi} R_{\xi, \vartheta}$ and $\abs{\xi} \leq \abs{\Sigma}$. Plugging \eqref{backtracking_coincidences} into \eqref{R_xi_theta_estimate} therefore yields
\begin{equation} \label{R_xi_theta_estimate2}
\abs{R_\xi} \;\leq\; C_\Sigma \, I_0(\f x_s) 
\sum_{\f b_\xi} \sum_{\f x_\xi} \absb{\cal J_\xi (\f x_s, \f x_\xi)}\,
\pBB{\prod_{\{e,e'\} \in L_{\xi}(\Sigma, \f b_\zeta, \f b_\xi, \f 1)} J_{\{e,e'\}} S_{x_e}}
\pBB{\prod_{\{e,e'\} \in L_{\bar \xi}(\Sigma, \f b_\zeta, \f b_\xi, \f 1)} J_{\{e,e'\}} \cal Z_{x_e}(\{e,e'\})}\,,
\end{equation}
where we split $\f b = (\f b_\zeta, \f b_\xi, \f b_{\bar \xi})$ and abbreviated $\f b_{\bar \xi} = \f 1$ to denote $b_\sigma = 1$ for all $\sigma \in \bar \xi$.

\begin{remark} \label{rem:small1}
For future purposes we observe that if $\vartheta \neq \emptyset$ then the estimate \eqref{R_xi_theta_estimate2} is in fact valid with an extra factor $M^{2 \mu - 1}$ on the right-hand side. (In \eqref{R_xi_theta_estimate2} we simply estimated this factor by one.) This factor arises from the preceding argument for $\sigma \in \vartheta$: for each indicator function $a_i$ in \eqref{product_exp} we get a factor $M^{-1}$, there are $M^\mu$ such terms, and the $b_\sigma$-summation yields another factor $M^\mu$.
\end{remark}

\subsubsection*{Step 6. Summing over the $\xi$-variables}
Fix $\xi\subset \Sigma$.
We may now sum over $\f b_\xi$ and $\f x_\xi$ on the right-hand side of \eqref{R_xi_theta_estimate2}. 
The summation over  $\f b_\xi$ will be performed trivially, yielding a factor $M^{\abs{\xi} \mu}$. Hence it suffices to regard all $\f b_\xi$ as fixed.
Since for any fixed $\xi$ the sum on the right-hand side of \eqref{Jxi}
 contains $O_\Sigma(1)$ terms, it suffices to estimate the contribution of a single term $\cal J_{\alpha, \beta, \gamma}$ to the right-hand side of \eqref{R_xi_theta_estimate2}.
 Thus, for the following we fix $\alpha, \beta \subset \bar \zeta$ and $\gamma \subset \bar \zeta^2$ satisfying $\xi = \alpha \cup \beta \cup [\gamma]$. Define $\beta' \deq \beta \setminus \alpha$. Moreover, let $\gamma'$ be a minimal subset of $\gamma$ such that $\xi = \alpha \cup \beta' \cup [\gamma']$. Since $\beta' \subset \beta$ and $\gamma' \subset \gamma$, we may estimate 
\begin{equation} \label{cal_J_estimate}
\abs{\cal J_{\alpha, \beta, \gamma}} \;\leq\; \prod_{\sigma \in \alpha} U_\sigma \prod_{\sigma \in \beta'} \wt U_{\sigma}
 \prod_{\{\sigma, \sigma'\} \in \gamma'} U_{\sigma,\sigma'}\,.
\end{equation}
We therefore need to estimate
\begin{multline} \label{typical}
C_\Sigma \, I_0(\f x_s)  \, M^{\abs{\xi} \mu} \sum_{\f x_\xi} \Bigg(\prod_{\sigma \in \alpha} U_\sigma \prod_{\sigma \in \beta'} \wt U_{\sigma}
 \prod_{\{\sigma, \sigma'\} \in \gamma'} U_{\sigma,\sigma'}
 \Bigg) \\
\times\pBB{\prod_{\{e,e'\} \in L_{\xi}(\Sigma, \f b_\zeta, \f b_\xi, \f 1)} J_{\{e,e'\}} S_{x_e}}
\pBB{\prod_{\{e,e'\} \in L_{\bar \xi}(\Sigma, \f b_\zeta, \f b_\xi, \f 1)} J_{\{e,e'\}} \cal Z_{x_e}(\{e,e'\})}\,.
\end{multline}
For each factor $U$ in \eqref{typical} we shall use the estimates 
\begin{equation}\label{expU}
 U_\sigma \;\leq\; \sum_{\pi\ne \pi'\in L_\sigma} J_{\pi, \pi'}\,, 
\qquad
\wt U_{\sigma} \;\leq\; \sum_{\{ e, e'\}\in L_{\sigma}} \indb{[x_e] \in E_{\f b_\zeta}}\,,
\qquad
U_{\sigma,\sigma'} \;\leq\; \sum_{\pi\in L_{\sigma}}\sum_{\pi'\in L_{\sigma'}}  
J_{\pi, \pi'}\,,
\end{equation} 
which follow by applying the estimate \eqref{product_exp} to 
 the definitions \eqref{def_U_sigma} and \eqref{def_V_sigma}.

Next, we
 sum over $\f x_\xi$, by summing over the labels of the ladder vertices $\f x_\sigma$ for $\sigma \in \xi$. The basic intuition behind this summation is that, thanks to the coincidences
imposed by the factors in the first line of \eqref{typical}, the ladder associated with $\sigma$ may be estimated by an admissible factor $\cal Z$
(see Definition \ref{def_admissible_Z}) times a small factor $M^{-1}$ as in \eqref{gainM}.
This factor will compensate the factors of $M^{\mu}$ arising from summations over $b_\sigma$, over $\pi \in \sigma$, and possibly over $\pi'\in \sigma$.
(Recall that $\mu < 1/3$, so that three factors of $M^\mu$ are affordable for each factor of $M^{-1}$. In fact, the estimates presented here only produce at most two factors of $M^\mu$ for each factor $M^{-1}$, and in particular work up to $\mu < 1/2$.) Some care is needed in accounting for the coincidences in $U_{\sigma, \sigma'}$, since here \emph{two} $b_\sigma$-summations need to be compensated. Together with the summations over $\pi, \pi'$, 
the combinatorics give a factor $M^{4\mu}$, which would not be affordable against the gain of $M^{-1}$, but it turns
out that in the generic situation the gain is in fact $M^{-2}$. In order to observe this gain, we need to perform this summation in an appropriate order.
To that end, it is convenient to introduce a graph $\cal S = (V(\cal S), E(\cal S)) \deq (\xi, \gamma')$, whose vertices
are identified with the ladders in $\xi$, and which encodes which distinct ladders are connected by a factor $J_{\pi, \pi'}$.
Note that $\cal S$ is  a forest graph by the minimality assumption on $\gamma'$.
We say that $\{\sigma, \sigma'\} \in \gamma'$ is a \emph{twig} of $\cal S$ if $\sigma$ and $\sigma'$ both have degree one and do not belong to $\alpha \cup \beta'$. The twigs represent the problematic factors of $J_{\pi, \pi'}$ from which $M^{-2}$ needs to be gained.

\begin{figure}[ht!]
\begin{center}
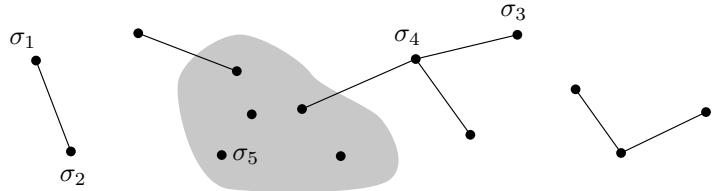
\end{center}
\caption{The forest $\cal S$.
 Vertices in $\alpha \cup \beta'$ are inside the shaded region. The leftmost connected component of $\cal S$ is its only twig. \label{fig: tree}} 
\end{figure}

The summation proceeds as follows. First, we sum recursively over all $\f x_\sigma$ for $\sigma \notin \alpha \cup \beta'$. We do this by starting at the leaves of $\cal S$, at each step removing the leaf $\sigma$ after the summation over $\f x_\sigma$ has been completed. If we encouter a leaf that belongs to a twig $\{\sigma, \sigma'\}$, we sum $\f x_\sigma$ and $\f x_{\sigma'}$ simultaneously. Finally, after only vertices in $\alpha \cup \beta'$ remain, we sum up the associated labels $\f x_\sigma$ one by one.

Now we give more details about each step of the summation. Suppose that $\sigma \notin \alpha \cup \beta'$ is a leaf of $\cal S$ that does not belong to a twig (e.g.\ $\sigma_3$ in Figure \ref{fig: zx}).  Let $\sigma'$ be the vertex of $\cal S$ adjacent to $\sigma$  (continuing with the example $\sigma = \sigma_3$ from Figure \ref{fig: zx}, we have $\sigma' = \sigma_4$). We now estimate
\begin{equation} \label{sum_sigma_sigmap}
\sum_{\f x_{\sigma}} U_{\sigma, \sigma'}(\f x_s, \f x_{\sigma}, \f x_{\sigma'}) \prod_{\{e,e'\} \in L_{\sigma}} J_{\{e,e'\}} S_{x_e} \;\leq\; C M^{2\mu - 1} \cal Z_{x_1(\sigma) x_2(\sigma)}(\sigma)\,,
\end{equation}
where $x_1(\sigma)$ and $x_2(\sigma)$ are skeleton labels incident to $L_\sigma$ (see Figure \ref{fig: zx}). The proof of \eqref{sum_sigma_sigmap} follows by estimating $U_{\sigma, \sigma'}$ as in \eqref{expU}, and by noting that the sum contains $O(M^{2\mu})$ terms, and in each term the number of free summation labels is reduced by at least one, leading to a factor $M^{-1} \cal Z$; we omit further details. (In fact, it is not hard to check that the factor $M^{2 \mu - 1}$ on the right-hand side of \eqref{sum_sigma_sigmap} may be improved to $M^{\mu - 1}$, by the simple observation that the generic case leads to a reduction of $M^{-2}$, while a weaker reduction of $M^{-1}$ is obtained only for $O(M^\mu)$ terms and not the generic $O(M^{2 \mu})$ terms.)  Having summed over $\f x_\sigma$, we strike the vertex $\sigma$ and the edge $\{\sigma, \sigma'\}$ from the tree and repeat this process until $\cal S$ has no more leaves in $\xi \setminus (\alpha \cup \beta')$ that do not belong to twigs.

Next, we estimate the $\f x$-summation associated with twigs. Let $\{\sigma, \sigma'\}$ be a twig of $\cal S$ (e.g.\ $\{\sigma_1, \sigma_2\}$ in Figure \ref{fig: zx}). Similarly to above, we claim that
\begin{equation} \label{sum_sigma_twig}
\sum_{\f x_\sigma, \f x_{\sigma'}} U_{\sigma, \sigma'}(\f x_s, \f x_{\sigma}, \f x_{\sigma'}) \prod_{\{e,e'\} \in L_{\sigma} \cup L_{\sigma'}} J_{\{e,e'\}} S_{x_e} \;\leq\; C M^{-1} \cal Z_{x_1(\sigma) x_2(\sigma)}(\sigma) \, \cal Z_{x_1(\sigma') x_2(\sigma')}(\sigma')\,.
\end{equation}
To see this, we estimate $U_{\sigma, \sigma'}$  using \eqref{expU}
and note that the coincidence $J_{\pi, \pi'}$ typically reduces the number of free summation labels by two.  Here ``typically'' means that at least one of $\pi$ and $\pi'$ is only incident to ladder vertices (or, more informally, is away from  its skeleton vertices, i.e.\ the ends of the ladder). 
This yields a gain $M^{-2}$ and there are  $O(M^{2\mu})$ such terms, so the total gain is bounded by $M^{2\mu-2} \leq M^{-1}$.
In the non-typical situation (i.e.\ where both $\pi$ and $\pi'$ are incident to skeleton  vertices), we only gain a factor $M^{-1}$ from the coincidences,
since in this case
one of the two labels incident to a bridge is a skeleton label, which is fixed so that $J_{\pi, \pi'}$ forces 
only one instead of two ladder labels to coincide. 
But this may happen only if $\pi$ and $\pi'$ are both at an end of their ladders, so the number of such terms is only $O(1)$. This concludes the proof of\eqref{sum_sigma_twig}.
We repeat this process for each twig of $\cal S$, after which we strike the twig from $\cal S$.

This leaves us with the summation over $\f x_\sigma$ for $\sigma \in \alpha \cup \beta'$  (e.g.\ $\sigma_5$ in Figure \ref{fig: tree}). 
 These summations factorize over $\sigma \in \alpha \cup \beta'$, and may be estimated individually (again by using the bound from \eqref{expU}). For $\sigma \in \alpha$ we get 
\begin{equation} \label{sum_sigma_alpha}
\sum_{\f x_\sigma} U_{\sigma}(\f x_s, \f x_{\sigma}) \prod_{\{e,e'\} \in L_{\sigma}} J_{\{e,e'\}} S_{x_e} \;\leq\; C M^{\mu-1} \cal Z_{x_1(\sigma) x_2(\sigma)}(\sigma)\,,
\end{equation}
 and a similar estimate applies for the case $\sigma \in \beta'$, where $U_\sigma(\f x_s, \f x_\sigma)$ in \eqref{sum_sigma_alpha} is replaced with $\wt U_\sigma(\f x_s, \f x_\zeta, \f x_\sigma)$.

We may now conclude the summation over the labels $\f x_\xi$ in \eqref{typical}
by noting that on the right-hand sides of \eqref{sum_sigma_sigmap}, \eqref{sum_sigma_twig}, and \eqref{sum_sigma_alpha} each factor $\cal Z$ carries a factor that is bounded by $CM^{-\mu}$. (This follows from $\mu < 1/3$; in fact, $\mu < 1/2$ would be enough here.) 
 Going back to \eqref{R_xi_theta_estimate2}, we have therefore proved that
\begin{align*}
\abs{R_\xi} &\;\leq\; C_\Sigma \,  I_0(\f x_s)
\sum_{\f b_\xi}
\pBB{\prod_{\{e,e'\} \in L_{\xi}(\Sigma, \f b_\zeta, \f 1)} M^{-\mu} J_{\{e,e'\}} \cal Z_{x_e}(\{e,e'\})}
\pBB{\prod_{\{e,e'\} \in L_{\bar \xi}(\Sigma, \f b_\zeta, \f 1)} J_{\{e,e'\}} \cal Z_{x_e}(\{e,e'\})}
\\
&\;\leq\; C_\Sigma \, I_0(\f x_s) 
\prod_{\{e,e'\} \in L_{\bar \zeta}(\Sigma, \f b_\zeta, \f 1)} J_{\{e,e'\}} \cal Z_{x_e}(\{e,e'\})\,,
\end{align*}
where the multiplicities $\f 1$ refer to $\f b_{\bar \zeta}$.

Recalling \eqref{V_T_prime_4} and the definition of $R_\xi$ given before \eqref{def_R_xi}, we therefore conclude that
\begin{multline} \label{V_T_prime_5}
\abs{\cal V_\zeta'(\Sigma)} \;\leq\; C_\Sigma  \sum_{\f b_\zeta}
\sum_{\Gamma \in \fra F_\zeta(\Sigma, \f b_\zeta, \f 1)} \sum_{\Xi \in \fra H(\Gamma)} \indb{\Phi(\Gamma, \Xi) = \cal G(\Sigma, \f b_\zeta, \f 1)}
\sum_{\f x_s} I_0(\f x_s)
\\
\times
\sum_{\f x_\zeta}
B_{\Gamma, \Xi}\p{\f x_s, \f x_\zeta}
\pBB{\prod_{\{e,e'\} \in L_\zeta(\Sigma, \f b_\zeta, \f 1)} S_{x_e}}
\,
\pBB{\prod_{\{e,e'\} \in L_{\bar \zeta}(\Sigma, \f 1)} \cal Z_{x_e}(\{e,e'\})}\,,
\end{multline}
where we used the bounds $D(\Gamma, \Xi), J_{\{e,e'\}} \leq 1$ as well as the canonical bijection $L_{\bar \zeta}(\Sigma, \f b_\zeta, \f 1) \simeq L_{\bar \zeta}(\Sigma, \f 1)$ described in Step 2 at the beginning of this subsection.

\begin{remark} \label{rem:small2}
Similarly to Remark \ref{rem:small1}, we note that if $\xi \neq \emptyset$ then the estimate \eqref{V_T_prime_5} is in fact valid with an extra factor $M^{2 \mu - 1}$ on the right-hand side. More precisely, from each $\sigma \in \alpha \cup \beta'$ we gain a factor $M^{2 \mu - 1}$. Moreover, from each $\sigma \in [\gamma']$ we also get a factor $M^{2 \mu - 1}$, except if $\sigma$ belongs to a twig $\{\sigma, \sigma'\}$, in which case the whole twig $\{\sigma, \sigma'\}$ gives rise to a single factor $M^{2 \mu - 1}$. All of these observations follow easily from the preceding argument, as in Remark \ref{rem:small1}.
\end{remark}

\subsubsection*{Step 7. Summing over the $\zeta$-variables}
Our next goal is to sum over the variables $\f b_\zeta$ and $\f x_\zeta$ in \eqref{V_T_prime_5}. The basic philosophy is similar to that of Step 6: the restrictions in the summation arising from the constraint $\Gamma \in \fra F_\zeta(\Sigma, \f b_\zeta, \f 1)$ yield powers of $M^{-1}$, which will compensate the factor $M^{\mu \abs{\zeta}}$ arising from the $\f b_\zeta$-summation.

The key observation behind estimating the right-hand side of \eqref{V_T_prime_5} is that, by definition of $\fra F_\zeta(\Sigma, \f b_\zeta, \f 1)$, if $\Gamma \in \fra F_\zeta(\Sigma, \f b_\zeta, \f 1)$ then each ladder $L_\sigma$ with $\sigma \in \zeta$ contains a bridge $\pi \in L_\sigma$ that is contained in a block of $\Gamma$ of size greater than two.
Thus, for any $\f x_s$, $\f x_\zeta$, and $\Gamma$ for which the corresponding summand on the right-hand side of \eqref{V_T_prime_5} is nonzero, we have
\begin{equation} \label{1leqconn}
1 \;\leq\; \sum_{\alpha \subset \zeta} \sum_{\gamma \subset \zeta^2} \ind{\zeta = \alpha \cup [\gamma]} \prod_{\sigma \in \alpha} U_\sigma \prod_{\{\sigma, \sigma'\} \in \gamma} U_{\sigma, \sigma'}\,.
\end{equation}
The interpretation of the right-hand side of \eqref{1leqconn} is that each ladder $L_\sigma$ must contain a bridge that is either (i) in the same block as another bridge of $L_\sigma$ (implemented by the factor $U_\sigma$) or (ii) in the same block as a bridge of a different ladder $L_{\sigma'}$ (implemented by the factor $U_{\sigma, \sigma'}$). Plugging \eqref{1leqconn} into the right-hand side of \eqref{V_T_prime_5} and using \eqref{estimate for breaking lumps} yields
\begin{multline} \label{V_T_prime_6}
\abs{\cal V_\zeta'(\Sigma)} \;\leq\; C_\Sigma \, I_0(\f x_s)  \sum_{\alpha \subset \zeta} \sum_{\gamma \subset \zeta^2} \ind{\zeta = \alpha \cup [\gamma]} \, \sum_{\f x_s} \sum_{\f b_\zeta} \sum_{\f x_\zeta}
\pBB{\prod_{\sigma \in \alpha} U_\sigma} \pBB{\prod_{\{\sigma, \sigma'\} \in \gamma} U_{\sigma, \sigma'}}
\\
\times
\pBB{\prod_{\{e,e'\} \in L_\zeta(\Sigma, \f b_\zeta, \f 1)} J_{\{e,e'\}} S_{x_e}}
\,
\pBB{\prod_{\{e,e'\} \in L_{\bar \zeta}(\Sigma, \f 1)} J_{\{e,e'\}} \cal Z_{x_e}(\{e,e'\})}\,.
\end{multline}
We may now estimate the sum over $\f x_\zeta$ exactly as in Step 6, by estimating the indicator functions $U_\sigma$ and $U_{\sigma, \sigma'}$ 
as in \eqref{expU} and using the facts that the sums over $\alpha$ and $\gamma$ contain $O_\Sigma(1)$ terms and that for each $\alpha$ and $\gamma$ we have $\zeta = \alpha \cup [\gamma]$. The result is
\begin{align}
\abs{\cal V_\zeta'(\Sigma)} &\;\leq\; C_\Sigma \sum_{\f x_s} \,  I_0(\f x_s)  \sum_{\f b_\zeta}
\pBB{\prod_{\{e,e'\} \in L_\zeta(\Sigma, \f 1)} M^{-\mu} J_{\{e,e'\}} \cal Z_{x_e}(\{e,e'\})}
\,
\pBB{\prod_{\{e,e'\} \in L_{\bar \zeta}(\Sigma, \f 1)} J_{\{e,e'\}} \cal Z_{x_e}(\{e,e'\})}
\notag \\ \label{V_T_prime_7}
&\;\leq\; C_\Sigma \,   \sum_{\f x_s} I_0(\f x_s)
\pBB{\prod_{\{e,e'\} \in \Sigma} J_{\{e,e'\}} \cal Z_{x_e}(\{e,e'\})}\,,
\end{align}
where in the last step we used that $\cal G(\Sigma, \f 1) = \Sigma$.

\begin{remark} \label{rem:small3}
Exactly as in Remark \ref{rem:small2}, we note that if $\zeta \neq \emptyset$ the estimate \eqref{V_T_prime_7} is valid with an additional factor $M^{2 \mu - 1}$ on the right-hand side. 
\end{remark}

\subsubsection*{Step 8. Summing over $\f x_s$ and conclusion of the proof of Proposition \ref{prop:V_general_estimate}}

To conclude the estimate of $\cal V'_\zeta(\Sigma)$, we use the following estimate on skeleton pairings from \cite[Lemma 4.22]{EK3}.

\begin{lemma} \label{lem: gen_skeleton_sum}
Suppose that $\Sigma \notin \fra S_D$ and that $\cal Z$ satisfies Definition \ref{def_admissible_Z}. Then for small enough $\delta$ there exists a $c_0 > 0$ such that
\begin{equation*}
\sum_{\f x \in \bb T^{V(\Sigma)}} I_0(\f x)
\pBB{\prod_{\{e,e'\} \in \Sigma} J_{\{e,e'\}} \cal Z_{x_e}(\{e,e'\})} \;\leq\; \frac{C_\Sigma N}{M} R_2(\omega + \eta)  M^{-c_0}\,.
\end{equation*}
\end{lemma}

Now \eqref{V_T_prime_7} and Lemma \ref{lem: gen_skeleton_sum} yield
\begin{equation}
\abs{\cal V_\zeta'(\Sigma)} \;\leq\; \frac{C_\Sigma N}{M} R_2(\omega + \eta)  M^{-c_0}
\end{equation}
for all $\zeta \subset \Sigma$. 
An identical argument yields the same bound for $\abs{\cal V_\zeta''(\Sigma)}$ for all $\zeta \subset \Sigma$, where $\cal V_\zeta''(\Sigma)$ we defined after \eqref{def_V_T_prime}. The only difference is that $\chi_\sigma$ in \eqref{def_chi} is replaced with
\begin{equation*}
\chi_\sigma \;=\;
\begin{cases}
-\ee^{2 \ii A_1} & \text{if } \sigma \in \Sigma_1
\\
-\ee^{2 \ii A_2} & \text{if } \sigma \in \Sigma_2
\\
\ee^{\ii (A_1 + A_2)} & \text{if } \sigma \in \Sigma_c\,.
\end{cases}
\end{equation*}
The estimates are otherwise the same, since the resulting factors \eqref{def_wt_Z} again satisfy Definition \ref{def_admissible_Z}. (In fact, they satisfy even better bounds, since in \eqref{connecting_bound2} the argument $\omega + \eta$ is replaced with the larger constant $\kappa$). This concludes the estimate of $\wt {\cal V}_0(\Sigma)$ in \eqref{V_general_estimate}.

What remains is the estimate of $\wt {\cal V}_1(\Sigma)$. This is straightforward, since the set $\N^\Sigma \setminus B(\Sigma)$ is finite by Lemma \ref{lem: skeletons}, so that the $\f b$-summation in the definition of $\wt {\cal V}_1(\Sigma)$ ranges over a set of size $O(1)$. Hence we do not need to make use of oscillations, and may perform a brutal estimate by immediately taking the absolute value inside the summation. The result is that $\abs{\wt {\cal V}_1(\Sigma)}$ is bounded by the right-hand side of \eqref{V_T_prime_7} for some admissible matrices $\cal Z$;  we omit the details. Lemma \ref{lem: gen_skeleton_sum} then completes the estimate of $\wt {\cal V}_1(\Sigma)$. This concludes the proof of Proposition \ref{prop:V_general_estimate}.

\subsection{Computation of  $\wt{\cal V}(\Sigma)$  for $\Sigma \in \fra S_D$} \label{sec:computation of main term}
In this section we compute $\wt{\cal V}(\Sigma)$ for $\Sigma$ being one of the eight skeletons of $\fra S_D$, depicted in Figure \ref{fig: dumbbell}.

\begin{proposition} \label{prop:Vmain_gen}
For small enough $\delta$ in Proposition \ref{prop: expansion with trunction} there exists a constant $c_0 > 0$ such that 
\begin{equation*}
\sum_{\Sigma \in \fra S_D} \wt{\cal V}(\Sigma) \;=\; \cal V_{\rm main} + O\pbb{\frac{N}{M} M^{-c_0} R_2(\omega + \eta)}\,,
\end{equation*}
where $\cal V_{\rm main}$ was defined in \eqref{VDdef}.
\end{proposition}

The rest of this subsection is devoted to the proof of Proposition \ref{prop:Vmain_gen}.
The argument is very similar to that of Section \ref{sec_53}, and we shall therefore only highlight the differences.
As before, we split the right-hand side of \eqref{def_V_general} using \eqref{2re2re}, which gives $\wt {\cal V}(\Sigma) = 2 \re  (\wt{\cal V}\!\,'(\Sigma) + \wt{\cal V}\!\,''(\Sigma))$. We focus on the leading order term, $\wt{\cal V}\!\,'(\Sigma)$; the term $\wt{\cal V}\!\,''(\Sigma)$ may be dealt with in exactly the same way, and yields a smaller-order contribution. Exactly as in Step 3 of Section \ref{sec_53}, we split $\wt{\cal V}\!\,'(\Sigma)$ using $\zeta \subset \Sigma$ (see \eqref{V_o_V_zeta}), $\xi \subset \bar \zeta$ (see \eqref{Jxi}) and $\vartheta \subset \bar \xi$ (see \eqref{intr_vartheta}); this gives
\begin{equation*}
\wt{\cal V}\!\,'(\Sigma) \;=\; \sum_{\xi \subset \bar \zeta} \sum_{\vartheta \subset \bar \xi} \cal V'_{\zeta, \xi, \vartheta}(\Sigma)
\end{equation*}
in self-explanatory notation.
First, we observe that for the leading terms $\cal V'_{\emptyset,\emptyset,\emptyset}(\Sigma)$ we have
\begin{equation*}
\sum_{\Sigma \in \fra S_D} \cal V'_{\emptyset,\emptyset,\emptyset}(\Sigma) \;=\; 2 \re \cal V_{\rm main}' + O_q(NM^{-q})\,,
\end{equation*}
where $\cal V_{\rm main}'$ was introduced in \eqref{vdd}. Indeed, if $\zeta = \emptyset$, $\xi = \emptyset$, and $\vartheta = \emptyset$ then it is not hard to see that the $\f x$-summations of $\cal V'_{\zeta, \xi, \vartheta}(\Sigma)$ within all ladders of $\cal G(\Sigma, \f b)$ are unconstrained. This yields precisely $2 \re \cal V_{\rm main}'$ (see \eqref{VDdef} and \eqref{vdd}).

What remains is to estimate the error terms, i.e.\ $\cal V'_{\zeta, \xi, \vartheta}(\Sigma)$ for $\zeta \cup \xi \cup \vartheta \neq \emptyset$. The strategy is very similar to that of Section \ref{sec_53}: a nonempty $\zeta \cup \xi \cup \vartheta$ induces coincidences among the $\f x$-labels, which yields a subleading contribution even after taking absolute value inside summation. Here, however, the bounds \eqref{domestic_bound2} and \eqref{connecting_bound2} are not quite good enough:  as observed in \cite[Equation (4.58)]{EK3}, applying them to a dumbbell skeleton yields an error bound that is larger than the main term.  For $\zeta = \xi  = \vartheta = \emptyset$ the arguments of Section \ref{sec_53} yield the bound $\abs{\cal V'_{\emptyset, \emptyset, \emptyset}(D_i)} \leq (N/M) R_2(\omega + \eta) M^\mu$ for any skeleton $D_i$ (see \cite[Section 4.5]{EK3} for more details); this bound cannot be improved using the methods of Section \ref{sec_53}.  However, if $\zeta \cup \xi \cup \vartheta \neq \emptyset$ then the estimates of Section \ref{sec_53} get an additional factor $M^{2 \mu - 1}$ on the right-hand side. This fact was observed in Remarks \ref{rem:small1}--\ref{rem:small3}. 

\begin{figure}[ht!]
\begin{center}
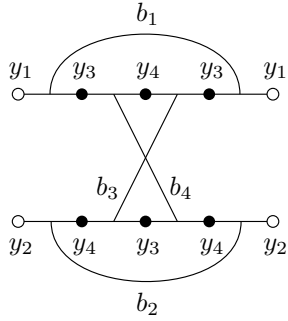
\end{center}
\caption{ The skeleton $D_8$. We indicate the four independent labels $y_1, \dots, y_4$ next to the vertices that carry them, and the multiplicities $b_1, \dots, b_4$ next to their associated bridges of $D_8$. \label{fig: D8}} 
\end{figure}

In order to derive the precise estimate, we consider only
the most complicated dumbbell skeleton, $\Sigma = D_8$, and we use the labelling from Figure \ref{fig: D8} to analyse $\cal V'_{\zeta, \xi, \vartheta}(D_8)$. In particular, we identify the bridges of $\Sigma$ with the set $\{1,2,3,4\}$. The key estimate is the following.
\begin{lemma} \label{lem:D8_error}
Suppose that $\zeta \cup \xi \cup \vartheta \neq \emptyset$. Then we have
\begin{equation*}
\abs{\cal V'_{\zeta, \xi, \vartheta}(D_8)} \;\leq\; M^{2 \mu - 1} \sum_{y_1, y_2, y_3, y_4} \abs{\cal Z_{y_1 y_3}(1)} \abs{\cal Z_{y_2 y_4}(2)} \abs{\cal Z_{y_4 y_3}(3)} \abs{\cal Z_{y_3 y_4}(4)}\,,
\end{equation*}
where $\cal Z(1)$ and $\cal Z(2)$ satisfy \eqref{domestic_bound2}, and $\cal Z(3)$ and $\cal Z(4)$ satisfy \eqref{connecting_bound2}.
\end{lemma}
\begin{proof}
The proof follows the argument of Section \ref{sec_53} to the letter, except that we need the squeeze out an extra factor $M^{2 \mu - 1}$. This extra factor is however already present in the estimates of Section \ref{sec_53}, as observed in Remarks \ref{rem:small1}--\ref{rem:small3}. 
\end{proof}

Using Lemma \ref{lem:D8_error}, it is easy to conclude the proof of Proposition \ref{prop:Vmain_gen}. Combining Lemma \ref{lem:D8_error} with \eqref{domestic_bound2} and \eqref{connecting_bound2} yields, for $\zeta \cup \xi \cup \vartheta \neq \emptyset$,
\begin{equation*}
\abs{\cal V'_{\zeta, \xi, \vartheta}(D_8)} \;\leq\; C N M^{2 \mu - 1} M^\mu (\log N)^2 \frac{M^{2 \delta}}{M} R_2(\omega + \eta) \;\leq\; \frac{C N}{M} \, M^{-c_0} R_2(\omega + \eta)\,,
\end{equation*}
for small enough $\delta > 0$ in Proposition \ref{prop: expansion with trunction}, since $3 \mu - 1 < 0$. Similar estimates hold for $\Sigma  = D_1, \dots, D_7$. This concludes the proof of Proposition \ref{prop:Vmain_gen}.

\subsection{Conclusion of the proof of Proposition \ref{prop: main} and Theorems \ref{thm: main result}--\ref{thm: Theta 1}} \label{sec:conclusion_proof}
The estimate \eqref{main_prop_error} follows from Propositions \ref{prop:large_Sigma_general}, \ref{prop:V_general_estimate}, and \ref{prop:Vmain_gen}. Recalling Propositions \ref{prop: leading term} and \ref{prop: leading term 2}, we have therefore proved Proposition \ref{prop: main}.

The rest of the proof is the same as in \cite[Section 4.7]{EK3}, and we summarize the argument. Using Proposition \ref{prop: main} and \eqref{F - wt F}, we have computed the numerator of \eqref{ThetatoF}. The denominator is easily computed from the following result, proved in \cite[Lemma 4.24]{EK3}.

\begin{lemma} \label{lem:EY}
For $E \in [-1+ \kappa, 1 - \kappa]$ we have
\begin{equation}\label{semicir}
\E \, Y^\eta_\phi(E) \;=\; 4 \sqrt{1 - E^2} + O(\eta) \; = \; 2\pi\nu(E) + O(\eta) \,.
\end{equation}
\end{lemma}
Now Theorems \ref{thm: main result}--\ref{thm: Theta 1} follow immediately with \eqref{D_const}, \eqref{Q_const} and
\begin{equation} \label{def:theta}
\Theta_{\phi_1,\phi_2}^\eta(E_1, E_2) \;\deq\; \frac{(LW)^d}{N^2} 
 \, \frac{\cal V_{\mathrm{main}}}{\E Y^\eta_{\phi_1}(E_1) \E Y^\eta_{\phi_2}(E_2)}\,.
\end{equation}

\section{Extensions} \label{sec:generalizations}

\subsection{Higher-order correlations}\label{sec:clt}

In this section we prove Theorem \ref{thm:clt}. Since the proof is almost identical to that of Theorem \ref{thm: main result}, we only outline the differences. In this section we use the notation $\bb M X \deq X - \E X$. 
 As in \eqref{def_F_eta},
 it suffices to compute
\begin{equation*}
F^\eta(E_1, \dots, E_k) \;\deq\; \E \pBB{\prod_{i = 1}^k \bb M \pB{\tr \phi_i^\eta (H/2 - E_i)}} \;=\; \wt F^\eta(E_1, \dots, E_k) + O_q(N^k M^{-q})
\end{equation*}
for all $q > 0$, where we defined
\begin{equation*}
\wt F^\eta(E_1, \dots, E_k) \;\deq\; \sum_{n_1 + \cdots + n_k \leq M^\mu} \pBB{\prod_{i = 1}^k 2 \re \pb{\wt \gamma_{n_2}(E_2, \phi_2)}} \E \pBB{\prod_{i = 1}^k \bb M (\tr H^{(n_i)})}\,.
\end{equation*}
(See \eqref{F - wt F}.) In order to evaluate the expectation on the right-hand side, we use a trivial extension of the graphical technology introduced in Section \ref{sec: part 2}. Here, the graph $\cal C \equiv \cal C(n_1, \dots, n_k) = \cal C_1(n_1) \sqcup \cdots \sqcup \cal C_k(n_k)$ consists of $k$ disjoint oriented chains, whereby the $i$-th chain $\cal C_i$ has $n_i$ edges. Plugging in \eqref{def: nb}, we get a sum over the labels $\f x = (x_i)_{i \in V(\cal C)}$.

We now proceed as in Section \ref{sec: part 2}, introducing the partitions $\Gamma \in \fra P(E(\cal C))$.
The only novel ingredient is a classifications of such partitions $\Gamma$ according to their induced partitions on the chains of $\cal C$, indexed by the set $\{1, \dots, k\}$. 
More precisely, for each $\Gamma \in \fra P(E(\cal C))$ we define the induced partition $P(\Gamma) \in \fra P(k) \deq \fra P(\{1, \dots, k\})$ as the finest partition on $\{1, \dots, k\}$ such that $i$ and $j$ belong to the same block of $P(\Gamma)$ if there exists a $\gamma \in \Gamma$ such that $i,j \in \gamma$. The interpretation of $P(\Gamma)$ is that components of $\cal C$ that are linked by a block of $\Gamma$ belong to the same block of $P(\Gamma)$. 
Previously, when computing the two-point correlation function we considered only partitions that linked the two chains of $\cal C$, so that $P(\Gamma)$ was always trivial.  Now we have several chains and their connectivity structure, described by $P(\Gamma)$, is more complicated.
We note that
$P(\Gamma)$ would have remained trivial if we had considered the $k$-th order cumulants instead of the correlation functions of $X_i$ in Theorem~\ref{thm:clt}. In our case, the subtraction of the expectation value in the definition of $X_i$ guarantees only that $P(\Gamma)$ has no atoms.

We write
\begin{equation*}
\E \pBB{\prod_{i = 1}^k \bb M (\tr H^{(n_i)})} \;=\; \sum_{\f x} I(\f x) \, \pBB{\prod_{e \in E(\cal C)} \sqrt{S_{x_e}}} \E \qBB{\prod_{i = 1}^k \bb M \pBB{\prod_{e \in E(\cal C_i)} A_{x_e}}}\,,
\end{equation*}
and decompose the expectation according to
\begin{equation} \label{A_clt}
\E \qBB{\prod_{i = 1}^k \bb M \pBB{\prod_{e \in E(\cal C_i)} A_{x_e}}} 
\;=\; \sum_{p \in \fra P^*(k)} \sum_{\Gamma \in \fra P(E(\cal C))} \ind{P(\Gamma) = p} \sum_{\Xi \in \fra H(\Gamma)} B_{\Gamma, \Xi}(\f x) \, D(\Gamma, \Xi)
\end{equation}
(compare with \eqref{A3}).
Here $\fra P^*(k) \subset \fra P(k)$ is the set of partitions with no atoms, $\fra H(\Gamma)$ is 
the set of halving partitions of $\Gamma$, and
\begin{equation} \label{def_D_clt}
D(\Gamma, \Xi) \;\deq\; \prod_{\gamma \in \Gamma} \mu_\gamma(\Xi) \prod_{i = 1}^k \pbb{1 - \prod_{\gamma \in \Gamma} \mu_{\gamma \cap E(\cal C_i)}(\Xi)}\,,
\end{equation}
which generalizes the indicator function \eqref{def_D_ind} to the case $k \geq 3$. In order to see this, we note (recalling that any monomial of in the entries of $A$ is either equal to one or has zero expectation) that the expectation on the left-hand side of \eqref{A_clt} is equal to one if and only if $\E \prod_{e \in E(\cal C_i)} A_{x_e} = 0$ for all $i$ and $\E \prod_{e \in E(\cal C)} A_{x_e} = 1$. These two conditions easily yield \eqref{def_D_clt}. Note that the restriction $p \in \fra P^*(k)$ imposes that $\Gamma$ has to be a connected partition in the sense that for any chain $\cal C_i$ there is another chain $\cal C_j$ such that some block
of $\Gamma$ has a nontrivial intersection with both $\cal C_i$ and $\cal C_j$. For the purposes of this proof,  this is the appropriate generalization
of the concept of connected partition  \eqref{conn_part}  to a $k$-component graph.
 Note also that, for $P(\Gamma) = p$ and $\Xi \in \fra H(\Gamma)$, the indicator function $D(\Gamma, \Xi)$ factorizes over the blocks of $p$.

Note that summation over $p$ on the right-hand side of \eqref{A_clt} yields the decomposition
\begin{equation} \label{F_split_p}
\wt F^\eta(E_1, \dots, E_k) \;=\; \sum_{p \in \fra P^*(k)} \wt F^\eta_p(E_1, \dots, E_k)
\end{equation}
in self-explanatory notation.
Now we split the right-hand side of \eqref{F_split_p} into the terms $p \in \fra M(k)$ (i.e.\ the pairings) and $p \in \fra P^*(k) \setminus \fra M(k)$. The former terms will yield the leading term of \eqref{clt_statement} and the latter terms the error term of \eqref{clt_statement}.

Let us consider an error term arising from a $p \in \fra P^*(k) \setminus \fra M(k)$. We may repeat the estimates of Section \ref{sec: part 2} and \cite[Section 4]{EK3} with merely cosmetic changes. Omitting the details, we get the following bounds. Each block $b$ of $p$ of size $\abs{b}$ yields a contribution bounded by
\begin{equation} \label{clt_main_estimate}
\frac{N}{M} M^{-\abs{b} + 2} (R_2(\omega_0 + \eta) M^\mu)^{\abs{b} - 1}
\end{equation}
to $\wt F^\eta(E_1, \dots, E_k)$. (Recall $\omega_0$ defined in the statement of Theorem \ref{thm:clt}.)  The bound \eqref{clt_main_estimate} may be seen from the power counting estimate \cite[Equation (4.57)]{EK3}, and we omit the details. Note that \eqref{clt_main_estimate} generalizes \cite[Equation (4.58)]{EK3} to the case $\abs{b} \geq 3$. In addition to \eqref{clt_main_estimate}, we have the stronger bound
\begin{equation} \label{clt_main_estimate2}
\frac{N}{M} R_4(\omega_0 + \eta)
\end{equation}
for blocks $b$ of size $\abs{b} = 2$, as follows by the explicit computation from 
 Proposition \ref{prop: main}. 
Again, we omit the details of the uninteresting modifications to the argument of Section \ref{sec: part 2} and \cite[Section 4]{EK3}.

Using \eqref{clt_main_estimate} and \eqref{clt_main_estimate2}, it is easy to conclude the proof. To illustrate the procedure, we give the details for the worst-case scenario: $k$ is odd, $p$ consists of one block of size three and $(k-3)/2$ blocks have size two. In that case we get from \eqref{clt_main_estimate} and \eqref{clt_main_estimate2} the bound
\begin{equation*}
\absb{\wt F^\eta_p(E_1, \dots, E_k)} \;\leq\; \frac{N}{M} M^{-1} (R_2(\omega_0 + \eta) M^\mu)^{2} \pbb{\frac{N}{M} R_4(\omega_0 + \eta)}^{(k-3)/2}
\;\leq\; \sqrt{\frac{M}{N}} \pbb{\frac{N}{M} R_4(\omega_0 + \eta)}^{k/2}\,,
\end{equation*}
where in the first step we used that $3 \mu < 1$ and $R_2(\omega_0 + \eta)^2 \leq M^\mu$. 
Essentially, a block of size greater than two yields an additional  small  factor bounded by  $(M/N)^{1/2}$.
This concludes the estimate of the error terms.

The main terms arise from the pairings $p \in \fra M(k)$ for even $k$. If $p$ is a pairing, the quantity $\wt F^\eta_p(E_1, \dots, E_k)$ is equal to a product over the blocks of $p$, up to some constraints among the summation labels imposed by $B_{\Gamma, \Xi}(\f x)$. These constraints may be removed exactly as in Section \ref{sec: part 2} at the expense of the error in \eqref{clt_statement}, and the result is a pure product over $k/2$ independent two-point functions. This concludes the proof of Theorem \ref{thm:clt}.

\subsection{The real symmetric case, $\beta = 1$} \label{sec:sym}
In this section we explain the changes needed to the arguments of Section \ref{sec: part 2} to prove Theorems \ref{thm: main result}, \ref{thm: Theta 2}, and \ref{thm: Theta 1} for $\beta = 1$ instead of $\beta = 2$. The necessary changes under the simplifications {\bf (S1)}--{\bf (S3)} of \cite{EK3} were explained in \cite[Section 5.1]{EK3}. Here we describe the changes required for the full proof, as presented in Section \ref{sec: part 2}.

The origin of the difference between the cases $\beta = 1$ and $\beta = 2$ is that for $\beta = 1$ we have $\E H_{xy}^2 = S_{xy}$, while for $\beta = 2$ we have $\E H_{xy}^2 = 0$ (in addition to $\E H_{xy} H_{yx} = \E \abs{H_{xy}}^2 = S_{xy}$, which is valid in both cases). This leads to additional terms for $\beta = 1$, which may be included in the argument of Section \ref{sec: part 2} as follows. Recall that the main idea of the partitions $\Gamma$ and $\Xi$ introduced in Section \ref{sec_41} is that blocks of $\Gamma$ correspond to edges that have the same unordered labels, and blocks of $\Xi$ to edges that have the same ordered labels. Since $\E A_{xy}^k A_{yx}^l$ is now nonzero if and only if $k + l$ is even, we find that the indicator function $\mu_\gamma(\Xi)$ defined in \eqref{def_mu_Xi} has to be replaced with
\begin{equation*}
\wt \mu_\gamma(\Xi) \;\deq\; \ind{\abs{\gamma} \text{ is even}} \, \ind{\Xi|_\gamma \text{ has at most two blocks}}\,.
\end{equation*}
This definition ensures that for $B_{\Gamma, \Xi}(\f x) = 1$ (recall \eqref{def:B}) we have $\E \prod_{e \in \gamma} A_{x_e} = 1$ if and only if $\wt \mu_\gamma(\Xi) = 1$. For a given $\Gamma$, we let $\Xi$ range over $\wt {\fra H}(\Gamma)$, defined as $\fra H(\Gamma)$ in \eqref{def_fra_H}
 except that $\mu_\gamma(\Xi)$ is replaced with $\wt \mu_\gamma(\Xi)$. With these definitions, the construction of Section \ref{sec_41} remains the same. In defining the refining pairing $\Pi$ of $(\Gamma, \Xi)$ as in Section \ref{sec:refining_partition}, we forgo the condition (a) of Section \ref{sec:refining_partition} by choosing the edge $e'$ in Step (ii) of $\Phi$ to be simply the largest edge of $\gamma$; otherwise the algorithm is unchanged.

This leads to pairings $\Pi$ whose two edges may either have different or coinciding ordered labels. We call the former \emph{twisted bridges} and the latter \emph{straight bridges}; these concepts were first introduced in Section 9 of \cite{EK1}, and were discussed in more detail in the current context in \cite[Section 5.1]{EK3}. Formally, we assign to each bridge of $\Pi$ a binary tag, which indicates whether the bridge is straight or twisted. The constraint imposed by a straight bridge $\{e,e'\}$ is given by the indicator function $J_{\{e,e'\}}(\f x)$ defined in \eqref{def_I_sigma}. Similarly, the constraint imposed by a twisted bridge $\{e,e'\}$ is given by the indicator function
\begin{equation*}
\wt J_{\{e,e'\}}(\f x) \;\deq\; \ind{x_e = x_{e'}} \;=\; \ind{x_{a(e)} = x_{a(e')}} \ind{x_{b(e')} = x_{b(e)}}\,.
\end{equation*}
By augmenting the pairings $\Pi$ to tagged pairings, the argument of Section \ref{sec: part 2} carries over easily.

The graphical representation and construction of the tagged skeleton $(\Sigma, \f b)$ from the tagged pairing $\Pi$ is covered in detail in \cite[Section 5.1]{EK3}. Here we give a short summary. Recall that the key observation behind the definition of a skeleton was that parallel straight bridges yield a large contribution but a small combinatorial complexity. Now \emph{antiparallel} twisted bridges behave analogously, whereby two bridges $\{e_1, e_1'\}$ and $\{e_2, e_2'\}$ are \emph{antiparallel} if $b(e_1) = a(e_2)$ and $b(e_1') = a(e_2')$.
(Recall that they are parallel if $b(e_1) = a(e_2)$ and $b(e_2') = a(e_1')$.) See Figure \ref{fig:twisted bridges} for an illustration. An \emph{antiladder} is a sequence of bridges such that two consecutive bridges are antiparallel. We represent straight bridges (as before) by solid lines and twisted bridges by dashed lines. 
\begin{figure}[ht!]
\begin{center}
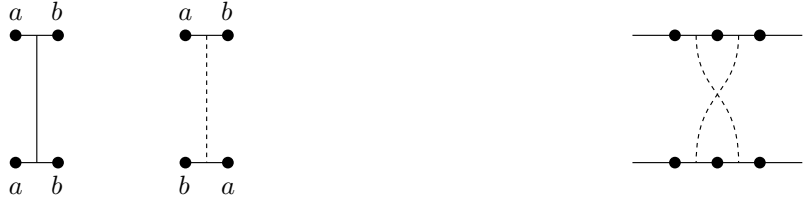
\end{center}
\caption{Left picture: a straight bridge (left) and a twisted bridge (right); labels with the same name are forced to coincide by the bridge. Right picture:
two antiparallel twisted bridges, which form an antiladder of size two.
\label{fig:twisted bridges}} 
\end{figure}

As in Section \ref{sec_42}, to each tagged pairing $\Gamma$ we assign a tagged skeleton $\Sigma$ with associated multiplicities $\f b$. The skeleton $\Sigma$ is obtained from $\Gamma$ by successively collapsing \emph{parallel straight bridges} and \emph{antiparallel twisted bridges} until none remains. 
We take over all notions from Section \ref{sec_42}, such as $\wt {\cal V}(\cdot)$, with the appropriate straightforward modifications for tagged skeletons.

As it turns out, allowing twisted bridges results in eight new dumbbell skeletons, called $\wt D_1, \dots, \wt D_8$ below, each of which has the same value $\wt{\cal V}(\cdot)$ as its counterpart without a tilde. They are defined in Figure \ref{fig:antidumbbell}.
\begin{figure}[ht!]
\begin{center}
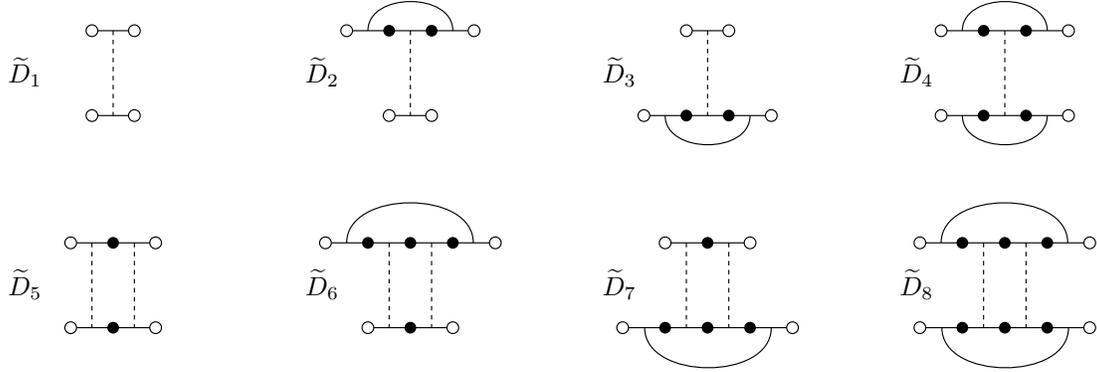
\end{center}
\caption{The eight dumbbell skeletons $\wt D_1, \dots, \wt D_8$ with twisted bridges. \label{fig:antidumbbell}} 
\end{figure}
Hence, for $\beta = 1$  the leading term  is simply twice the leading term of $\beta = 2$, 
which accounts for the trivial prefactor $2/\beta$ in the
final formulas.
Any other skeleton may be estimated by a trivial modification of the argument from Sections \ref{sec: part 2}; we omit further details.

\subsection{General band matrices} \label{sec:general_proof}

In this subsection we outline how to prove the results of Section \ref{sec:gen1} for general band matrices satisfying Definition \ref{def:gen_band}.  The main difference is that, since $H$ is not unimodular as in Section \ref{sec:setup1}, the key identity \eqref{H^n and U_n} does not hold. Instead, we have to correct the identity \eqref{H^n and U_n} with some error terms, thus effectively perturbing around \eqref{H^n and U_n}. This generalization was explained in detail in \cite[Section 6.1]{EK2}, and we summarize 
it here.

The error terms entering the general recursion relation are the random matrices $\Phi_2$ and $\Phi_3$, defined through
\begin{equation*}
(\Phi_2)_{xy} \;\deq\; \delta_{xy} 
\sum_z \pb{\abs{H_{xz}}^2 - S_{xz}}\,, \qquad
(\Phi_3)_{xy} \;\deq\; - \abs{H_{xy}}^2 H_{xy}\,.
\end{equation*}
We also introduce the notations
\begin{equation*}
(\ul{\Phi_3 H^{(n)}})_{x_0 x_{n + 1}} \;\deq\; \sum_{x_1, \dots, x_n} \qBB{\prod_{i = 0}^{n-1} 
\ind{x_i \neq x_{i + 2}}} \, (\Phi_3)_{x_0 x_1} H_{x_1 x_2} \cdots H_{x_n x_{n+1}}\,, \qquad \ul{\Phi_2 H^{(n)}} \;\deq\; \Phi_2 H^{(n)}\,.
\end{equation*}
Then \eqref{H^n and U_n} is replaced with 
\begin{equation} \label{H^n and U_n general}
U_n(H/2) \;=\; \sum_{k \geq 0} \sum_{a \in \{2,3\}^k} \; \sum_{\ell_0 + \cdots + \ell_k = n - \abs{a}} H^{(\ell_0)} \, 
\ul{\Phi_{a_1} H^{(\ell_1)}} \cdots \ul{\Phi_{a_k} H^{(\ell_k)}}\,,
\end{equation}
where the sum ranges over $\ell_i \geq 0$ for $i = 0, \dots, k$. Here we use the abbreviation 
$a = (a_1, \dots, a_k)$ as well as $\abs{a} \deq \sum_{i = 1}^k a_i$.
 For a proof of \eqref{H^n and U_n general}, see \cite[Proposition 6.2]{EK2}.
Note that in the unimodular case of Section \ref{sec:setup1} 
 we have $\ul{\Phi_2 H^{(n)}} = \Phi_2 H^{(n)}=0$ and $\ul{\Phi_3 H^{(n)}} = 
-\frac{1}{M-1}H^{(n-2)}$, so that \eqref{H^n and U_n general} reduces to \eqref{H^n and U_n}.

The leading term of \eqref{H^n and U_n general} is the term $k = 0$, which gives $U_n(H/2) = H^{(n)} + \cdots$. The error terms in $\cdots$ contain the matrices $\Phi_2$ and $\Phi_3$. The presence of either one of these factors leads to a subleading contribution. The basic idea why $\Phi_3$ is small is that it carries an extra factor $\abs{H_{xy}}^2 \leq M^{-1}$. The reason why $\Phi_2$ is small is that it has zero expectation. Proving that the terms $k \geq 1$ in \eqref{H^n and U_n general} yield a subleading contribution is a nontrivial extension of the analysis in the unimodular case,
and requires in particular the introduction of more complicated graphs and a more intricate analysis of the partitions induced on the edges of these graphs. The details for the case of quantum diffusion were carefully worked out in \cite{EK2}. The argument
 of \cite{EK2} may be taken over to the current setup.
Owing to the increased complexity of the error terms, the oscillations cannot be exploited as effectively as in Section \ref{sec: part 2}, but, following carefully the arguments of Section \ref{sec: part 2} combined with \cite[Sections 6--9]{EK2}, one finds  that the error estimates in Theorem \ref{thm: main result} remain valid
for any sufficiently small positive $\rho$. The only new ingredient needed in the proof is the observation that in addition to entries of $S$ arising from rungs $S_{xy} = \E H_{xy} H_{yx}$ of straight ladders, we get entries of $T$ arising from rungs $T_{xy} = \E H_{xy}^2$ of twisted antiladders (see Section \ref{sec:sym} and \cite[Sections 5.1]{EK3}). It is not hard to see that the estimates of Proposition \ref{prop: bounds on S} on the matrix $S$ remain true for the matrix $T$. (See Proposition \ref{prop: bounds on T} below.)

Having described the estimate of the error terms, we devote the rest of this subsection to the analysis of the main term $\Theta$, resulting from the dumbbell skeletons $D_1, \dots, D_8, \wt D_1, \dots, \wt D_8$. The contribution $\cal V_{\rm{main}}$ of the dumbbells $D_1, \dots, D_8$ was obtained in Section \ref{sec:computation of main term}, and was found to equal $\cal V_{\rm{main}}$ from \eqref{common expression for V(D)} up to negligible error terms. Similarly, the contribution of the skeletons $\wt D_1, \dots, \wt D_8$ is given by
\begin{multline} \label{V_main_wt}
\wt {\cal V}_{\rm{main}}
\;=\;
\sum_{b_1, b_2 = 0}^{\infty} \sum_{(b_3, b_4) \in \cal A} \, 2 \re \pb{\gamma_{2 b_1 + b_3 + b_4} * \psi_1^\eta}(E_1) \, 2 \re \pb{\gamma_{2 b_2 + b_3 + b_4} * \psi^\eta_2}(E_2) \, \cal I^{b_1 + b_2} \tr T^{b_3 + b_4}
\\
+ O_q(N M^{-q})\,,
\end{multline}
Note that the only difference between \eqref{V_main_wt} and \eqref{common expression for V(D)} is that the factor $\tr S^{b_3 + b_4}$ of \eqref{common expression for V(D)} was replaced with $\tr T^{b_3 + b_4}$ in \eqref{V_main_wt}. Indeed, as explained in Section \ref{sec:sym} and \cite[Sections 5.1]{EK3}, each twisted bridge gives rise to an entry of $T$ just as each straight bridge gives rise to an entry of $S$. As in \eqref{def:theta}, the leading term $\Theta$ is given by
\begin{equation} \label{def:theta_gen}
\Theta_{\phi_1,\phi_2}^\eta(E_1, E_2) \;=\; \frac{(WL)^d}{N^2} \, \frac{\cal V_{\mathrm{main}} + \wt {\cal V}_{\mathrm{main}}}{\E Y^\eta_{\phi_1}(E_1) \E Y^\eta_{\phi_2}(E_2)}\,.
\end{equation}
The asymptotics of $\cal V_{\rm{main}}$ were worked out in Proposition \ref{prop: main} (i)--(iii).

What therefore remains is an asymptotic analysis of $\wt {\cal V}_{\rm{main}}$ from \eqref{V_main_wt}. In order to give it, we introduce the quantities
\begin{equation} \label{def_Delta}
\Delta \;\deq\; \frac{1}{2} \sum_{x \in \bb T}  g \Big(\frac{x}{W} \Big)^2 S_{x0} - \abs{D^{-1/2} w}^2\,, \qquad w \;\deq\; \frac{1}{2} \sum_{x \in \bb T} \frac{x}{W} g \Big(\frac{x}{W} \Big) S_{x0}\,,
\end{equation}
as well as 
\begin{equation} \label{def_Upsilon}
\Upsilon \;\deq\; \sum_{x \in \bb T} h \Big(\frac{x}{W} \Big)  S_{x0}\,.
\end{equation}
Recalling the definitions \eqref{def_Delta_0} and \eqref{def_Upsilon_0}, it is not hard to see that
\begin{equation} \label{lim_Delta}
\Delta \;=\; \Delta_0 + O(W^{-1})\,, \qquad
\Upsilon \;=\; \Upsilon_0 + O(W^{-1})\,.
\end{equation}
In particular, $\Delta$ and $\Upsilon$ are bounded from below by some positive constant since $\Delta_0$ and $\Upsilon_0$ are.
Moreover, we set
\begin{equation} \label{def_sigma_wt}
\wt \sigma \;\deq\; \Delta \lambda^2 + \Upsilon \varphi\,,
\end{equation}
which is analogous to $\sigma$ from \eqref{def_sigma}.  Using $\Delta_0 > 0$ and $\Upsilon_0 > 0$, we therefore have
\begin{equation} \label{sigma_wt_sigma}
\wt \sigma \;=\; \sigma (1 + O(W^{-1}))\,.
\end{equation}

The following result gives the asymptotic behaviour of $\wt {\cal V}_{\rm{main}}$, in analogy to (i)--(iii) for $\cal V_{\rm{main}}$ from Proposition \ref{prop: main}.

\begin{proposition}[Asymptotics of $\wt {\cal V}_{\rm{main}}$] \label{prop: main T}
The quantity $\wt {\cal V}_{\rm{main}}$ from \eqref{V_main_wt} satisfies the following estimates.
\begin{enumerate}
\item
Suppose that \eqref{eta_Delta_2} holds. Then for $d = 1,2,3$  we have
\begin{equation} \label{V3C2 T}
\wt {\cal V}_{\rm{main}} \;=\;
\frac{ (2/\pi)^{d/2} B_d}{\nu(E)^2 \sqrt{\det D}}\pbb{\frac{L}{2 \pi W}}^d \pBB{2 \re \pbb{\frac{\pi \wt \sigma}{2} + \ii \frac{\omega}{\nu(E)}}^{d/2 - 2} + O \pB{ (\omega + \wt \sigma)^{d/2 - 2} \pb{\sqrt{\omega + \wt \sigma} + M^{-\tau/2}}}}\,.
\end{equation}
Moreover, for $d = 4$ we have
\begin{equation} \label{V4C2 T}
\wt {\cal V}_{\rm{main}} \;=\; \frac{8}{\nu(E)^2\sqrt{\det D}} \pbb{\frac{L}{2 \pi  W}}^d
 \pB{\min\{\abs{\log \omega}, \abs{\log \wt \sigma}\}  + O(1)}\,.
\end{equation}
\item
Suppose that \eqref{eta_Delta_2} holds and that $d = 2$. If $\phi_1$ and $\phi_2$ satisfy \textbf{(C1)} then
\begin{equation} \label{V2C1 T}
\wt {\cal V}_{\rm{main}} \;=\; \frac{8}{\pi\nu(E)^2 \sqrt{\det D}} \pbb{\frac{L}{2 \pi W}}^2  \pBB{\frac{4 \pi\eta \nu(E) + \pi^2 \nu(E)^2 \wt \sigma} {4 \omega^2 + (4 \eta + \pi \nu(E) \wt \sigma)^2} +
 \p{Q - 1} \min\{\abs{\log \omega}, \abs{\log \wt \sigma}\} + O(1)}\,,
\end{equation}
and if $\phi_1$ and $\phi_2$ satisfy \textbf{(C2)} then
\begin{equation} \label{V2C2 T}
\wt {\cal V}_{\rm{main}} \;=\; \frac{8}{\pi\nu(E)^2 \sqrt{\det D}} \pbb{\frac{L}{2 \pi W}}^2  \pBB{\frac{\pi^2 \nu(E)^2 \wt \sigma} {4 \omega^2 + (\pi \nu(E) \wt \sigma)^2} +
 \p{Q - 1} \min\{\abs{\log \omega}, \abs{\log \wt \sigma}\} + O(1)}\,.
\end{equation}
\item 
Suppose that $\omega = 0$. Then the exponent $\mu$ from Proposition \ref{prop: expansion with trunction} may be chosen so that there exists an exponent $c_1 > 0$ such that for $d =1,2,3$ we have
\begin{equation} \label{V3C2D0 T}
\wt {\cal V}_{\rm{main}} \;=\; \frac{2^{d/2}}{ \nu(E)^2\sqrt{\det D}} \,  \pbb{\frac{L}{2 \pi W}}^d
\pbb{\frac{\eta}{\nu(E)}}^{d/2 - 2} \pBB{ V_d\pbb{\phi_1, \phi_2; \frac{2 \wt \sigma}{\pi \nu(E) \eta}}  + O(M^{-c_1})}
\end{equation}
(recall the definition \eqref{def_V_d_a})
and for $d = 4$ we have
\begin{equation} \label{V4C2D0 T}
\wt {\cal V}_{\rm{main}} \;=\; \frac{4}{\nu(E)^2 \sqrt{\det D}} \pbb{\frac{L}{2 \pi W}}^4 \, 
\pB{  V_4(\phi_1, \phi_2) \min\{\abs{\log \eta}, \abs{\log \wt \sigma}\} + O (1)}\,.
\end{equation}
\end{enumerate}
\end{proposition}

Once Proposition \ref{prop: main T} is proved, Theorems \ref{thm: Theta 2 gen} and \ref{thm: Theta 1 gen} easily follow using \eqref{def:theta_gen}, \eqref{lim_Delta}, \eqref{def_sigma_wt},  \eqref{sigma_wt_sigma}, Lemma \ref{lem:EY}, and Proposition \ref{prop: main} (i)--(iii).

The rest of this subsection is devoted to the proof of Proposition \ref{prop: main T}. The argument is similar to the computation of $\cal V_{\rm{main}}$ in the proof of Proposition \ref{prop: main}. It relies on the following result, which is the generalization of Proposition \ref{prop: bounds on S} to the case $\wt \sigma > 0$.
\begin{proposition} \label{prop: bounds on T}
Let $T$ be as in \eqref{general T} and $\alpha \in \C$ satisfy $\abs{\alpha} \leq 1$ and $\abs{1 - \alpha} \geq 4 / M + (W/L)^2$. Then all estimates of Proposition \ref{prop: bounds on S} remain valid provided we replace $S$ with $T$ and $\alpha$ 
with $\alpha - \wt \sigma$. (In particular, $u\geq0$ and $\zeta\in \bb S^1$ are the polar coordinates of
$1-\alpha +\wt \sigma = u\zeta$.)  All constants depend in addition on $g$ and $h$.
\end{proposition}
\begin{proof}
See Appendix \ref{appendix: T}.
\end{proof}

The following result is the generalization of Proposition \ref{prop:S_int} to the case $\wt \sigma > 0$.
\begin{proposition} \label{prop:T_int}
Suppose that \eqref{LW_assump} holds. Let $b > 0$ be fixed and $\cal J \deq 1 - M^{-c_2} \eta$ for some $c_2 > 0$.  Fix a smooth real function $e \in L^1(\R)$ 
 satisfying the condition \textbf{(C2)} (see \eqref{non_Cauchy}), and recall the notation  $e^\eta(v) = \eta^{-1}e(\eta^{-1}v)$.
Then for $d \leq 3$ we have
\begin{multline} \label{intR3S}
\frac{1}{2 \pi} \int \dd v \, e^\eta(v) \tr \frac{T}{\p{1 + \ii b v -\cal JT}^2} \;=\; \frac{(b \eta)^{d/2 - 2}}{\sqrt{\det D}} \pbb{\frac{L}{2 \sqrt{\pi} W}}^d \int_0^\infty \dd t \, \ee^{-t \wt \sigma / (\eta b)} \, t^{1-d/2} \, \ol{\wh e(t)}
\\
+ O \pbb{\frac{N}{M} R_4(\eta) M^{-c_0}}
\end{multline}
for some constant $c_0 > 0$, and for $d = 4$ we have
\begin{equation} \label{intR4S}
\frac{1}{2 \pi} \int \dd v \, e^\eta(v) \tr \frac{T}{\p{1 + \ii b v -\cal JT}^2} \;=\;
\frac{\min\h{\abs{\log \eta}, \abs{\log \wt \sigma}}}{\sqrt{\det D}} \pbb{\frac{L}{2 \sqrt{\pi} W}}^4 \, \ol{\wh e(0)}  + O \pbb{\frac{N}{M}}\,.
\end{equation}
\end{proposition}
\begin{proof}
See Appendix \ref{appendix: T}.
\end{proof}

Armed with Propositions \ref{prop: bounds on T} and \ref{prop:T_int}, we may complete the proof of Proposition \ref{prop: main T} by following the arguments of Sections \ref{sec:dumbbell_C1} and \ref{sec:dumbbell_C2} almost to the letter; we omit the uninteresting details.

\section{Beyond the diffusive regime} \label{sec:mean-field}

Throughout this paper we made the assumption \eqref{AS_regime} that we are in the diffusive regime; see Remark \ref{rem:regimes} for a more detailed discussion. While \eqref{AS_regime} is physically important and necessary for computing the asymptotics of the leading
terms (Theorems \ref{thm: Theta 2} and \ref{thm: Theta 1}),
 it is not fundamental for the proof of Theorem \ref{thm: main result}. We now restate
this theorem without the lower  bound  in \eqref{LW_assump}
(which implied \eqref{AS_regime}) and explain the proof.

\begin{theorem} \label{thm:main result mean-field}
Fix $\rho \in (0, 1/3)$ and $d \in \N$, and set $\eta \deq M^{-\rho}$. Suppose that the test functions $\phi_1$ and $\phi_2$ satisfy either both \textbf{(C1)} or both \textbf{(C2)}.
Suppose moreover that
\begin{equation}
W \;\leq\; L \;\leq\; W^C
\end{equation}
for some constant $C$.

Then there exists a constant $c_0 > 0$ such that, for any $E_1, E_2$ satisfying \eqref{D leq kappa} for small enough $c_* > 0$,  the local density-density correlation satisfies
\begin{equation}
\frac{\avg{Y^\eta_{\phi_1}(E_1) \, ; Y^\eta_{\phi_2}(E_2)}}{\avg{Y^\eta_{\phi_1}(E_1)} \avg{Y^\eta_{\phi_2}(E_2)}} \;=\; \frac{1}{(LW)^d} \pBB{\Theta_{\phi_1,\phi_2}^\eta(E_1,E_2) + M^{-c_0} O \pbb{R_2(\omega + \eta) + \frac{M}{N (\omega + \eta)}}} \,,
\end{equation}
where $\Theta$ is defined in \eqref{def:theta} with $\cal V_{\text{main}}$ defined in \eqref{VDdef}.
\end{theorem}

\begin{proof}
The proof follows that of Theorem \ref{thm: main result} to the letter. Recall that the key input for the proof of Theorem \ref{thm: main result} is Proposition \ref{prop: bounds on S} (i). (The other parts of Proposition \ref{prop: bounds on S} are only needed for Theorems \ref{thm: Theta 2} and \ref{thm: Theta 1}). With the assumption $\abs{1 - \alpha} \geq 4 / M + (W/L)^2$ of Proposition \ref{prop: bounds on S} relaxed to $\abs{1 - \alpha} \geq 4 / M$, the estimate \eqref{bound res 1} remains true, and the estimate \eqref{bound res 2} is replaced with
\begin{equation} \label{bound res 2 mean-field}
\sup_{x,y} \absbb{\pbb{\frac{S}{(1 - \alpha S)^k}}_{xy}} \;\leq\; \frac{C}{M} R_{2k}(\abs{1 - \alpha}) + \frac{C}{N \abs{1 - \alpha}^k}\,.
\end{equation}
The proof of \eqref{bound res 2 mean-field} follows that of \eqref{bound res 2} up to \eqref{bound res 2 proof}. Since the lattice spacing of the Riemann sum is no longer smaller than one,  the bound of the right-hand side of \eqref{bound res 2 proof} requires, in addition to the integral approximation, the contribution of the origin $r = 0$. This yields \eqref{bound res 2 mean-field}.

Using \eqref{bound res 2 mean-field} instead of \eqref{bound res 2}, it is easy to conclude the proof of Theorem \ref{thm:main result mean-field}.
\end{proof}

Having established Theorem \ref{thm:main result mean-field}, what remains is the asymptotic analysis of $\Theta$, i.e.\ of \eqref{VDdef}. This was done in the diffusive regime \eqref{AS_regime} in Section \ref{sec:3}, using Proposition \ref{prop: bounds on S} (ii) and (iii) as input; the result was given in Theorems \ref{thm: Theta 2} and \ref{thm: Theta 1}. If \eqref{AS_regime} does not hold, this analysis, along with \eqref{asymptotics of trace d leq 3}--\eqref{precise asymptotics d=2}, has to be modified. We omit the details.

We conclude this section by noting that, instead of correlation functions of real linear statistics of the form \eqref{def_Y}, we may also consider correlation functions of Green functions and even at spectral parameters with different imaginary parts.
 The estimate of the subleading graphs is unchanged. In order to compute the contribution of the leading dumbbell graphs, instead of the right-hand side of \eqref{2re2re}, we have to compute $x_1 \ol x_2$ or $x_1 x_2$ (in the notation of \eqref{2re2re}) for the case {\bf (C1)}. Thus we get, for instance, for 
$\eta_1, \eta_2\geq W^{-\rho/3}$ with $\rho \in (0,1/3)$, that
\begin{equation}
\avgB{\tr \pb{H/2 - E_1 - \ii \eta_1}^{-1} \,; \tr \pb{H/2 - E_2 + \ii \eta_2}^{-1}} \;=\; \wh {\cal V}_{\text{main}} +  O \pbb{\frac{N}{M}R_2(\wh \omega) + \frac{1}{\wh \omega}}
\end{equation}
and
\begin{equation}
\avgB{\tr \pb{H/2 - E_1 - \ii \eta_1}^{-1} \,; \tr \pb{H/2 - E_2 - \ii \eta_2}^{-1}} \;=\;   O \pbb{\frac{N}{M}R_2(\wh \omega) + \frac{1}{\wh \omega}}\,,
\end{equation}
where we defined $\wh \omega \deq \omega + \min\{\eta_1, \eta_2\}$, and the leading term is given by
\begin{equation}
\wh {\cal V}_{\text{main}} \;\deq\; T(E_1) \ol{T(E_2)} \, \frac{\ee^{\ii A_1^{\eta_1}}}{1 + \ee^{2 \ii A_1^{\eta_1}} \cal I} \, \frac{\ee^{-\ii \ol A_2^{\eta_2}}}{1 + \ee^{-2 \ii \ol A_2^{\eta_2}} \cal I}
\tr \frac{\ee^{\ii (A_1^{\eta_1} - \ol A_2^{\eta_2})} S}{\pb{1 - \ee^{\ii (A_1^{\eta_1} - \ol A_2^{\eta_2})} S}^2}\,.
\end{equation}
We again omit the asymptotic analysis of $\wh {\cal V}_{\text{main}}$.

\appendix

\section{Proof of Proposition \ref{prop: expansion with trunction}} \label{app: path expansion}

Using the identity
\begin{equation*}
\phi^\eta(H/2 - E) \;=\; 2 \re \int_0^\infty \dd t \, \wh \phi(\eta t) \, \ee^{\ii t E} \ee^{-\ii t H/2}
\end{equation*}
 (recall that $\phi$ is real), we get from \eqref{def_F_eta}
\begin{equation*}
F^\eta(E_1,E_2) \;=\; \avgbb{ 2 \re \tr \int_0^\infty \dd t \, \wh \phi_1(\eta t) \, \ee^{-\ii (H/2) t+ \ii E_1 t} \,; 2 \re \tr \int_0^\infty \dd t \,  \wh \phi_2(\eta t) \, \ee^{-\ii (H/2) t + \ii E_2 t}}\,.
\end{equation*}

We now truncate in the norm of $H$. Fix $\delta$ satisfying $0 < 2 \delta < \mu - \rho$. From \cite{EK2}, Proposition 5.4, we find that that there exists an $\epsilon> 0$ and an event $\Xi$ (both depending on $\delta$) such that
\begin{equation} \label{properties of Xi}
\ind{\Xi} \, \norm{H} \;\leq\; M^\delta \,, \qquad \P(\Xi^c) \;\leq\; M^{-\epsilon M}\,.
\end{equation}
Using $\int_0^\infty \dd t \,  \wh \phi_i(\eta t) = O(\eta^{-1}) = O(M^\rho)$ for $i = 1,2$, we get
\begin{multline*}
F^\eta(E_1,E_2) \;=\; \avgbb{\ind{\Xi} \, 2 \re \tr \int_0^\infty \dd t \, \wh \phi_1(\eta t) \, \ee^{-\ii (H/2) t+ \ii E_1 t} \,; \ind{\Xi} \, 2 \re  \tr \int_0^\infty \dd t \,  \wh \phi_2(\eta t) \, \ee^{-\ii (H/2) t + \ii E_2 t}}
\\
+ O(N^2 M^{2 \rho -\epsilon M})\,.
\end{multline*}

Next, we truncate in $t$.
Since $\phi_1$ and $\phi_2$ are smooth,
for any $q > 0$ there is a $C_q$ such that $\abs{\wh \phi_i(t)} \leq C_q \, t^{-q}$ for $i = 1,2$. We conclude that
\begin{multline*}
F^\eta(E_1,E_2)
\;=\;
\avgbb{\ind{\Xi} \, 2 \re \tr \int_0^{M^{\rho + \delta}} \dd t \, \wh \phi_1(\eta t) \, \ee^{-\ii (H/2) t+ \ii E_1 t} \,; \ind{\Xi}\, 2 \re \tr \int_0^{M^{\rho + \delta}} \dd t \,  \wh \phi_2(\eta t) \, \ee^{-\ii (H/2) t + \ii E_2 t}}
\\
+ O_q \pb{N^2 M^{2 \rho - \delta(q - 1)}}
\end{multline*}
for all $q > 0$.
Recalling \eqref{definition of gamma tilde}, we therefore get using \eqref{Chebyshev expansion of propagator}
\begin{multline} \label{F eta truncated}
F^\eta(E_1,E_2)
\\
=\; \sum_{n_1,n_2 = 0}^\infty 2 \re \pb{\wt \gamma_{n_1}(E_1, \phi_1)} \, 2 \re \pb{\wt \gamma_{n_2}(E_2, \phi_2)} \, \avgb{\ind{\Xi} \, \tr H^{(n_1)}\,; \ind{\Xi} \, \tr H^{(n_2)}}
+ O_q \pb{N^2 M^{2 \rho - \delta(q - 1)}}\,.
\end{multline}

Next, we truncate in $n_1$ and $n_2$. From \eqref{bounds on a_n} and $\abs{\wh \phi_i(t)} \leq 1$ we find that
\begin{equation} \label{estimate of gamma gamma}
\absb{\wt \gamma_{n_1}(E_1, \phi_1) \wt \gamma_{n_2}(E_2,\phi_2)} \;\leq\; C\frac{M^{(\rho + \delta)(n_1 + 1)}}{n_1!} \frac{M^{(\rho + \delta) (n_2 + 1)}}{n_2!} \;\leq\; C \frac{(2 M^{\rho + \delta})^{n_1 + n_2 + 2}}{(n_1 + n_2)!}
\,.
\end{equation}
In order to estimate the contribution of the terms $n_1 + n_2 \geq M^\mu$ in \eqref{F eta truncated}, we need the following rough bound on Chebyshev polynomials, proved in \cite{EK2}, Lemma 5.5:
\begin{equation} \label{rough bound on Un}
\abs{U_n(E/2)} \;\leq\; C^n (1 + \abs{E})^n\,.
\end{equation}
Using \eqref{H^n and U_n} we therefore get
\begin{equation*}
\ind{\Xi} \norm{H^{(n)}} \;\leq\; (C M^\delta)^n\,.
\end{equation*}
Recalling \eqref{estimate of gamma gamma}, we have the bound
\begin{multline*}
\absBB{\sum_{n_1+n_2 > M^\mu} 2 \re \pb{\wt \gamma_{n_1}(E_1,\phi_1)} \, 2 \re \pb{\wt \gamma_{n_2}(E_2,\phi_2)} \, \avgb{\ind{\Xi} \, \tr H^{(n_1)}\,; \ind{\Xi} \, \tr H^{(n_2)}}}
\\
\leq\; N^2 M^{2 \mu} \sum_{n_1 + n_2 > M^\mu} \pbb{\frac{C M^{\rho+ \delta}}{n_1 + n_2}}^{n_1 + n_2} (C M^\delta)^{n_1 + n_2}
\;\leq\; N^2 M^{2 \mu} \sum_{n_1 + n_2 > M^\mu} \pB{C M^{\rho + 2 \delta - \mu}}^{n_1 + n_2} \;\leq\; N^2 \exp(- M^\mu)\,.
\end{multline*}

Summarizing, we have proved that
\begin{multline*} 
F^\eta(E_1,E_2) \;=\; \sum_{n_1 + n_2 \leq M^\mu} 2 \re \pb{\wt \gamma_{n_1}(E_1,\phi_1)} \, 2 \re \pb{\wt \gamma_{n_2}(E_2,\phi_2)} \avgb{\ind{\Xi} \, \tr H^{(n_1)}\,; \ind{\Xi} \, \tr H^{(n_2)}}
\\
+ O_q \pb{N^2 M^{2 \rho - \delta(q - 1)}}\,.
\end{multline*}
Next, we remove the indicator function $\ind{\Xi}$. We use \eqref{rough bound on Un} with the deterministic bound
\begin{equation*}
\norm{H} \;\leq\; \max_x \sum_y \abs{H_{xy}} \;\leq\; 2 \sqrt{M}
\end{equation*}
to get
\begin{multline*}
\absBB{\sum_{n_1 + n_2 \leq M^\mu} 2 \re \pb{\wt \gamma_{n_1}(E_1,\phi_1)} \, 2 \re \pb{\wt \gamma_{n_2}(E_2,\phi_2)} \, \avgb{\ind{\Xi^c} \, \tr H^{(n_1)}\,; \tr H^{(n_2)}} }
\\
\leq\; N^2 M^{4 \mu} \pb{C \sqrt{M}}^{M^\mu} \P(\Xi^c) \;\leq\; N^2 \exp(-M^c)\,,
\end{multline*}
where we used the trivial bound $\abs{\wt \gamma_n(E,\phi)} \leq C M^{\mu}$ and \eqref{properties of Xi}. This 
 (together with analogous estimates for the other error terms) 
 proves \eqref{F - wt F}, after a renaming of $q$. The claim about 
the improved bound in \eqref{F - wt F} in the case $\phi$ is analytic follows in exactly the same way.

Moreover, \eqref{gamma - g} follows easily from \eqref{definition of gamma tilde}, \eqref{claim about gamma}, and the estimate $\abs{\wh \phi_i(t)} \leq C_q t^{-q}$, which yield
\begin{equation*}
\absb{(\gamma_n * \psi_i^\eta)(E_i) - \wt \gamma_n(E_i,\phi_i)} \;\leq\; \int_{M^{\rho + \delta}}^\infty C_q (t \eta)^{-q} \;\leq\; C_q M^\rho M^{-\delta (q - 1)}\,.
\end{equation*}

What remains is the proof of \eqref{bound on gamma}. We use the bound \eqref{bounds on a_n} and the identity \eqref{claim about gamma} to get, for large enough $K$,
\begin{multline*}
\absb{(\psi^\eta * \gamma_n)(E)} \;\leq\; \int_0^{n / K} \dd t \, \absb{\wh \phi(\eta t) a_n(t)} + \int_{n/K}^{\infty} \dd t \, \absb{\wh \phi(\eta t) a_n(t)}
\;\leq\; \int_0^{n / K} \dd t \, \pbb{\frac{C t}{n}}^n + \int_{n/K}^{\infty} \dd t \, \absb{\wh \phi(\eta t)}
\\
\leq\; \ee^{-n} + C_q (\eta n)^{-q}\,,
\end{multline*}
by the decay of $\wh \phi$. Moreover, the bound 
$\absb{(\psi^\eta * \gamma_n)(E)} \leq C$ follows from \eqref{claim about gamma n}. The estimate of $\wt \gamma(E,\phi)$ is similar, using \eqref{gamma - g}. This concludes the proof of Proposition \ref{prop: expansion with trunction}.

\section{Proofs of Propositions \ref{prop: bounds on S} and \ref{prop:S_int}} \label{appendix: S}

In this appendix we establish resolvent estimates and asymptotics for the matrix $S$, and give the proof of Propositions \ref{prop: bounds on S} and \ref{prop:S_int}. 
We use the discrete Fourier transform. To that end, we define the lattice
\begin{equation*}
\bb T^* \;\deq\; \frac{2 \pi}{L} \bb T \;=\; \pbb{[-\pi, \pi)\cap \frac{2\pi\Z}{L}}^d
\end{equation*}
dual to $\bb T$, so that
\begin{equation*}
S_{xy} \;=\; \frac{1}{N} \sum_{p \in \bb T^*} \ee^{\ii p \cdot (x - y)} \wh S(p)\,, \qquad
\wh S(p) \;\deq\; \sum_{x \in \bb T} \ee^{-\ii p \cdot x} S_{x0}\,.
\end{equation*}
Note that $\wh S(p)$ is naturally defined for all $p \in [-\pi,\pi)^d$.
For $q \in [-\pi W, \pi W)^d$, define
\begin{equation} \label{def_S_W}
\wh S_W(q) \;\deq\; \wh S(q/W) \;=\; \sum_{x \in \bb T} \ee^{-\ii q \cdot x / W} \frac{1}{M - 1} f \pbb{\frac{x}{W}} \;=\; \frac{1}{M - 1} \sum_{v \in W^{-1} \bb T} \ee^{-\ii q \cdot v} f(v)\,,
\end{equation}
where in the second equality we used that $[x]_L = x$ for $x \in \bb T$. For the following we  recall the definition \eqref{def_I} of $\cal I$.

\begin{lemma} \label{lem: basic properties of SW}
The function $\wh S_W$ is real and symmetric, and satisfies $\abs{\wh S_W(q)} \leq \cal I$. Moreover, for any $\epsilon > 0$ there is a $\delta_\epsilon > 0$ such that
\begin{equation} \label{SW prop 1}
\absb{\wh S_W(q)} \;\leq\; 1 - \delta_\epsilon \qquad \text{if} \qquad \abs{q} \;\geq\; \epsilon
\end{equation}
for large enough $W$ (depending on $\epsilon$). Finally, we have the expansion
\begin{equation} \label{SW exp}
\wh S_W(q) \;=\; \cal I - q \cdot D q + \cal Q(q) + O(\abs{q}^{4 + c})\,,
\end{equation}
where $c$ is the constant from \eqref{moment of f}. Here $D$ is the covariance matrix defined in \eqref{definition of D} and
\begin{equation*}
\cal Q(q) \;\deq\; \frac{1}{4!} \sum_{x \in \bb T} (x \cdot q / W)^4 S_{x0}\,.
\end{equation*}
\end{lemma}

An immediate consequence of Lemma \ref{lem: basic properties of SW} is the bound, valid for $n \geq 2$,
\begin{equation} \label{tr_Sn_bound}
\tr S^n \;=\; \sum_{p \in \bb T^*} \wh S(p)^n \;\leq\; \cal I^{n - 2} \sum_{p \in \bb T^*}\absb{\wh S(p)}^2 \;=\; \cal I^{n - 2} N \sum_{x \in \bb T} S_{x0}^2 \;\leq\; C \cal I^n N/M\,,
\end{equation}
where in the third step we used Parseval's identity. 

\begin{proof}[Proof of Lemma \ref{lem: basic properties of SW}]
The bound $\abs{\wh S_W(q)} \leq \cal I$ is trivial
 and that $\wh S_W$ is real and symmetric follows immediately from the fact that $f$ is real and symmetric. Next, \eqref{SW prop 1} follows from the identity
\begin{equation*}
\cal I - \wh S_W(q) \;=\; \frac{1}{M - 1} \sum_{u \in W^{-1} \bb T} \pb{1 - \cos(q \cdot u)} f(u)\,;
\end{equation*}
we omit the details.
Finally, \eqref{SW exp} follows by applying Taylor's theorem with a fourth order remainder, combined with the estimates \eqref{moment of f} and $\abs{\ee^{\ii t} - 1} \leq 2 \abs{t}^c$ for any $c \in (0,1)$. 
\end{proof}

\begin{proof}[Proof of Proposition \ref{prop: bounds on S}]
The proof of \eqref{bound res 1} is almost identical to the proof of \cite[Proposition A.2 (ii)]{EKYY4}, using \cite[Proposition A.3 (ii)]{EKYY4}. We omit the details.

In order to prove \eqref{bound res 2}, by translation invariance of $S$ we may set $y = 0$. To begin with, we note that the case $\abs{\alpha} \leq 1/2$ is trivial by the identity $\frac{S}{(1 - \alpha S)^k} = S \pb{\sum_{n = 0}^\infty \alpha^n S^{n}}^k$ combined with $S^{n}_{xy} \leq \cal I^{n} M^{-1}$. Throughout the following we therefore assume that $\abs{\alpha} \geq 1/2$.

The proofs of \eqref{bound res 2}, \eqref{asymptotics of trace d leq 3}, and \eqref{asymptotics of trace d 4} rely on the Fourier space representation
\begin{equation} \label{resolvent_fspace}
\pbb{\frac{S}{(1 - \alpha S)^k}}_{x0} \;=\; \frac{1}{N} \sum_{p \in \bb T^*} \ee^{\ii p \cdot x} \frac{\wh S(p)}{(1 - \alpha \wh S(p))^k} \;=\;
\frac{1}{N} \sum_{q \in W \bb T^*} \ee^{\ii q \cdot x/W} \frac{\wh S_W(q)}{(1 - \alpha \wh S_W(q))^k}\,.
\end{equation}
Before giving the full proofs, we give a short overview of the argument. Approximating the summation in \eqref{resolvent_fspace} with a integral, we get
\begin{equation*}
\pbb{\frac{S}{(1 - \alpha S)^k}}_{x0} \;\approx\;
\frac{1}{(2 \pi W)^d} \int_{[-\pi W,\pi W]^d} \dd q \, \ee^{\ii q \cdot x/W} \frac{\wh S_W(q)}{(1 - \alpha \wh S_W(q))^k}\,.
\end{equation*}
Next, we approximate $\wh S_W(q) \approx 1 - q \cdot D q$ (see \eqref{SW exp}); this approximation is only valid for small $q$, but for $d \leq 4$ the main contribution to the integral comes precisely from small values of $q$. This yields, for $\alpha \approx 1$,
\begin{align*}
\pbb{\frac{S}{(1 - \alpha S)^k}}_{x0} &\;\approx\;
\frac{1}{(2 \pi W)^d} \int_{[-\pi W,\pi W]^d} \dd q \, \ee^{\ii q \cdot x/W} \frac{1}{(1 - \alpha (1 - q \cdot D q))^k}
\\
&\;\approx\;
\frac{1}{(2 \pi W)^d} \int_{[-\pi W,\pi W]^d} \dd q \, \ee^{\ii q \cdot x/W} \frac{1}{(1 - \alpha + q \cdot D q)^k}\,.
\end{align*}
By assumption on $\alpha$, we have $\re (1 - \alpha) > 0$. Now \eqref{bound res 2} will follow using $D \geq c$ and \eqref{use_of_R}. In order to prove \eqref{asymptotics of trace d leq 3} and \eqref{asymptotics of trace d 4}, we write
\begin{equation*}
\pbb{\frac{S}{(1 - \alpha S)^2}}_{00} \;\approx\;
\frac{1}{(2 \pi W)^d} \int_{[-\pi W,\pi W]^d} \dd q \, \frac{1}{(u \zeta + q \cdot D q)^k} \;=\;
\frac{u^{d/2 - 2}}{(2 \pi W)^d} \int_{[-\pi u^{-1/2} W,\pi u^{-1/2} W]^d} \dd q \, \frac{1}{(\zeta + q \cdot D q)^2}\,,
\end{equation*}
from which \eqref{asymptotics of trace d leq 3} and \eqref{asymptotics of trace d 4} will easily follow 
using the elementary identity
\begin{equation} \label{computation_V_d}
\int_{\R^d} \dd r \, \frac{1}{(\zeta +  \abs{r}^2)^2} \;=\; B_d \, \zeta^{d/2 - 2}\,.
\end{equation}

Now we give the full proof of \eqref{bound res 2},
\eqref{asymptotics of trace d leq 3}, and \eqref{asymptotics of trace d 4}. Let $\epsilon > 0$ to be chosen later, and split the summation on the right-hand side of \eqref{resolvent_fspace} according to
\begin{equation} \label{main splitting for frac S}
\pbb{\frac{S}{(1 - \alpha S)^k}}_{x0} \;=\;
\frac{1}{N} \sum_{q \in W \bb T^*} \ind{\abs{q} \leq \epsilon} \, \ee^{\ii q \cdot x/W} \frac{\wh S_W(q)}{(1 - \alpha \wh S_W(q))^k}
+
\frac{1}{N} \sum_{q \in W \bb T^*} \ind{\abs{q} > \epsilon} \, \ee^{\ii q \cdot x/W} \frac{\wh S_W(q)}{(1 - \alpha \wh S_W(q))^k}.
\end{equation}
Using \eqref{SW prop 1}, 
we write the second term of \eqref{main splitting for frac S} as
\begin{equation} \label{split large momenta}
\frac{1}{N} \sum_{q \in W \bb T^*} \ind{\abs{q} > \epsilon} \pbb{\ee^{\ii q \cdot x/W}  \wh S_W(q) + O \pbb{\frac{\abs{\wh S_W(q)}^2}{\delta^k_\epsilon}}}\,.
\end{equation}
We write the first term of \eqref{split large momenta} as
\begin{multline*}
\frac{1}{N} \sum_{q \in W \bb T^*} \ind{\abs{q} > \epsilon} \, \ee^{\ii q \cdot x/W} \wh S_W(q) \;=\; \frac{1}{N} \sum_{q \in W \bb T^*} \, \ee^{\ii q \cdot x/W} \wh S_W(q) - \frac{1}{N} \sum_{q \in W \bb T^*} \ind{\abs{q} \leq \epsilon} \, \ee^{\ii q \cdot x/W} \wh S_W(q)
\\
=\; S_{x0} + O\pbb{\frac{1}{M} \frac{M}{N} \sum_{q \in W \bb T^*} \ind{\abs{q} \leq \epsilon}} \;=\; O(M^{-1})\,.
\end{multline*}
Moreover, we estimate the second term of \eqref{split large momenta} by
\begin{equation*}
\frac{C}{\delta^k_\epsilon} \frac{1}{N} \sum_{q \in W \bb T^*} \abs{\wh S_W(q)}^2 \;=\; \frac{C}{\delta^k_\epsilon} \sum_{x \in \bb T} S_{x0}^2 \;\leq\; \frac{C}{\delta^k_\epsilon M}\,,
\end{equation*}
where in the first step we used Parseval's identity. We conclude that the second term of \eqref{main splitting for frac S} is bounded by $C \delta_\epsilon^{-k} M^{-1}$.

In order to estimate the first term of \eqref{main splitting for frac S}, we use \eqref{SW exp} and \eqref{D_bounded}
to choose $\epsilon > 0$ such that $\wh S_W(q) = \cal I - a(q)$ with
\begin{equation*}
c \abs{q}^2 \;\leq\; a(q) \;\leq\; C \abs{q}^2 \;\leq\; 1 \qquad \text{for} \qquad \abs{q} \;\leq\; \epsilon\,,
\end{equation*}
for some $c > 0$. 
Thus we get
\begin{equation} \label{small momenta 1}
\absBB{\frac{1}{N} \sum_{q \in W \bb T^*} \ind{\abs{q} \leq \epsilon} \, \ee^{\ii q \cdot x/W} \frac{\wh S_W(q)}{(1 - \alpha \wh S_W(q))^k}}
\;\leq\;
\frac{2}{N} \sum_{q \in W \bb T^*} \frac{\ind{\abs{q} \leq \epsilon}}{\absb{1 - \alpha (\cal I - a(q))}^k}\,.
\end{equation}
Consider first the case $\re \alpha < 0$. In that case the right-hand side of \eqref{small momenta 1} is bounded by
\begin{equation*}
\frac{2}{M} \frac{M}{N} \sum_{q \in W \bb T^*} \ind{\abs{q} \leq \epsilon} \;\leq\; \frac{C}{M}\,.
\end{equation*}
For the following we therefore assume that $\re \alpha > 0$, as well as $\abs{\alpha} \geq 1/2$.
Defining the polar variables
\begin{equation} \label{def_zeta_t}
t \;\deq\; \absbb{\frac{1 - \alpha}{\alpha}} \,, \qquad \xi \;\deq\; \frac{1 - \alpha}{t \alpha}
\end{equation}
and writing $q = t^{1/2} r$, we may estimate the right-hand side of \eqref{small momenta 1} by
\begin{equation} \label{small momenta 2}
\frac{2}{N \abs{\alpha}^k t^k} \sum_{r \in W t^{-1/2} \bb T^*} \frac{\ind{\abs{r} \leq \epsilon t^{-1/2}}}{\absb{\xi + (1 - \cal I) t^{-1} + t^{-1} a(t^{1/2} r)}^k}\,.
\end{equation}
An elementary estimate shows that
\begin{equation} \label{partition for alpha}
\h{\alpha \in \ol{\bb D} \col \re \alpha \geq 0} \;\subset\; \hbb{\alpha \in \ol{\bb D} \col \re \frac{1 - \alpha}{\alpha} \geq 0} \cup \hbb{\alpha \in \ol{\bb D} \col \absbb{\im \frac{1 - \alpha}{\alpha}} \geq \absbb{\re \frac{1 - \alpha}{\alpha}}}\,,
\end{equation}
where $\ol{\bb D}$ denotes the closed unit disc (the first set is the disc $\abs{\alpha - 1/2} \leq 1/2$, while the second one contains the union of the two discs $\abs{\alpha - (1/2 \pm i/2)}\leq 1/\sqrt{2}$ in the complementary regime $\abs{\alpha - 1/2}> 1/2$, which
together clearly cover the left-hand side). Using that $\abs{\xi}=1$, we conclude that for each $\alpha \in \ol{\bb D}$ with $\re \alpha > 0$ we have
$\abs{\im \xi} > 1/2$ or $\re \xi > 1/2$. Since  $-1/4 \leq (1 - \cal I) t^{-1} \leq 0$ by assumption on $\alpha$, we find that \eqref{small momenta 2} is bounded by
\begin{equation} \label{bound res 2 proof}
\frac{2}{N \abs{\alpha}^k t^k} \sum_{r \in W t^{-1/2} \bb T^*} \frac{\ind{\abs{r} \leq \epsilon t^{-1/2}}}{\absb{1/8 +  t^{-1} a(t^{1/2} r)}^k} \;\leq\; 
\frac{Ct^{d/2 - k}}{M}  \frac{M}{N t^{d/2}} \sum_{r \in W t^{-1/2} \bb T^*} \frac{\ind{\abs{r} \leq \epsilon t^{-1/2}}}{(1 + \abs{r}^2)^k}\,.
\end{equation}
Now \eqref{bound res 2} follows easily using a Riemann sum estimate, \eqref{use_of_R},
 and the fact that $c (W/L)^2 < t \leq 4$.

In order to prove \eqref{asymptotics of trace d leq 3} and \eqref{asymptotics of trace d 4}, we write
\begin{equation} \label{R_sum_tr_S}
\tr \frac{S}{\p{1 - \alpha S}^2} \;=\; \sum_{q \in W \bb T^*} \frac{\wh S_W(q)}{(1 - \alpha \wh S_W(q))^2}\,.
\end{equation}
Exactly as in the proof of \eqref{bound res 2} (see \eqref{main splitting for frac S} and the rest of its paragraph), we find that
\begin{equation} \label{tr S 1 - alpha S 1}
\tr \frac{S}{\p{1 - \alpha S}^2} \;=\; \sum_{q \in W \bb T^*} \frac{\wh S_W(q)}{(1 - \alpha \wh S_W(q))^2} \ind{\abs{q} \leq \epsilon} + O \pbb{\frac{N}{\delta_\epsilon^2 M}}\,.
\end{equation}
Writing $\wh S_W(q) = \cal I - a(q)$ as above and recalling the notation $u = \abs{1 - \alpha}$, we find, repeating the estimates from the proof of \eqref{bound res 2} and using $a(q) \leq C \abs{q}^2$, 
\begin{equation} \label{tr S 1 - alpha S 2}
\sum_{q \in W \bb T^*} \frac{\wh S_W(q)}{(1 - \alpha \wh S_W(q))^2} \ind{\abs{q} \leq \epsilon} \;=\; \cal I
\sum_{q \in W \bb T^*} \frac{\ind{\abs{q} \leq \epsilon}}{(1 - \alpha \wh S_W(q))^2} + O \pbb{\frac{N}{M} R_2(u)}\,.
\end{equation}
Using \eqref{SW exp} we get
\begin{equation*}
1 - \alpha \wh S_W(q) \;=\; 1 - \alpha + q \cdot D q + O(M^{-1} + u \abs{q}^2 + \abs{q}^4)\,.
\end{equation*}
A simple estimate using Riemann sums, as in the previous paragraph, therefore yields
\begin{equation} \label{tr S 1 - alpha S 3}
\sum_{q \in W \bb T^*} \frac{\ind{\abs{q} \leq \epsilon}}{(1 - \alpha \wh S_W(q))^2} \;=\;
\sum_{q \in W \bb T^*} \frac{\ind{\abs{q} \leq \epsilon}}{(1 - \alpha + q \cdot D q)^2} + \frac{N}{M} O \pbb{R_2(u) + \frac{u^{d/2 - 2}}{M u}}\,.
\end{equation}

At this point we differentiate between the cases $d \leq 3$ and $d = 4$. Suppose first that $d \leq 3$.
We define
\begin{equation*}
v(r) \;\deq\; \frac{1}{(\zeta + r \cdot D r)^2}\,,
\end{equation*}
where $1 - \alpha = u \zeta$.
Writing $q = u^{1/2} r$, we get
\begin{equation} \label{Riemann_approx}
\sum_{q \in W \bb T^*} \frac{\ind{\abs{q} \leq \epsilon}}{(1 - \alpha + q \cdot D q)^2} \;=\; \frac{1}{u^2}
\sum_{r \in W u^{-1/2} \bb T^*} \ind{\abs{r} \leq \epsilon u^{-1/2}} v(r) \;=\; \frac{1}{u^2}
\sum_{r \in \Lambda^*} v(r) + O(1)\,,
\end{equation}
where we introduced the infinite lattice $\Lambda^* \deq 2 \pi u^{-1/2} W L^{-1} \bb Z^d$, which
is the dual lattice of $\Lambda = u^{1/2} LW^{-1/2} \bb Z^d$.
In the last step we used a Riemann sum estimate using the fact that the spacing of $\Lambda^*$ is bounded from above
by some positive constant (since $\abs{1 - \alpha} \geq (W/L)^2$). We denote by $\wh v(y) \deq \int_{\R^d}
 \ee^{-\ii r \cdot y} v(r) \, \dd r$ the Fourier transform of $v$. Using the Poisson summation 
formula we therefore find
\begin{align*}
\sum_{q \in W \bb T^*} \frac{\ind{\abs{q} \leq \epsilon}}{(1 - \alpha + q \cdot D q)^2} &\;=\; \frac{1}{u^2}
 \pbb{\frac{L u^{1/2}}{2 \pi W}}^d \sum_{y \in \Lambda} 
 \wh v(y) + O(1)
\\
&\;=\; \frac{1}{u^2} \pbb{\frac{L u^{1/2}}{2 \pi W}}^d \int_{\R^d} v(r) \, \dd r + O \pb{1 + \ee^{-c L W^{-1} u^{1/2}}}\,,
\end{align*}
for some positive constant $c$ depending only on $D$. Here 
we used that $v$ is analytic in a neighbourhood of the real axis that is uniform in $\zeta$ by assumption on $\alpha$. Therefore $\wh v(y)$ is exponentially small in the lattice constant of $\Lambda$  for all $y\in \Lambda\setminus \{ 0\}$.
We may now get rid of the constraint $\abs{q} \leq \epsilon$ to get
\begin{equation*}
\sum_{q \in W \bb T^*} \frac{\ind{\abs{q} \leq \epsilon}}{(1 - \alpha + q \cdot D q)^2} \;=\; \frac{1}{u^2} \pbb{\frac{L u^{1/2}}{2 \pi W}}^d \pbb{\int_{\R^d} \frac{1}{(\zeta + r \cdot D r)^2} \, \dd r + O \pbb{\ee^{-c L W^{-1} u^{1/2}} + u^{2 - d/2}}}\,.
\end{equation*}
Recalling \eqref{computation_V_d}, \eqref{tr S 1 - alpha S 1}, \eqref{tr S 1 - alpha S 2}, and \eqref{tr S 1 - alpha S 3}, we therefore find
\begin{multline*}
\tr \frac{S}{\p{1 - \alpha S}^2}
\\
=\; \frac{u^{d/2 - 2}}{\sqrt{\det D}} \pbb{\frac{L}{2 \pi W}}^d  \qBB{B_d \, \zeta^{d/2 - 2} + O \pbb{\ee^{-c L W^{-1} u^{1/2}} + \frac{1}{M u} + u^{2 - d/2} + \ind{d = 1} u + \ind{d = 2} u \abs{\log u}}}\,.
\end{multline*}
This concludes the proof of \eqref{asymptotics of trace d leq 3}.

Suppose now that $d = 4$. Writing $q = u^{1/2} r$, we get
\begin{align*}
\sum_{q \in W \bb T^*} \frac{\ind{\abs{q} \leq \epsilon}}{(1 - \alpha + q \cdot D q)^2} &\;=\;
\frac{1}{u^2} \sum_{r \in W t^{-1/2} \bb T^*} \frac{\ind{\abs{r} \leq \epsilon u^{-1/2}}}{(\zeta + r \cdot D r)^2}
\\
&\;=\;
\pbb{\frac{L}{2 \pi  W}}^4 \pbb{\int_{\R^4} \dd r \, \frac{\ind{\abs{r} \leq \epsilon u^{-1/2}}}{(\zeta + r \cdot D r)^2} + O(1)}\,,
\end{align*}
Using that $W L^{-1} u^{-1/2} \leq C$ by assumption on $\alpha$, we get
\begin{align*}
\sum_{q \in W \bb T^*} \frac{\ind{\abs{q} \leq \epsilon}}{(1 - \alpha + q \cdot D q)^2}
&\;=\; \pbb{\frac{L}{2 \pi  W}}^4 \pbb{\int_{\R^4} \dd r \, \frac{\ind{1 \leq \abs{r} \leq u^{-1/2}}}{(\zeta + r \cdot D r)^2} + O(1)}
\\
&\;=\; \pbb{\frac{L}{2 \pi  W}}^4 \pbb{\int_{\R^4} \dd r \, \frac{\ind{1 \leq \abs{r} \leq u^{-1/2}}}{(r \cdot D r)^2} + O(1)}\,,
\end{align*}
where in the second step we used that $\abs{\zeta} = 1$  together with $\re \zeta \geq 0$. 
The change of variables $r = D^{-1/2} p$ and the estimate
\begin{equation} \label{ellipse}
\absB{\indb{1 \leq \abs{D^{-1/2} p}} - \indb{1 \leq \abs{p}}} \;\leq\; \indB{\norm{D^{-1}}^{-1/2} \leq \abs{p} \leq \norm{D}^{1/2}}
\end{equation}
(together with a similar estimate to translate the upper bound $\abs{r} \leq u^{-1/2}$ into
  $\abs{p} \leq u^{-1/2}$)
yield, after a short calculation,
\begin{equation*}
\sum_{q \in W \bb T^*} \frac{\ind{\abs{q} \leq \epsilon}}{(1 - \alpha + q \cdot D q)^2} \;=\; \frac{1}{\sqrt{\det D}} \pbb{\frac{L}{2 \pi  W}}^4  \pbb{\int_{\R^4} \dd p \, \frac{\ind{1 \leq \abs{p} \leq u^{-1/2}}}{\abs{p}^4} + O(1)}\,.
\end{equation*}
Recalling \eqref{tr S 1 - alpha S 1}, \eqref{tr S 1 - alpha S 2}, and \eqref{tr S 1 - alpha S 3}, we find that \eqref{asymptotics of trace d 4} follows.

What remains is the proof of \eqref{precise asymptotics d=2}. Let therefore $d=2$. 
 The proof is very similar to that of \eqref{asymptotics of trace d leq 3} given above, except that we have to keep one more term  from  the expansion \eqref{SW exp}. 
In order to convey its ideas more clearly, we unburden the notation by explaining the computation in terms of the integral
\begin{equation*}
\pbb{\frac{L}{2 \pi W}}^2 \int_{\R^2} \dd q \, \frac{\wh S_W(q)}{(1 - \alpha \wh S_W(q))^2}
\end{equation*}
instead of the Riemann sum \eqref{R_sum_tr_S}; the Riemann sum approximation may be controlled using Poisson summation, exactly as after \eqref{Riemann_approx} above. Throughout the argument we tacitly use \eqref{SW exp}, in which we replace the factor $\cal I$ with $1$; the resulting error is a distraction of order $(Mu)^{-1}$ and may be controlled as above. For small enough $\epsilon$ we get
\begin{equation*}
\int_{\R^2} \dd q \, \frac{\wh S_W(q)}{(1 - \alpha \wh S_W(q))^2} \;=\; \int_{\abs{q} \leq \epsilon} \dd q \, \frac{\wh S_W(q)}{(1 - \alpha \wh S_W(q))^2} + O(1)
\;=\;
\int_{\abs{q} \leq \epsilon} \dd q \, \frac{1 - q \cdot D q}{(1 - \alpha \wh S_W(q))^2} + O(1)\,.
\end{equation*}
We write
\begin{equation*}
1 - \alpha \wh S_W(q) \;=\; 1 - \alpha + q \cdot D q - \cal Q(q) + O\pb{\abs{1 - \alpha} \, \abs{q}^2 + \abs{q}^{4 + c}}\,,
\end{equation*}
which yields
\begin{equation*}
\frac{1}{(1 - \alpha \wh S_W(q))^2} \;=\; \frac{1}{(1 - \alpha + q \cdot D q - \cal Q(q))^2} + O\pbb{\frac{(\abs{1 - \alpha} \, \abs{q}^2 + \abs{q}^{4 + c})(\abs{1 - \alpha} + \abs{q}^2)}{(1 - \alpha \wh S_W(q))^2 (1 - \alpha + q \cdot D q - \cal Q(q))^2}}\,.
\end{equation*}
A routine estimate similar to the one yielding \eqref{tr S 1 - alpha S 3} above therefore gives
\begin{multline} \label{int_q_d2}
\int_{\R^2} \dd q \, \frac{\wh S_W(q)}{(1 - \alpha \wh S_W(q))^2} \;=\; \int_{\abs{q} \leq \epsilon} \dd q \, \frac{1 - q \cdot D q}{(1 - \alpha + q \cdot D q - \cal Q(q))^2} + O(1)
\\
=\; \int_{\R^2} \dd q \, \frac{1}{(1 - \alpha + q \cdot D q)^2}
- \int_{\abs{q} \leq \epsilon} \dd q \, \frac{q \cdot D q}{(1 - \alpha + q \cdot D q)^2}
+ \int_{\abs{q} \leq \epsilon} \dd q \, \frac{2 \cal Q(q)}{(1 - \alpha + q \cdot D q)^3}
 + O(1)\,.
\end{multline}
The first term of \eqref{int_q_d2} is
\begin{equation*}
\int_{\R^2} \dd q \, \frac{1}{(u \zeta + q \cdot D q)^2} \;=\; \frac{1}{\sqrt{\det D}}\, \frac{\pi}{u \zeta}\,.
\end{equation*}
To compute the second term of \eqref{int_q_d2}, we introduce the variable $p$ through $q = u^{1/2} D^{-1/2} p$. Using \eqref{ellipse}, we get
\begin{equation*}
- \int_{\abs{q} \leq \epsilon} \dd q \, \frac{q \cdot D q}{(1 - \alpha + q \cdot D q)^2} \;=\;
- \frac{1}{\sqrt{\det D}} \int_{\abs{p} \leq \epsilon u^{-1/2}} \dd p \, \frac{\abs{p}^2}{(\zeta + \abs{p}^2)^2} + O(1) \;=\; - \frac{\pi \abs{\log u}}{\sqrt{\det D}} + O(1)\,.
\end{equation*}
Finally, the third term of \eqref{int_q_d2} is
\begin{align*}
\int_{\abs{q} \leq \epsilon} \dd q \, \frac{2 \cal Q(q)}{(1 - \alpha + q \cdot D q)^3} &\;=\; \frac{1}{12 W^4} \sum_{x \in \bb T} S_{x0} \int_{\abs{q} \leq \epsilon} \dd q \, \frac{(x \cdot q)^4}{(1 - \alpha + q \cdot D q)^3}
\\
&\;=\; \frac{1}{12 W^4 \sqrt{\det D}} \sum_{x \in \bb T} S_{x0} \int_{1 \leq \abs{p} \leq u^{-1/2}} \dd p \, \frac{(D^{-1/2} x \cdot p)^4}{\abs{p}^6} + O(1)
\\
&\;=\; \frac{\pi}{32 W^4 \sqrt{\det D}} \sum_{x \in \bb T} S_{x0} \abs{D^{-1/2} x}^4 \abs{\log u} + O(1)\,.
\end{align*}
This concludes the proof of \eqref{precise asymptotics d=2}.
\end{proof}

\begin{proof}[Proof of Proposition \ref{prop:S_int}]
The proof relies on the Fourier space representation \eqref{resolvent_fspace}. Before giving the full proof, we sketch the calculation
 for $d \leq 3$.  Approximating the Riemann sum from \eqref{resolvent_fspace} with an integral and noting that, as in the proof of Proposition \ref{prop: bounds on S}, the leading contribution to the integral comes from the singularity at $q = 0$, we find
\begin{equation*}
\tr \frac{S}{\p{1 + \ii b v -\cal JS}^2} \;\approx\; \pbb{\frac{L}{2 \pi W}}^d \int_{\R^d} \dd q \, \frac{1}{\pb{1 - \cal J + \ii bv + q \cdot D q}^2}\,.
\end{equation*}
Using the identity
\begin{equation} \label{FT_identity}
\frac{1}{2 \pi} \int_\R \dd v \, e^\eta(v) \, \frac{1}{(x + \ii v)^2} \;=\; \int_0^\infty \dd t \, \ee^{-xt} \, t \, \ol{\wh e(\eta t)}\,, 
\end{equation}
valid for $x \geq 0$,
we therefore get
\begin{align*}
\frac{1}{2 \pi} \int_\R \dd v \, e^\eta(v) \tr \frac{S}{\p{1 + \ii b v -\cal JS}^2} &\;\approx\; \pbb{\frac{L}{2 \pi W}}^d \int_{\R^d}
 \dd q \, \frac{1}{2 \pi} \int_\R \dd v \, e^\eta(v) \, \frac{1}{\pb{1 - \cal J + \ii bv + q \cdot D q}^2}
\\
&\;=\; b^{- 2} \pbb{\frac{L}{2 \pi W}}^d \int_{\R^d} \dd q \, \int_0^\infty \dd t \, \ee^{- (1 - \cal J) t/b} \, \ee^{-t q \cdot D q /b} \, t \, \ol{\wh e(\eta t)}
\\
&\;\approx\; \frac{\pi^{d/2} b^{d/2 - 2}}{\sqrt{\det D}}\pbb{\frac{L}{2 \pi W}}^d \int_0^\infty \dd t \, t^{1-d/2} \, \ol{\wh e(\eta t)}\,,
\end{align*}
where in the third step we used that $1 - \cal J \ll \eta$. The claim will now follow by the change of variables $\eta t \mapsto t$.

Now we give the full proof of Proposition \ref{prop:S_int}.
Similarly to \eqref{resolvent_fspace} we get
\begin{equation*}
\tr \frac{S}{\p{1 + \ii b v -\cal JS}^2} \;=\; \sum_{q \in W \bb T^*} \frac{\wh S_W(q)}{\pb{1 + \ii b v - \cal J \wh S_W(q)}^2}
\;=\; \sum_{q \in W \bb T^*} \frac{\wh S_W(q)}{\pb{1 + \ii b v - \cal J \wh S_W(q)}^2} \ind{\abs{q} \leq \epsilon} + O \pbb{\frac{N}{\delta_\epsilon^2 M}}
\,,
\end{equation*}
where the last step holds for all $\epsilon > 0$, exactly as in \eqref{main splitting for frac S} and \eqref{tr S 1 - alpha S 1} above. Next, following the argument from \eqref{tr S 1 - alpha S 1} to \eqref{tr S 1 - alpha S 3} almost to the letter, we get, in analogy to \eqref{tr S 1 - alpha S 3},
\begin{align}
\tr \frac{S}{\p{1 + \ii b v -\cal JS}^2} &\;=\; \sum_{q \in W \bb T^*} \frac{\wh S_W(q)}{\pb{1 + \ii b v - \cal J \wh S_W(q)}^2}
\notag \\ \label{S_u_int}
&\;=\; \sum_{q \in W \bb T^*} \frac{\ind{\abs{q} \leq \epsilon}}{\pb{1 + \ii b v - \cal J (\cal I - q \cdot D q)}^2}  + O\pbb{\frac{N}{M} R_2(\eta M^{-c_2})}
\end{align}
for some $\epsilon$ which we fix in the following.
Plugging \eqref{S_u_int} into \eqref{FT_identity} and changing variables $v \to v/b$ yields
\begin{equation} \label{intr_T}
\frac{1}{2 \pi} \int_\R \dd v \, e^\eta(v) \tr \frac{S}{\p{1 + \ii b v -\cal JS}^2} \;=\; U + O \pbb{\frac{N}{M} R_2(\eta M^{-c_2})}\,,
\end{equation}
where we defined
\begin{align}
U &\;\deq\; \frac{1}{b^2} \sum_{q \in W \bb T^*} \ind{\abs{q} \leq \epsilon} \int_0^\infty \dd t \, \exp \pbb{-\frac{1 - \cal J (\cal I - q \cdot D q)}{b} \, t} \, t \, \ol{\wh e(\eta t)}
\notag \\ \label{def_U}
&\;=\;
\frac{1}{(b \eta)^2} \int_0^\infty \dd s \, s \, \ol{\wh e(s)} \, \ee^{- (1 - \cal J \cal I) (b\eta)^{-1} s} \sum_{q \in W \bb T^*} \ind{\abs{q} \leq \epsilon}  \ee^{- (b \eta)^{-1} s q \cdot D q}\,.
\end{align}
Introducing the variable $r \deq s^{1/2} (b\eta)^{-1/2} q$, we find by Riemann sum approximation
\begin{align}
&\mspace{-30mu}\sum_{q \in W \bb T^*} \ind{\abs{q} \leq \epsilon}  \ee^{-(b \eta)^{-1} s q \cdot D q}
\\
&\;=\; 
\sum_{r \in s^{1/2} (b\eta)^{-1/2} W \bb T^*} \ind{\abs{r} \leq \epsilon s^{1/2} (b \eta)^{-1/2}} \, \ee^{- r \cdot D r}
\notag \\
&\;=\; \pbb{\frac{(b \eta)^{1/2} L}{2 \pi s^{1/2} W}}^d \qBB{\int_{\R^d} \dd r \, \ind{\abs{r} \leq \epsilon s^{1/2} (b\eta)^{-1/2}} \, \ee^{-r \cdot D r} + O \pbb{\frac{s^{1/2} W}{\eta^{1/2} L}}}
\\ \label{G1}
&\;=\; \pbb{\frac{(b \eta)^{1/2} L}{2 \pi s^{1/2} W}}^d \qBB{\frac{\pi^{d/2}}{\sqrt{\det D}} + O \pbb{\frac{s^{1/2} W}{\eta^{1/2} L} + \ee^{-c \eta^{-1} s}}}\,,
\end{align}
where $c$ is some positive constant depending on $\epsilon$ and $b$. This estimate will be used where $s\gg \eta$.
On the other hand, we have the trivial bound
\begin{equation} \label{G2}
\sum_{q \in W \bb T^*} \ind{\abs{q} \leq \epsilon}  \ee^{- (b \eta)^{-1} s q \cdot D q} \;\leq\; \frac{CN}{M}\,,
\end{equation}
which will be used for small $s$.

Next, let $\delta$ be an exponent satisfying
\begin{equation*}
0 \;<\; 2 \delta \;<\; \min \h{1/3 - \rho\,,\, c_2\,,\, \rho}\,.
\end{equation*}
We split the integration over $s \in [0,\infty)$ in \eqref{def_U} into the interval $[\eta, M^\delta]$ and its complement.
Using \eqref{G2}  and the rapid decay of $\wh e(s)$ for large $s$ 
 to estimate the integrand for $s \notin [\eta, M^\delta]$, together with the bound
 $\ee^{- (1 - \cal J \cal I) (b\eta)^{-1} s} = 1 + O(M^{-\delta})$ for $s \leq M^\delta$, we therefore get 
\begin{align}
U &\;=\; \frac{1}{(b \eta)^2} \int_{\eta}^{M^\delta} \dd s \, s \, \ol{\wh e(s)} \, \ee^{- (1 - \cal J \cal I) (b\eta)^{-1} s} \pbb{\frac{(b \eta)^{1/2} L}{2 \pi s^{1/2} W}}^d \qBB{\frac{\pi^{d/2}}{\sqrt{\det D}} + O \pbb{\frac{s^{1/2} W}{\eta^{1/2} L} + \ee^{-c \eta^{-1} s}}} + O \pbb{\frac{N}{M}}
\notag \\ \label{T_computed}
&\;=\; \frac{(b \eta)^{d/2 - 2}}{\sqrt{\det D}} \pbb{\frac{L}{2 \sqrt{\pi} W}}^d \int_{\eta}^{M^\delta} \dd s \, s^{1-d/2} \, \ol{\wh e(s)} \, \qbb{1 + O \pbb{\frac{s^{1/2} W}{\eta^{1/2} L} + \ee^{-c \eta^{-1} s}}} + O \pbb{\frac{N}{M}}\,.
\end{align}
At this point we distinguish the cases $d \leq 3$ and $d = 4$. Let us start with $d \leq 3$. Observing that by \eqref{LW_assump} we have $\frac{s^{1/2} W}{\eta^{1/2} L} \leq M^{-\delta}$ for $s \leq M^\delta$, we find
\begin{equation*}
U \;=\; \frac{(b \eta)^{d/2 - 2}}{\sqrt{\det D}} \pbb{\frac{L}{2 \sqrt{\pi} W}}^d \int_0^\infty \dd s \, s^{1-d/2} \, \ol{\wh e(s)} + O \pbb{\frac{N}{M} \pb{1 + R_4(\eta) M^{-\delta}}}\,.
\end{equation*}
Now \eqref{intR3} follows from \eqref{intr_T} and the fact that $R_2(\eta M^{-c_2}) \leq C R_4(\eta) M^{-\delta}$.

Finally, let $d = 4$. From \eqref{T_computed} we get
\begin{align*}
U &\;=\; \frac{1}{\sqrt{\det D}} \pbb{\frac{L}{2 \sqrt{\pi} W}}^4 \int_{\eta}^1 \dd s \, s^{-1} \, \ol{\wh e(s)} + O \pbb{\frac{N}{M}}
\\
&\;=\; \frac{\abs{\log \eta}}{\sqrt{\det D}} \pbb{\frac{L}{2 \sqrt{\pi} W}}^4 \, \ol{\wh e(0)}  + O \pbb{\frac{N}{M}}\,.
\end{align*}
Now \eqref{intR4} follows from \eqref{intr_T}.
\end{proof}

\section{Proofs of Propositions \ref{prop: bounds on T} and \ref{prop:T_int}} \label{appendix: T}

The arguments of this appendix are similar to those of Appendix \ref{appendix: S}. We take over the notations from Appendix \ref{appendix: S} without further comment, and only give the proofs when they differ significantly from those of Appendix \ref{appendix: S}.

For $q \in [-\pi W, \pi W)^d$ we define (in analogy to $\wh S_W(q)$ from \eqref{def_S_W})
\begin{align}
\wh T_W(q) &\;\deq\;
\frac{1}{M - 1} \sum_{r \in W^{-1} \bb T} \ee^{-\ii q \cdot r} \, f(r) \, [1 - \varphi h(r)]\, \ee^{\ii \lambda g(r)}
\notag \\ \label{TW}
&\;=\; \sum_{x \in \bb T}
\cos (q \cdot x/W - \lambda g(x/W)) \, [1 - \varphi h(x/W)] \, S_{x0}\,.
\end{align}
We remark that for all practical purposes $\wh T_W(q)$ should be thought of as its limit as $W \to \infty$, i.e.\ the integral
\begin{equation} \label{limW_B}
\lim_{W \to \infty} \wh T_W(q) \;=\; \int_{\R^d} \dd r \, \cos (q \cdot r - \lambda g(r)) \, [1 - \varphi h(r)] \, f(r)\,.
\end{equation}

The following result generalizes Lemma \ref{lem: basic properties of SW} to the case $\varphi, \lambda \neq 0$.
\begin{lemma} \label{lem:That}
\begin{enumerate}
\item
For each fixed $\epsilon > 0$ there exists a $\delta_\epsilon > 0$ such that if $\max\{\lambda, \varphi, \abs{q}\} \geq \epsilon$ then $\abs{\wh T_W(q)} \leq 1 - \delta_\epsilon$ for large enough $W$ (depending on $\epsilon$).
\item
We have the expansion
\begin{equation} \label{gen_exp_phi}
\wh T_W(q) \;=\; \cal I - \wt \sigma - (q - \lambda D^{-1} w) \cdot D (q - \lambda D^{-1} w) + \cal Q(q)
+ O\pb{\abs{q}^{4 + c} + \abs{q}^3 \lambda + \abs{q}^2 \varphi + \varphi^2 + \lambda^4}\,,
\end{equation}
where $c$ is the constant from \eqref{moment of f}.
\end{enumerate}
\end{lemma}

\begin{proof}
For $\epsilon \in (0,1)$ and $K > 0$ define the compact domain
\begin{equation*}
\cal D_K \;\deq\; \hb{(\lambda, \varphi, q) \in [0,1] \times [0,1] \times \R^d \col \abs{q} \leq K \,,\, \max\{\lambda, \varphi, \abs{q}\} \geq \epsilon}\,.
\end{equation*}
Then for each fixed triple $(\lambda, \varphi, q) \in \cal D_K$ it is not hard to see from \eqref{limW_B} that $\lim_{W \to \infty} \abs{\wh T_W(q)} < 1$, by assumption on $g$ and $h$. Since $\cal D_K$ is compact and the map $(\lambda, \varphi, q) \mapsto \lim_{W \to \infty} \abs{\wh T_W(q)}$ is continuous, we conclude that there exists a $\delta_\epsilon > 0$ such that $\lim_{W \to \infty}\abs{\wh T_W(q)} \leq 1 - \delta_\epsilon$ for all $(\lambda, \varphi, q) \in \cal D_K$. Since the convergence of $T_W$ is uniform in $(\lambda, \varphi, q) \in \cal D_K$, we conclude (after renaming $\delta_\epsilon$) that $\abs{\wh T_W(q)} \leq 1 - \delta_\epsilon$ for $(\lambda, \varphi, q) \in\cal D_K$ and large enough $W$.

What remains in the proof of (i) is to prove $\abs{\wh T_W(q)} \leq 1 - \delta_\epsilon$ in the case $\abs{q} > K$ and $q \in [-\pi W, \pi W)^d$ for some large enough $K > 0$. To that end, we use summation by parts. Let $\abs{q} \geq K$. Without loss of generality, suppose that $\abs{q_1} \geq \abs{q_i}$ for all $i = 1, \dots, d$, so that $\abs{q_1} \geq K d^{-1/2}$. To simplify notation, we assume that $f$, $g$, and $h$ are $C^1$ and not just piecewise $C^1$. The piecewise $C^1$ case may be handled similarly by restriction to individual pieces combined with a simple estimate of the boundary terms arising from the summation by parts. Define the discrete derivative $(D_i f)(r) \deq W \pb{f(r + W^{-1} e_i) - f(r)}$ where $e_i$ is the standard unit vector in the $i$-direction. Then we have
\begin{equation*}
\wh T_W(q) \;=\; \frac{\ii \ee^{\ii q_1 / (2W)}}{2 W \sin(q_1 / (2W))} \, \frac{1}{M - 1} \sum_{r \in W^{-1} \bb T} D_1(\ee^{-\ii q \cdot r}) \, f(r) \, [1 - \varphi h(r)]\, \ee^{\ii \lambda g(r)}\,.
\end{equation*}
Since $\abs{q_1} \leq \pi W$, the first fraction is bounded in absolute value by $C / \abs{q_1} \leq C / \abs{q}$. The rest may be estimated using summation by parts by a constant, using the assumptions on $f$, $g$, and $h$; we omit the details. The result is $\abs{\wh T_W(q)} \leq C / \abs{q}$, where the constant $C$ does not depend on $\lambda$ or $\varphi$. Choosing $K$ large enough that $C / K \leq 1/2$ completes the proof of part (i).

In order to prove (ii), we expand $\cos$ in \eqref{TW} to fourth order and use \eqref{moment of f} to get
\begin{equation*}
\wh T_W(q)
\;=\; \cal I - A_W(q) - \varphi \sum_{x \in \bb T} h(x/W) S_{x0} + \cal Q(q)
+ O\pb{\abs{q}^{4 + c} + \abs{q}^3 \lambda + \abs{q}^2 \varphi + \varphi \lambda^2 + \lambda^4}\,,
\end{equation*}
where we introduced the quadratic term
\begin{equation} \label{def_AW}
A_W(q) \;\deq\; \frac{1}{2} \sum_{x \in \bb T} \pb{q \cdot x/W - \lambda g(x/W)}^2 S_{x0}
\;=\; (q - \lambda D^{-1} w) \cdot D (q - \lambda D^{-1} w) + \lambda^2 \Delta\,.
\end{equation}
This concludes the proof of (ii).
\end{proof}

\begin{proof}[Proof of Proposition \ref{prop: bounds on T}]
The proof of Proposition \ref{prop: bounds on S} may be taken over with minor modifications, using Lemma \ref{lem:That} as input. The bound \eqref{bound res 1} for $T$ is proved similarly to \eqref{bound res 1} for $S$. All of the remaining claims rely on Fourier space analysis. We use the expansion \eqref{gen_exp_phi} instead of \eqref{SW exp}. In the summation over $q$ in \eqref{resolvent_fspace} we shift the origin by introducing the new variable $\wt q \deq q - \lambda D^{-1} w$. Using it we may write \eqref{gen_exp_phi} as
\begin{equation} \label{shifted_F_space}
\wh T_W(q) \;=\; \cal I - \wt \sigma - \wt q \cdot D \wt q + \cal Q(\wt q)
+ O\pb{\abs{\wt q}^{4 + c} + \abs{\wt q}^3 \lambda + \abs{\wt q}^2 \varphi + \varphi^2 + \lambda^4}\,.
\end{equation}
Then all the Fourier space arguments from the proof of Proposition \ref{prop: bounds on S} carry over provided one replaces $1 - \alpha$ with $1 - \alpha + \wt \sigma$.
\end{proof}

\begin{proof}[Proof of Proposition \ref{prop:T_int}]
The proof follows along the same lines as that of Proposition \ref{prop:S_int}, using \eqref{gen_exp_phi} instead of \eqref{SW exp}. We use the same shift in Fourier space as in the proof of Proposition \ref{prop: bounds on T}, and work with \eqref{shifted_F_space}. The only change to the proof of Proposition \ref{prop:S_int} is that in \eqref{T_computed} we get the additional factor $\ee^{- s (\lambda^2 \Delta + \varphi \Upsilon)/(\eta b)}$. For the case $d = 4$ we use the elementary estimate
\begin{equation*}
\int_\eta^1 \frac{1}{s} \, \ee^{- s \wt \sigma/(\eta b)} \, \dd s \;=\; \min\h{\abs{\log \eta}, \abs{\log \wt \sigma}} + O(1)\,. \qedhere
\end{equation*}
\end{proof}

{\small
\providecommand{\bysame}{\leavevmode\hbox to3em{\hrulefill}\thinspace}
\providecommand{\MR}{\relax\ifhmode\unskip\space\fi MR }
% \MRhref is called by the amsart/book/proc definition of \MR.
\providecommand{\MRhref}[2]{%
  \href{http://www.ams.org/mathscinet-getitem?mr=#1}{#2}
}
\providecommand{\href}[2]{#2}

}

\end{document}